\documentclass[11pt]{article}
\usepackage{amssymb,amsthm,amsmath}
\usepackage[colorlinks,linkcolor=black,anchorcolor=blue,citecolor=green]{hyperref}
\usepackage[numbers,sort&compress]{natbib}
\usepackage{amssymb,amsmath}
\usepackage{amsfonts}
\usepackage{mathrsfs}
\usepackage{latexsym}
\usepackage{amssymb}
\usepackage{amsthm}
\usepackage{indentfirst}
\date{Sep. 20, 2014}

\numberwithin{equation}{section}

\newtheorem{theorem}{Theorem}[section]

\newtheorem{lemma}[theorem]{Lemma}

\newtheorem{remark}[theorem]{Remark}

\begin{document}

\def\del{\partial}
\def\DOT{\!\cdot\!}
\def\dt{{\Delta t}}
\def\dx{{\Delta x}}
\def\dy{{\Delta y}}
\def\dz{{\Delta z}}
\def\div{\text{div}}
\def\u{{\mbox{\boldmath $u$}}}
\def\g{{\mbox{\boldmath $g$}}}
\def\bu{{\mbox{\boldmath $u$}}}
\def\x{{\mbox{\boldmath $x$}}}
\def\n{{\mbox{\boldmath $n$}}}
\def\bn{{\mbox{\boldmath $n$}}}
\def\veceps{\mbox{\boldmath $\epsilon$}}
\def\vecdel{\mbox{\boldmath $\delta$}}
\def\vecphi{\mbox{\boldmath $\phi$}}
\def\vecpsi{\mbox{\boldmath $\psi$}}
\def\ds{\mbox{d \boldmath $\gamma$}}
\def\vecomi{\mbox{\boldmath $\theta$}}
\def\eps{\varepsilon}
\def\oneb{{1\mbox \tiny b}}
\def\twob{{2\mbox \tiny b}}
\def\disp{\displaystyle}
\def\dspace{\displaystyle \vspace{.15in}}
\def\half{\textstyle \frac12}
\def\Circ{{\footnotesize$\bigcirc$}}
\def\Triangle{{\small$\triangle$}}
\def\Bullet{{\large$\bullet$}}
\def\exac{{\mbox \tiny e}}
\def\dif{\mathrm{d}}
\def\di{\mathrm{div}}
\def\C{\mathrm{curl}}
\newcommand{\Dtil}{\widetilde{D}}

\def\TIME{\!\times\!}

\title{Global Classical Solutions to the 2D Compressible MHD Equations with Large Data and Vacuum}

\author {Yu \, Mei\footnote{E-mail addresses: ymei@math.cuhk.edu.hk}}
\date{}
\maketitle

 \centerline{\em The Institute of Mathematical Sciences
and Department of Mathematics}

\centerline{\em The Chinese University of Hong Kong}
\bigskip

{\bf Abstract:} In this paper, we study the global well-posedness of classical solutions to the 2D compressible MHD equations with large initial data and vacuum. With the assumption $\mu=const.$ and $\lambda=\rho^\beta,~\beta>1$ (Va\v{i}gant-Kazhikhov Model) for the viscosity coefficients, the global existence and uniqueness of classical solutions to the initial value problem is established on the torus $\mathbb{T}^2$ and the whole space $\mathbb{R}^2$ (with vacuum or non-vacuum far fields). These results generalize the previous ones for the Va\v{i}gant-Kazhikhov model of compressible Navier-Stokes.  


\section{Introduction and Main Results}

\noindent\;\;\;\; In this paper, we study the following compressible isentropic magnetohydrodynamic equations which describe the motion of conducting fluids in an electromagnetic field in a domain $\Omega$ of $\mathbb{R}^2$
\begin{equation}\label{MHD}
\left\{\begin{array}{lr} \rho_t +\di(\rho u)=0, ~~~~~~~~~~~~~~~~~~\\
(\rho u)_t+\di(\rho u\otimes u)+\nabla P(\rho)=H\cdot \nabla H-\frac{1}{2}\nabla|H|^2+\mathcal{L}_\rho u,~~~~~~~ x\in \Omega,~~ t>0\\
H_t+u\cdot\nabla H-H\cdot\nabla u+H\di u-\nu\Delta H=0,~~~~~~~~~\div H=0,\\
\end{array}\right.
\end{equation}
where $\rho(t,x)\geq 0,u(t,x)=(u_1,u_2)(t,x),H(t,x)=(H_1,H_2)(t,x)$ represent the density, the velocity and the magnetic field respectively, and the pressure $P$ is given by 
\begin{equation}\label{Pre}
P(\rho)=A\rho^\gamma,~~~~~~~~\gamma>1.
\end{equation}
In the sequel, we set $A=1$ without loss of generality. The operator $\mathcal{L}_\rho$ is defined by
\begin{equation}\label{Lame}
\mathcal{L}_\rho u=\mu\Delta u+\nabla((\mu+\lambda(\rho))\di u)
\end{equation}
where the shear and bulk viscosity coefficients $\mu,\lambda$ are assumed to satisfy
\begin{equation}\label{VKC}
\mu=const.>0,~~~~~~~~\lambda(\rho)=\rho^\beta.
\end{equation} 
such that $\mathcal{L}_\rho$ is strictly elliptic.
Moreover, the magnetic diffusive coefficient $\nu>0$ is a given positive constant.

There are extensive studies on the well-posedness theory of MHD system when the viscosity coefficients are both constant. The local existence and uniqueness of strong solutions were obtained by Vol'pert-Hudjaev \cite{vol} as the initial density is strictly positive and by Fan-Yu \cite{fy} as the initial density may contain vacuum. Kawashima \cite{Kaw}  first obtained the global existence and uniqueness of classical solutions when the initial data are close to a non-vacuum equilibrium in $H^3$-norm. For general large initial data, Hu-Wang \cite{hw10,hw08} proved the global existence of weak solutions with finite energy in the Lion's framework for compressible Navier-Stokes equations\cite{lions}, provided the adiabatic exponent $\gamma$ suitable large. Recently, Li-Xu-Zhang\cite{lxz} established the global existence and uniqueness of classical solutions to the Cauchy problem for the isentropic compressible MHD system in 3D with smooth initial data which are of small energy but possibly large oscillations and vacuum. This generalize the result for compressible Navier-Stokes obtained by Huang-Li-Xin\cite{hlx}. 

The density-dependent viscosity models has received a lot attention recently in particular for the compressible Navier-Stokes equations where the effect of magnetic field was ignored in the MHD system, refer to \cite{Bre03A, Bre03B, Bre06, gjx, jiang, jx, lxy, Mel07, Mel08, yang01, yang02A, yang02B} and the references therein. Such models can be derived from the Boltzmann equations by the Chapman-Enskog expansions where the viscosity coefficients depend on the temperature so that depending on the density for isentropic flows\cite{lxy}. Moreover, the viscous Saint-Vanant system for shallow water model can be expressed exactly in the form of density-dependent viscosity\cite{gp}. However, the progress for multi-dimensional problems of density-dependent viscosity model is very limited since the high degeneracy of equations except for the existed difficulties for compressible Navier-Stokes. The global existence of weak solutions to the compressible Navier-Stokes equations with density-dependent viscosities also remains open, except for symmetric data case\cite{gjx} and the $L^1$ stability results\cite{Mel07}. It is remarkable that Kazhikhov-Va\v{i}gant\cite{vk} first proposed and studied the global well-posedness of strong solutions uniformly away from vacuum for the two-dimensional periodic problem for compressible Navier-Stokes equation with the special viscosity coefficients, that is, $\mu$ being a positive constant and $\lambda=\rho^\beta,\beta>3$. Moreover this is the first ever global strong solutions with no restrictions on the size of initial data in multidimensional space. Recently, there are some generalized results for this Va\v{i}gant-Kazhikhov model of viscosity coefficients with initial data permitting vacuum, for example, Jiu-Wang-Xin\cite{jwx1} and Huang-Li\cite{hl1} proved the global existence and uniqueness of classical solutions to the periodic problem of this model for $\beta>3$ and $\beta>\frac{4}{3}$ respectively. Huang-Li\cite{hl2} and Jiu-Wang-Xin\cite{jwx3} established the global well-posedness theory to Cauchy problem of this model for $\beta>\frac{4}{3}$ with vacuum and non-vacuum at fat fields respectively. More recently, Huang-Wang\cite{hwc} took full advantage of the Brezis-Wainger inequality to generalize the results to $\beta>1$ which seems to be the extremal case for the global well-posedness of the Va\v{i}gant-Kazhikhov model of viscosity coefficients for the compressible Navier-Stokes equations.

In this paper, we consider the initial value problem (IVP) of Va\v{i}gant-Kazhikhov model of viscosity coefficients for the MHD system $(\ref{MHD})$ in the domain $\Omega$ either be the periodic one $\mathbb{T}^2$ or the whole space $\mathbb{R}^2$ with supplement initial data
\begin{equation}\label{IVC1}
(\rho,u,H)(x,t=0)=(\rho_0,u_0,H_0),\quad x\in\Omega
\end{equation}
Furthermore, we also impose the following far fields when $\Omega=\mathbb{R}^2$:
\begin{equation}\label{IVC2}
(\rho_0,u_0,H_0)\rightarrow(\tilde{\rho}\geq0,0,0), \quad as \quad |x|\rightarrow \infty.
\end{equation}
The main results in this paper can be stated as follows. The first result concerning the period problem reads as
\begin{theorem}\label{thm1}
If $\beta>1$ and the initial data $(\rho_0,u_0,H_0)$ satisfy
\begin{equation}
0\leq\rho_0\in W^{2,q}(\mathbb{T}^2)\cap L^1(\mathbb{T}^2),~~(u_0,H_0)\in H^2(\mathbb{T}^2)\times H^2(\mathbb{T}^2)
\end{equation}
for some $q>2$ and the compatibility condition
\begin{equation}\label{com}
\mu\Delta u_0+\nabla((\mu+\lambda(\rho_0))\di u_0)-\nabla P(\rho_0)+(\nabla\times H_0)\times H_0=\sqrt{\rho_0}g
\end{equation}
with some $g\in L^2(\mathbb{T}^2)$. Then there exists a unique global classical solution $(\rho,u,H)$ to the compressible MHD equation $(\ref{MHD})$ and $(\ref{IVC1})$ satisfying
\begin{align}\label{reg1}
& 0\leq\rho(t,x)\leq C,\forall(t,x)\in[0,T]\times\mathbb T^2,(\rho,P(\rho))\in C([0,T];W^{2,q}(\mathbb{T}^2)),\nonumber\\
&\rho_t\in C([0,T];L^q(\mathbb{T}^2)),(u,H)\in C([0,T];H^2(\mathbb{T}^2))\cap L^2(0,T;H^3(\mathbb{T}^2)),\nonumber\\
&\sqrt{t}(u,H)\in L^\infty(0,T;H^3(\mathbb T^2)),t(u,H)\in L^\infty(0,T;W^{3,q}(\mathbb T^2)),\\
& (u_t,H_t)\in L^2(0,T;H^1(\mathbb{T}^2)),\sqrt{t}(u_t,H_t)\in L^2(0,T;H^2(\mathbb T^2))\cap L^\infty(0,T;H^1(\mathbb T^2)),\nonumber\\
& t(u_t,H_t)\in L^\infty(0,T;H^2(\mathbb T^2)),\sqrt{t}(\sqrt{\rho}u_{tt},H_{tt})\in L^2(0,T;L^2(\mathbb T^2)),\nonumber\\
& t(\sqrt{\rho}u_{tt},H_{tt})\in L^\infty(0,T;L^2(\mathbb T^2)),t(\nabla u_{tt},\nabla H_{tt})\in L^2(0,T;L^2(\mathbb T^2))\nonumber
\end{align}
for any $0<T<\infty$.
\end{theorem}

\begin{remark}
If the initial data contains vacuum, then it is natural to impose the compatibility condition (\ref{com}) as the case of constant viscosity coefficients in \cite{chok}. This is also natural for the whole space case in Theorem \ref{thm2} and Theorem \ref{thm3}..
\end{remark}

\begin{remark}
It should be mentioned here that it seems that $\beta>1$ obtained here is the extremal case for the MHD system (\ref{MHD}) (c.f. Lemma \ref{BP}) for global theory of classical solutions. The same for the whole space case in Theorem \ref{thm2} and Theorem \ref{thm3}.
\end{remark}
Concerning the Cauchy problem in $\mathbb{R}^2$ with the initial data of vacuum far fields ($\tilde{\rho}=0$), we can obtain that
\begin{theorem}\label{thm2}
If $\beta>1$ and the initial data $(\rho_0,u_0,H_0)$ satisfy
\begin{align}
& 0\leq(\rho_0,P(\rho_0))\in W^{2,q}(\mathbb{R}^2)\times W^{2,q}(\mathbb{R}^2),~~(u_0,H_0)\in H^2(\mathbb{R}^2)\times H^2(\mathbb{R}^2),\nonumber\\
&\rho_0(1+|x|^{\alpha_1})\in L^1(\mathbb R^2),\sqrt{\rho_0}u_0(1+|x|^\frac{\alpha}{2})\in L^2(\mathbb R^2),\nabla u_0|x|^\frac{\alpha}{2}\in L^2(\mathbb R^2),
\end{align}
for some $q>2$ and the weights $0<\alpha<2\sqrt{\sqrt{2}-1}$, $\alpha<\alpha_1$ and the compatibility condition
\begin{equation}\label{3.1.6}
\mu\Delta u_0+\nabla((\mu+\lambda(\rho_0))\di u_0)-\nabla P(\rho_0)+(\nabla\times H_0)\times H_0=\sqrt{\rho_0}g
\end{equation}
with some $g(1+|x|^\frac{\alpha}{2})\in L^2(\mathbb{R}^2)$. Then there exists a unique global classical solution $(\rho,u,H)$ to the compressible MHD equation $(\ref{MHD})$ and $(\ref{IVC1})$-$(\ref{IVC2})$ satisfying
\begin{align}\label{reg2}
& 0\leq\rho(t,x)\leq C,\forall(t,x)\in[0,T]\times\mathbb R^2,(\rho,P(\rho))\in C([0,T];W^{2,q}(\mathbb{R}^2)),\nonumber\\
&\rho(1+|x|^{\alpha_1})\in C([0,T];L^1(\mathbb R^2)),\sqrt{\rho}u(1+|x|^\frac{\alpha}{2}),\sqrt{\rho}\dot{u}(1+|x|^\frac{\alpha}{2})\in C([0,T];L^2(\mathbb R^2)),\nonumber\\
&\nabla u|x|^\frac{\alpha}{2}\in C([0,T];L^2(\mathbb R^2)),(u,H)\in C([0,T];H^2(\mathbb{R}^2))\cap L^2(0,T;H^3(\mathbb{R}^2)),\nonumber\\
&\sqrt{t}(u,H)\in L^\infty(0,T;H^3(\mathbb R^2)),t(u,H)\in L^\infty(0,T;W^{3,q}(\mathbb R^2)),\\
&(u_t,H_t)\in L^2(0,T;H^1(\mathbb R^2)),\sqrt{t}(u_t,H_t)\in L^2(0,T;H^2(\mathbb R^2))\cap L^\infty(0,T;H^1(\mathbb R^2)),\nonumber\\
& t(u_t,H_t)\in L^\infty(0,T;H^2(\mathbb R^2)),\sqrt{t}(\sqrt{\rho}u_{tt},H_{tt})\in L^2(0,T;L^2(\mathbb R^2)),\nonumber\\
& t(\sqrt{\rho}u_{tt},H_{tt})\in L^\infty(0,T;L^2(\mathbb R^2)),t(\nabla u_{tt},\nabla H_{tt})\in L^2(0,T;L^2(\mathbb R^2))\nonumber
\end{align}
for any $0<T<\infty$.
\end{theorem}

\begin{remark}
The weight imposing on the velocity field was indicated by Jiu-Wang-Xin \cite{jwx2} which is used to overcome the failure of Poincar\'{e}-type inequality for unbounded domain. However, we also can obtain the spatial weighted estimate of the density provide that it initially has one. This is motivated by the decay of density for large value of the spatial variable $x$. We can refer to Huang-Li \cite{hl2} for more details
\end{remark}

If the initial data of the density has non-vacuum far fields, i.e. $\tilde{\rho}>0$, then the following results can be obtained
\begin{theorem}\label{thm3}
If $\beta>1,1<\gamma\leq 2\beta,$ and the initial data $(\rho_0,u_0,H_0)$ satisfy
\begin{align}
& 0\leq(\rho_0-\tilde{\rho},P(\rho_0)-P(\tilde{\rho}))\in W^{2,q}(\mathbb{R}^2)\times W^{2,q}(\mathbb{R}^2),~~(u_0,H_0)\in H^2(\mathbb{R}^2)\times H^2(\mathbb{R}^2),\nonumber\\
&\Psi(\rho_0,\tilde{\rho})(1+|x|^{\alpha})\in L^1(\mathbb R^2),\sqrt{\rho_0}u_0(1+|x|^\frac{\alpha}{2})\in L^2(\mathbb R^2),
\end{align}
for some $q>2$ and the weights $0<\alpha^2<\frac{4(\sqrt{2+\frac{\lambda(\tilde{\rho})}{\mu}}-1)}{1+\frac{\lambda(\tilde{\rho})}{\mu}}$, and the compatibility condition
\begin{equation}\label{3.1.9}
\mu\Delta u_0+\nabla((\mu+\lambda(\rho_0))\di u_0)-\nabla P(\rho_0)+(\nabla\times H_0)\times H_0=\sqrt{\rho_0}g
\end{equation}
with some $g\in L^2(\mathbb{R}^2)$. Then there exists a unique global classical solution $(\rho,u,H)$ to the compressible MHD equation $(\ref{MHD})$ and $(\ref{IVC2})$ satisfying
\begin{align}\label{reg3}
& 0\leq\rho(t,x)\leq C,\forall(t,x)\in[0,T]\times\mathbb R^2,(\rho-\tilde{\rho},P(\rho)-P(\tilde{\rho}))\in C([0,T];W^{2,q}(\mathbb{R}^2)),\nonumber\\
&\Psi(\rho,\tilde{\rho})(1+|x|^{\alpha})\in C([0,T];L^1(\mathbb R^2)),\sqrt{\rho}u(1+|x|^\frac{\alpha}{2})\in C([0,T];L^2(\mathbb R^2)),\nonumber\\
&(u,H)\in C([0,T];H^2(\mathbb{R}^2))\cap L^2(0,T;H^3(\mathbb{R}^2)),\nonumber\\
&\sqrt{t}(u,H)\in L^\infty(0,T;H^3(\mathbb R^2)),t(u,H)\in L^\infty(0,T;W^{3,q}(\mathbb R^2)),\\
&(u_t,H_t)\in L^2(0,T;H^1(\mathbb R^2)),\sqrt{t}(u_t,H_t)\in L^2(0,T;H^2(\mathbb R^2))\cap L^\infty(0,T;H^1(\mathbb R^2)),\nonumber\\
& t(u_t,H_t)\in L^\infty(0,T;H^2(\mathbb R^2)),\sqrt{t}(\sqrt{\rho}u_{tt},H_{tt})\in L^2(0,T;L^2(\mathbb R^2)),\nonumber\\
& t(\sqrt{\rho}u_{tt},H_{tt})\in L^\infty(0,T;L^2(\mathbb R^2)),t(\nabla u_{tt},\nabla H_{tt})\in L^2(0,T;L^2(\mathbb R^2))\nonumber
\end{align}
for any $0<T<\infty$.
\end{theorem}

\begin{remark}
As remarked by Jiu-Wang-Xin in \cite{jwx3}, if $\lambda(\tilde{\rho})<7\mu$, one has $\frac{4(\sqrt{2+\frac{\lambda(\tilde{\rho})}{\mu}}-1)}{1+\frac{\lambda(\tilde{\rho})}{\mu}}$. Then we can choose a weight $\alpha$ satisfying $1<\alpha^2<\frac{4(\sqrt{2+\frac{\lambda(\tilde{\rho})}{\mu}}-1)}{1+\frac{\lambda(\tilde{\rho})}{\mu}}$. In this case, the condition $\gamma\leq 2\beta$ in Theorem \ref{thm3} can be removed.
\end{remark}
\begin{remark}
If $\tilde{\rho}=0$, then $\frac{4(\sqrt{2+\frac{\lambda(\tilde{\rho})}{\mu}}-1)}{1+\frac{\lambda(\tilde{\rho})}{\mu}}=4(\sqrt{2}-1)$. This is exactly same as the result in Theorem \ref{thm2}.
\end{remark}
We now comment on the analysis of this paper. Since the local well-posedness of classical solutions to the MHD system can be achieved as in \cite{fy,st,vol} with given smooth initial data. One can only need to derive the a priori estimate to extend the local solution to global in time one. The key issue of the a priori estimate is to obtain the uniform upper bound of the density. Since the magnetic field is strongly coupled with the velocity field of the fluid in the compressible MHD system, some new difficulties arise in comparison with the problem for the compressible Navier-Stokes equations studied in \cite{vk,jwx1,jwx2,hl1,hl2}. The following key observations help us to deal with the interaction of the magnet field and the velocity field very well. First, we give the $L{_t^\infty}L{_x^p}$ $(p\geq 2)$ estimate of the magnetic field $H$ based on the elementary energy estimates for the equations (\ref{MHD}) and the standard $L^p$ estimates of the parabolic equation $(\ref{MHD})_3$. With this observation at hand, we can follow the ideas of elliptic estimates and standard $L^p$ estimate for transport equation, which are developed in \cite{vk,jwx1,jwx2}, to derive the crucial $L{_t^\infty}L{_x^p}$ $(p\geq 2)$ of the density. To derive the upper bound of the density, we take full advantage of the ideas developed for compressible Navier-Stokes equations in Huang-Li\cite{hl1,hl2}. That is, first, we show that $\log(1+\|\nabla u\|_2+\|\nabla H\|_2)$ is bounded by a power function of $\|\rho\|_\infty$ by energy type estimates (in particular estimate of effective viscous flux) and the compensated compactness analysis; then rewriting the momentum equation $(\ref{MHD})_2$ as the form in terms of a sum of commutators of Riesz transforms and the operators of multiplication by $u_i,H_i$ and using the $W^{1,p}$-estimate of the commutator due to Coifman-Meyer \cite{coif1} and the estimate of momentum $\|\rho u\|_p$, we obtain an estimate on the $L^1(0,T;L^\infty)$ of the commutators in terms of $\|\rho\|_\infty$; finally the Brezis-Wainger inequality gives the key upper bound of the density provided $\beta>1$. However, for the compressible MHD system, the new term $H\cdot \nabla H\cdot \dot{u}$ comes out in the process of the estimate for $\log(1+||\nabla u||_2+||\nabla H||_2)$ in terms of $||\rho||_\infty$ which is crucial in \cite{hl1,hl2}. We observe that this terms can be calculated by integration through parts and reduced into the $L{_t^2}L{_x^2}$ estimate of $\nabla^2H$. Moreover, the $L{_t^2}L{_x^2}$ estimate of $\nabla^2H$ can be derived by standard energy estimate through multiplying the magnet equation $(\ref{MHD})_3$ by $\Delta H$. The rest of this thesis is organized as follows. We first recall some preliminary lemmas in the next section. Section 3 and 4 concerning the a priori estimates is the main parts of this paper. In section 3, we give the complete lower and higher order estimates for the (IVP) in $\mathbb{T}^2$. However, we skip fully detailed estimates for simplicity and just show the key weighted estimate for the (IVP) in $\mathbb{R}^2$ to overcome the failure of Poincar\'{e}-type inequality. We can refer to \cite{jwx2,jwx3} and Section 3 for more details. The proof of main results is proved in the final section.\\

{\bf{Notations:}} Throughout this paper, positive generic constants are denoted by $C$, which may change line by line. The small constant to be chosen is denoted by $\varepsilon$, $\sigma$ and $\delta$. For functional spaces, $L^p(\Omega)$, $1\leq p\leq \infty$, denotes the usual Lebesgue spaces on $\Omega$ with the $L^p$ norm denoted by $\|\|_p$. $W^{k,p}(\Omega)$ denotes the standard Sobolev space and $H^k(\Omega):=W^{k,2}(\Omega)$.

\section{Preliminaries}

\noindent\;\;\;\; The well-known local existence theory for the MHD equations with constant viscosity coefficients, where the initial density is strictly away from vacuum, can be found in \cite{vol,st,fy}. The similar way, that is, using linearization, Schauder fixed point theorem and borrowing a priori estimates in sections 2.3-2.4, can gives the following local existence of classical solutions for the bulk viscosity $\lambda$ be the power of the density. We omit the details here for simplicity.
\begin{lemma}
Under the assumptions of Theorem $(\ref{thm1})$-$(\ref{thm3})$, there exists a $T_*>0$ and a unique classical solution $(\rho,u,H)$ to $(\ref{MHD})-(\ref{IVC2})$ satisfying the regularity properties $(\ref{reg1})$, $(\ref{reg2})$ and $(\ref{reg3}) $with $T$ replaced by $T_*$
\end{lemma}

Several elementary Lemmas are used frequently later. The first ones are the Gagliardo-Nirenberg inequality with best constant and the Sobolev inequality.
\begin{lemma}\label{GN}
(1)$($\cite{ns}$)$For any $h\in W{_0^{1,m}}(\mathbb T^2)$ or $h\in W^{1,m}(\mathbb T^2)$ with $\int_{\mathbb T^2}hdx=0$ or $h\in W^{1,m}(\mathbb R^2)\cap L^r(\mathbb R^2)$, it holds that
\begin{equation}
\|h\|_q\leq C\|\nabla h\|{_m^\theta}\|h\|{_r^{1-\theta}},
\end{equation}
where $\theta=(\frac{1}{r}-\frac{1}{q})(\frac{1}{r}-\frac{1}{m}+\frac{1}{2})^{-1}$, and if $m<2$, then $q\in[r,\frac{2m}{2-m}]$, if $m=2$, then $q\in[r,+\infty)$, if $m>2$, then $q\in[r,+\infty]$.
In particular, for any $f\in H^1(\mathbb{T}^2)$ with $\int_{\mathbb{T}^2}fdx=0$ or $f\in H^1(\mathbb R^2)$, it holds for any $p\in[2,+\infty)$ that
\begin{equation}
\|f\|_p\leq C\|f\|{_2^\frac{2}{p}}\|\nabla f\|{_2^\frac{p-2}{p}}.
\end{equation}

(2)$($\cite{Del}$)$(Best constant for the Gagliardo-Nirenberg inequality)\\
For any $h\in\mathbb D^m(\mathbb R^2):=\{h\in L^{m+1}(\mathbb R^2)|\nabla h\in L^2(\mathbb R^2),h\in L^{2m}(\mathbb R^2)\}$ with $m>1$, it holds that
\begin{equation}
\|h\|_{2m}\leq A_m\|\nabla h\|{_2^\theta}\|u\|{_{m+1}^{1-\theta}},
\end{equation}
where $\theta=\frac{1}{2}-\frac{1}{2m}$ and
\begin{equation*}
A_m=(\frac{m+1}{2\pi})^\frac{\theta}{2}(\frac{2}{m+1})^\frac{1}{2m}\leq Cm^\frac{1}{4}
\end{equation*}
with the positive constant $C$ independent of $m$, and $A_m$ is the optimal constant.
\end{lemma}

The following Lemma is the Sobolev-Poincar\'{e} inequality.
\begin{lemma}\label{SP}
(1)$($\cite{gt}$)$ For any $h\in W{_0^{1,m}}(\mathbb T^2)$ or $h\in W^{1,m}(\mathbb T^2)$ with $\int_{\mathbb T^2}hdx=0$ or $h\in W^{1,m}(\mathbb R^2)$, if $1\leq m<2$, then
\begin{equation}
\|h\|_\frac{2m}{2-m}\leq C(2-m)^{-\frac{1}{2}}\|\nabla h\|_m,
\end{equation}
where the positive constant $C$ is independent of $m$.

(2)$($\cite{fei04}$)$ Let $v\in H^1(\mathbb T^2)$, and let $\rho$ be a non-negative function such that
\begin{equation}
0<M\leq\int_{\mathbb T^2}\rho dx,\\\\\int_{\mathbb T^2}\rho^\gamma dx\leq E_0
\end{equation}
with $\gamma>1$. Then there exists a constant $C$ depending solely on $M,E_0$ such that 
\begin{equation}
\|v\|{_2^2}\leq C(\int_{\mathbb T^2}\rho v^2dx+\|\nabla v\|{_2^2}).
\end{equation}
\end{lemma}

Next, the following Caffarelli-Kohn-Nirenberg weighted inequalities is crucial to the weighted estimates in the two-dimensional Cauchy problem on the whole space $\mathbb{R}^2$.
\begin{lemma}\label{CKN}
(1)$($\cite{ckn}$)$ For any $h\in C{_0^\infty}(\mathbb R^2)$, it holds that
\begin{equation}
\||x|^\kappa h\|_r\leq C\||x|^\alpha|\nabla h|\|{_p^\theta}\||x|^\beta h\|{_q^{(1-\theta)}}
\end{equation}
where $1\leq p,q <+\infty, 0<r<+\infty, 0\leq \theta\leq 1,\frac{1}{p}+\frac{\alpha}{2}>0, \frac{1}{q}+\frac{\beta}{2}>0, \frac{1}{r}+\frac{\kappa}{2}>0$ and satisfying 
\begin{equation}
\frac{1}{r}+\frac{\kappa}{2}=\theta(\frac{1}{p}+\frac{\alpha-1}{2})+(1-\theta)(\frac{1}{q}+\frac{\beta}{2}),
\end{equation}
and
\begin{equation*}
\kappa=\theta\sigma+(1-\theta)\beta,
\end{equation*}
with $0\leq \alpha-\sigma$ if $\theta>0$ and $0\leq\alpha-\sigma\leq 1$ if $\theta>0$ and $\frac{1}{p}+\frac{\alpha-1}{2}=\frac{1}{r}+\frac{\kappa}{2}$.

(2)$($\cite{cw}$)$(Best constant for Caffarelli-Kohn-Nirenberg inequality)\\
For any $h\in C{_0^\infty}(\mathbb R^2)$, it holds that
\begin{equation}
\||x|^bh\|_p\leq C_{a,b}\||x|^a\nabla h\|_2
\end{equation} 
where $a>0, a-1\leq b\leq a$ and $p=\frac{2}{a-b}$. If $b=a-1$, then $p=2$ and the best constant in the above inequality is  $C_{a,b}=C_{a,a-1}=a$.
\end{lemma}

The following weighted-$L^p$-estimates can be proved through the $A_p$-weighted theory in Stein\cite{stein}(or see\cite{jwx2})
\begin{lemma}\label{WLP}
(1) It holds that for any $1<p<+\infty$ and $u\in C{_0^\infty}(\mathbb R^2)$,
\begin{equation}
\|\nabla u\|_p\leq C(\|\di u\|_p+\|\omega\|_p);
\end{equation}

(2) It holds that for $1<p<+\infty, -2<\alpha<2(p-1)$ and $u\in C{_0^\infty}(\mathbb R^2)$,
\begin{equation}
\||x|^\frac{\alpha}{p}|\nabla u|\|_p\leq C(\||x|^\frac{\alpha}{p}\di u\|_p+\||x|^\frac{\alpha}{p}\omega\|_p).
\end{equation}
\end{lemma}

The following Brezis-Wainger inequalities and properties of the commutator $[b,R_iR_j](f)$ will be used to derive the upper bound of the density $\rho$.
\begin{lemma}(\cite{bw,eng})\label{BW}
Let $\Omega$ be $\mathbb{T}^2$ or $\mathbb{R}^2$. For $q>2$, there exists a positive constant $C$ depending only on $q$ such that every function $v\in W^{1,q}(\Omega)$ satisfies
\begin{equation}
\|v\|_\infty\leq C\|\nabla v\|_2\log^\frac{1}{2}(e+\|\nabla v\|_q)+C\|v\|_2+C.
\end{equation}
\end{lemma}
\begin{lemma}(\cite{coif1,coif2})\label{CE}
Let $b,f\in C^\infty(\mathbb T^2)$ or $C{_0^\infty}(\mathbb{R}^2)$. Then for $p\in(1,+\infty)$, there exists a constant $C(p)$ such that 
\begin{equation}
\|[b,R_iR_j](f)\|_p\leq C\|b\|_{BMO}\|f\|_p.
\end{equation}
Moreover, for $q_k\in(1,+\infty)(k=1,2,3)$ with $\frac{1}{q_1}=\frac{1}{q_2}+\frac{1}{q_3}$, there exists $C$ depending on $q_k$ such that
\begin{equation}
\|\nabla[b,R_iR_j](f)\|_{q_1}\leq C\|\nabla b\|_{q_2}\|f\|_{q_3},
\end{equation}
where $[,]$ and $R_i$ are standard Lie bracket and Riesz transform respectively, that is, 
\begin{equation*}
[b,R_iR_j](f):=bR_i\circ R_j(f)-R_i\circ R_j(bf),\quad i,j=1,2.
\end{equation*}
\end{lemma}

The following Beale-Kato-Majda type inequality\cite{BKM,hlx2} will be crucial to derive the first order derivative estimate of $u$.
\begin{lemma}\label{BKM}
For $2<q<+\infty$, there exists a positive constant $C$ may depend on $q$ such that the following estimate holds for all $\nabla u\in W^{1,q}(\mathbb{T}^2)$ or $W^{1,q}(\mathbb{R}^2)\cap L^2(\mathbb{R}^2)$,
\begin{equation}
\|\nabla u\|_\infty\leq C(\|\di u\|_\infty+\|\omega\|_\infty)\log(e+\|\nabla^2u\|_q)+C\|\nabla u\|_2+C.
\end{equation}
\end{lemma}

Finally, the following well-known Aubin-Lions Lemma is the key to guarantee that the solution with regularity shown in (\ref{reg1}), (\ref{reg2}) and (\ref{reg3}) is a classical solution.
\begin{lemma}\label{AB}(\cite{simon})
Let $X,Y,Z$ be three Banach spaces with $X\subset Y\subset Z$. Suppose that $X$ is compactly embedded in $Y$ and that $Y$ is continuously embedded in $Z$.

(1) Let $G$ be bounded in $L^p(0,T;X)$ where $1\leq p< +\infty$, and $\frac{\partial G}{\partial t}$ be bounded in $L^1(0,T;Z)$. Then $G$ is relatively compact in $L^p([0,T];Y)$.

(2) Let $F$ be bounded in $L^\infty(0,T;X)$ and $\frac{\partial F}{\partial t}$ be bounded in $L^r(0,T;Z)$ where $r>1$. Then $F$ is relatively compact in $C([0,T];Y)$.
\end{lemma} 

\section{A priori estimates(I)}

\noindent\;\;\;\; In this section, we will give the a priori estimates for the (IVP) of $(\ref{MHD})$ and $(\ref{IVC1})$ on the periodic domain $\mathbb{T}^2$ under the assumption $\inf\limits_{x\in\mathbb{T}^2}\rho_0\geq\delta>0$. These estimates is uniform with respect to $\delta$.

\subsection{Lower and Upper bound of $\rho$}

\noindent\;\;\;\; First, we derive the elementary energy estimates. 
\begin{lemma}\label{BE}
There exists a positive constant $C$ only depending on $(\rho_0,u_0,H_0)$, such that
\begin{align}\label{3.1}
\sup\limits_{0 \leq t\leq T}&(\|\sqrt{\rho}u\|{^2_2}+\|\rho\|{^\gamma_\gamma}+\|\rho\|_1+\|H\|{^2_2})\nonumber\\
&+\int_{0}^{T}(\|\nabla u\|{^2_2}+\|\omega\|{^2_2}+\|(2\mu+\lambda(\rho))^{1/2}\di u\|{^2_2}+\|\nabla H\|{^2_2})dt\leq C.
\end{align}
\end{lemma}

\begin{proof}
Multiplying the equation $(\ref{MHD})_2$ by $u$ and the equation $(\ref{MHD})_3$ by $H$, then integrating over $\mathbb{T}^2\times[0,t]$, we can obtain that
\begin{align}\label{3.2}
&\int(\rho\frac{|u|^2}{2}+\frac{|H|^2}{2}+\frac{1}{\gamma-1}{\rho}^{\gamma})dx+\int_{0}^{t}\int[\mu|\nabla u|^2+\nu|\nabla H|^2+(\mu+\lambda(\rho)(\di u)^2)]dxdt\nonumber\\
&=\int(\rho_0\frac{|u_0|^2}{2}+\frac{|H_0|^2}{2}+\frac{1}{\gamma-1}{\rho_0}^{\gamma})dx\leq C
\end{align}
where we have used integration by parts, the continuity equation $(\ref{MHD})_1$ and the equation for the pressure writing as
\begin{equation}
P_t+u\cdot\nabla P+\gamma P\di u=0
\end{equation}
Note that
\begin{equation}
\int\rho dx=\int\rho_0dx,\nonumber
\end{equation}
and 
\begin{equation}
\int[\mu|\nabla u|^2+(\mu+\lambda(\rho))(\di u)^2]dx=\int[\mu{\omega}^2+(2\mu+\lambda(\rho))(\di u)^2]dx.\nonumber
\end{equation}
So, the proof of lemma is completed by combining the above equalities together.
 \end{proof}

Next we derive a priori $L{^\infty_t} L{^p_x}$ estimates for the magnetic field $H$ in 2D case which is a crucial different point and a fundamental observation for the well-posedness of compressible MHD equations in comparison with compressible Navier-Stokes Equations. 
\begin{lemma}\label{BH}
For any $p\geq 2$, there exists a positive constant $C$ such that 
\begin{equation}\label{3.4}
\sup\limits_{0\leq t\leq T}\|H\|_p\leq C. 
\end{equation}
\end{lemma}
\begin{proof}
Multiplying the equation $(\ref{MHD})_3$ by $p|H|^{p-2}H$ and integrating over $\mathbb{T}^2$, we obtain, by using of the Gagliardo-Nirenberg inequality, that
\begin{align}\label{3.5}
&\frac{d}{dt}\int|H|^pdx+\nu p\int|H|^{p-2}|\nabla H|^2dx+\nu p\int\nabla\frac{|H|^2}{2}\cdot\nabla|H|^{p-2}dx\nonumber\\
&=(1-p)\int|H|^p\di udx+p\int H\cdot \nabla u\cdot |H|^{p-2}Hdx\nonumber\\
&\leq C\int|H|^p|\nabla u|dx\leq C\|H^p\|_2\|\nabla u\|_2\leq C\|H^\frac{p}{2}\|_2\|\nabla|H|^\frac{p}{2}\|_2\|\nabla u\|_2\\
&\leq \frac{\nu p}{2}\int|H|^{p-2}|\nabla H|^2dx+C\|\nabla u\|{_2^2}\|H\|{_p^p}.\nonumber
\end{align}
which yields that
\begin{center}
$\frac{d}{dt}\|H\|{_p^p}\leq C\|\nabla u\|{_2^2}\|H\|{_p^p}$.
\end{center}
Applying Gronwall's inequality and using Lemma \ref{BE}, we have
\begin{center}
$\sup\limits_{0\leq t\leq T}\|H\|_p\leq C.$
\end{center}
 \end{proof}
 
Applying the operator $\di$ to the momentum equation $(\ref{MHD})_2$, we have
\begin{equation}\label{3.6}
[\di(\rho u)]_t+\di[\di(\rho u\otimes u-H\otimes H)]=\Delta F,
\end{equation}
where the effective flux $F$ is defined by
\begin{equation}
F:=(2\mu+\lambda(\rho))\di u-P-\frac{|H|^2}{2}.
\end{equation}
Consider the following two elliptic problems:
\begin{align}
-\Delta\xi=\di(\rho u),\int\xi(t,x)dx=0,\label{3.8}\\
-\Delta\eta=\di[\di(\rho u\otimes u-H\otimes H)],\int\eta(t,x)dx=0,\label{3.9}
\end{align}
both with the periodic boundary conditions on $\mathbb{T}^2$.Then, we can derive the following elliptic estimates. It can be easily established through a similar way in \cite{vk} and \cite{jwx1}. So, we omit it here for simplicity.  
\begin{lemma}\label{Ell1}
\begin{flalign*}
\begin{split}
&\quad(1)\|\nabla\xi\|_{2m}\leq Cm\|\rho\|_{\frac{2mk}{k-1}}\|u\|_{2mk},~ for~ any~ k>1,~m\geq 1;\\
&\quad(2)\|\nabla\xi\|_{2-r}\leq C\|\rho\|{_\frac{2-r}{r}^\frac{1}{2}},~ for~ any ~0<r<1;\\
&\quad(3)\|\eta\|_{2m}\leq Cm(\|\rho\|_\frac{2mk}{k-1}\|u\|{_{4mk}^2}+\|H\|{_{4m}^2}),~ for~ any~ k>1,~m\geq 1,
\end{split}&
\end{flalign*}
where $C$ are positive constants independent of $m,k$ and $r$.
\end{lemma}

\begin{lemma}\label{Ell2}
\begin{flalign*}
\begin{split}
&\quad(1)\|\xi\|_{2m}\leq Cm^\frac{1}{2}\|\nabla\xi\|_\frac{2m}{m+1}\leq Cm^\frac{1}{2}\|\rho\|{_m^\frac{1}{2}}, ~for~ any~ m\geq 2;\\
&\quad(2)\|u\|_{2m}\leq C[m^\frac{1}{2}\|\nabla u\|_2+1], ~for ~any~ m\geq 1;\\
&\quad(3)\|\nabla\xi\|_{2m}\leq C[m^\frac{3}{2}k^\frac{1}{2}\|\rho\|_\frac{2mk}{k-1}\phi(t)^\frac{1}{2}+m\|\rho\|_\frac{2mk}{k-1}],~ for~ any~ m\geq 1,k>1;\\
&\quad(4)\|\eta\|_{2m}\leq C[m^2k\|\rho\|_\frac{2mk}{k-1}+m\|\rho\|_\frac{2mk}{k-1}+m^2\phi(t)+m],~for~ any~ m\geq 1,k>1,
\end{split}&
\end{flalign*}
where $C$ are positive constants independent of $m,k$,and,
\begin{equation}\label{3.10}
\phi(t):=\int(\mu\omega^2+(2\mu+\lambda(\rho))(\di u)^2+\nu|\nabla H|^2)dx,~~~~~~t\in[0,T]
\end{equation}
\end{lemma}
 
Substituting $(\ref{3.8})$ and $(\ref{3.9})$ into $(\ref{3.6})$ yields that
\begin{equation}\label{3.11}
\Delta          (\xi_t+\eta+F-\int F(t,x)dx)=0
\end{equation}
with 
\begin{equation}
\int (\xi_t+\eta+F-\int F(t,x)dx)dx=0.
\end{equation}
Thus it holds that
\begin{equation}\label{3.13}
\xi_t+\eta+F-\int F(t,x)dx=0.
\end{equation}
Define 
\begin{equation}\label{3.14}
\Lambda(\rho):=\int_{1}^{\rho}\frac{2\mu+\lambda(s)}{s}ds=2\mu\log\rho+\frac{1}{\beta}(\rho^\beta-1).
\end{equation}
Using the definition of the effective viscous flux $F$ and the continuity equation $(\ref{MHD})_1$, one has
\begin{equation}\label{3.15}
(\Lambda(\rho)-\xi)_t+u\cdot\nabla(\Lambda(\rho)-\xi)+P+\frac{|H|^2}{2}-\eta+u\cdot{\nabla\xi}+\int F(t,x)dx=0
\end{equation}
 
Next, we derive the $L{^\infty_t}L{^p_x}$ estimate of the density using of the transport equation (\ref{3.15}) similar to Lemma 3.5 in \cite{jwx1}.
\begin{lemma}\label{BP}
Assume $\beta>1$, for any $p\geq 1$,
\begin{equation}\label{3.16}
\sup\limits_{0\leq t\leq T}\|\rho\|_p\leq Cp^\frac{2}{\beta-1},
\end{equation}
where C is a positive constant independent of $p$.
\end{lemma}

\begin{proof}
Multiplying (\ref{3.15}) by $\rho[(\Lambda(\rho)-\xi)_+]^{2m-1}$ and integrating over $\mathbb{T}^2$,where $(h)_+$ denotes the positive part of the function $h$, it holds that
\begin{align}\label{3.17}
&\frac{1}{2m}\frac{d}{dt}\int\rho[(\Lambda(\rho)-\xi)_+]^{2m}dx\nonumber\\
&~~~~+\int\rho P[(\Lambda(\rho)-\xi)_+]^{2m-1}dx+\frac{1}{2}\int\rho |H|^2[(\Lambda(\rho)-\xi)_+]^{2m-1}dx\nonumber\\
&=\int\rho\eta[(\Lambda(\rho)-\xi)_+]^{2m-1}dx-\int\rho u\cdot\nabla\xi[(\Lambda(\rho)-\xi)_+]^{2m-1}dx\\
&~~~~+(\int F(t,x)dx)(\int\rho[(\Lambda(\rho)-\xi)_+]^{2m-1}dx)=:\sum_{i=1}^{3}K_i\nonumber
\end{align}
Define
\begin{equation}
f(t):=\{\int\rho[(\Lambda(\rho)-\xi)_+]^{2m}dx\}^\frac{1}{2m}\nonumber 
\end{equation}
Next we estimate the terms on the RHS of (\ref{3.17}) 
\begin{align}\label{3.18}
|K_1|&\leq\int\rho^\frac{1}{2m}|\eta|[\rho(\Lambda(\rho)-\xi){_+^{2m}}]^\frac{2m-1}{2m}dx\leq C||\rho||{_{2m\beta+1}^\frac{1}{2m}}||\eta||_{2m+\frac{1}{\beta}}||\rho(\Lambda(\rho)-\xi){_+^{2m}}||{_1^\frac{2m-1}{2m}}\nonumber\\
&\leq C||\rho||{_{2m\beta+1}^\frac{1}{2m}}[(m+\frac{1}{2\beta})^2k||\rho||_{\frac{2(m+\frac{1}{2\beta})k}{k-1}}\phi(t)+(m+\frac{1}{2\beta})||\rho||_{\frac{2(m+\frac{1}{2\beta}k)}{k-1}}]f^{2m-1}(t)\nonumber\\
&\leq C||\rho||{_{2m\beta+1}^{1+\frac{1}{2m}}}f^{2m-1}(t)[m^2\phi(t)+m],
\end{align}
where using of Lemma \ref{Ell2} and choosing $k=\frac{\beta}{\beta-1}$.
\begin{align}\label{3.19}
|K_2|&\leq\int\rho^\frac{1}{2m}|u||\nabla\xi|[\rho(\Lambda(\rho)-\xi){_+^{2m}}]^{2m-1}dx\nonumber\\
&\leq C||\rho||{_{2m\beta+1}^\frac{1}{2m}}||u||_{2mp}||\nabla\xi||_{2mq}||\rho(\Lambda(\rho)-\xi){_+^{2m}}||{_1^\frac{2m-1}{2m}}\nonumber\\
&\leq C||\rho||{_{2m\beta+1}^\frac{1}{2m}}[(mp)^\frac{1}{2}||\nabla u||_2+1][(mq)^\frac{3}{2}k^\frac{1}{2}||\rho||_\frac{2mqk}{k-1}\phi^\frac{1}{2}(t)+m||\rho||_\frac{2mqk}{k-1}]f^{2m-1}(t)\nonumber\\
&\leq C||\rho||{_{2m\beta+1}^\frac{1}{2m}}f^{2m-1}(t)[m^\frac{1}{2}\phi^\frac{1}{2}(t)+1][m^\frac{3}{2}\phi^\frac{1}{2}(t)+m]\\
&\leq C||\rho||{_{2m\beta+1}^\frac{1}{2m}}f^{2m-1}(t)[m^2\phi(t)+m]\nonumber
\end{align}
where using of Lemma(2.3.3) and choosing $p=q=\frac{2m\beta+1}{\beta},k=\frac{\beta}{\beta-1}$.
\begin{align}\label{3.20}
|K_3|&\leq\int|(2\mu+\lambda(\rho))\di u-P-\frac{1}{2}|H|^2|dx\int\rho^\frac{1}{2m}[\rho(\Lambda(\rho)-\xi){_+^{2m}}]^{\frac{2m-1}{2m}}dx\nonumber\\
&\leq[(\int(2\mu+\lambda(\rho))^\frac{1}{2}(\di u)^2(\int 2\mu+\lambda(\rho))^\frac{1}{2}+\int P(\rho)dx+\frac{1}{2}\int|H|^2dx]\nonumber\\
&~~~~||\rho||{_1^\frac{1}{2m}}||\rho(\Lambda(\rho)-\xi){_+^2m}||{_1^\frac{2m-1}{2m}}\\
&\leq C[\phi^\frac{1}{2}(t)+\phi^\frac{1}{2}(t)(\int\rho^\beta dx)^\frac{1}{2}]f^{2m-1}(t)\nonumber\\
&\leq Cf^{2m-1}(t)[\phi^\frac{1}{2}(t)+\phi^\frac{1}{2}(t)||\rho||{_{2m\beta+1}^\frac{\beta}{2}}+1].\nonumber
\end{align}
Substituting (\ref{3.18})-(\ref{3.20}) into (\ref{3.17}), one has
\begin{equation}\label{3.21}
\frac{d}{dt}f(t)\leq C[1+\phi^\frac{1}{2}(t)+\phi^\frac{1}{2}(t)||\rho||{_{2m\beta+1}^\frac{\beta}{2}}+(m^2\phi(t)+m)||\rho||{_{2m\beta+1}^{1+\frac{1}{2m}}}].
\end{equation}
Integrating over [0,t], we have
\begin{equation}
f(t)\leq f(0)+C[1+\int_{0}^{t}\phi^\frac{1}{2}(\tau)||\rho||{_{2m\beta+1}^\frac{\beta}{2}}(\tau)d\tau+\int_{0}^{t}(m^2\phi(\tau)+m)||\rho||{_{2m\beta+1}^{1+\frac{1}{2m}}}(\tau)d\tau]\nonumber
\end{equation}
where $f(0)=(\int\rho_0[(\Lambda(\rho_0)-\xi_0)_+]^{2m}dx)^\frac{1}{2m}$.\\
It is easy to show that $||\xi_0||_\infty\leq C$. By the definition of $\Lambda(\rho_0)=2\mu\log\rho_0+\frac{1}{\beta}(\rho{_0^\beta}-1)$, we have $\Lambda(\rho_0)-\xi_0 \rightarrow -\infty$ as $\rho_0 \rightarrow 0^+$. Thus, there exists a $\sigma>0$ such that if $0\leq \rho_0\leq\sigma$, then $(\Lambda(\rho_0)-\xi_0)_+\equiv 0$.\\ 
Now, one has 
\begin{equation}
f(0)=[\int_{\sigma}^{M}\rho_0(\Lambda(\rho_0)-\xi_0)^{2m}dx]^\frac{1}{2m}\leq C(\sigma,M).\nonumber
\end{equation}
It follows that
\begin{equation}\label{3.22}
f(t)\leq C[1+\int_{0}^{t}\phi^\frac{1}{2}(\tau)||\rho||{_{2m\beta+1}^\frac{\beta}{2}}(\tau)d\tau+\int_{0}^{t}(m^2\phi(\tau)+m)||\rho||{_{2m\beta+1}^{1+\frac{1}{2m}}}(\tau)d\tau].
\end{equation}
Set $\Omega_1(t)=\{{x\in\mathbb{T}^2|\rho(t,x)>2}\}$ and $\Omega_2(t)=\{{x\in\Omega_1(t)|\Lambda(\rho)(t,x)-\xi(t,x)>0}\}$. Then one has
\begin{align}\label{3.23}
||&\rho||{_{2m\beta+1}^\beta}(t)=(\int_{\Omega_1(t)}\rho^{2m\beta+1}dx+\int_{{\mathbb{T}^2-\Omega_1(t)}}\rho^{2m\beta+1}dx)^\frac{\beta}{2m\beta+1}\nonumber\\
&\leq (\int_{\Omega_1(t)}\rho^{2m\beta+1}dx)^\frac{\beta}{2m\beta+1}+C\leq C(\int_{\Omega_1(t)}\rho|\Lambda(\rho)|^{2m}dx)^\frac{\beta}{2m\beta+1}+C\nonumber\\
&=C(\int_{\Omega_2(t)}\rho|\Lambda(\rho)-\xi+\xi|^{2m}dx+\int_{\Omega_1(t)-\Omega_2(t)}\rho|\Lambda(\rho)|^{2m}dx)^\frac{\beta}{2m\beta+1}+C\\
&\leq C(\int_{\Omega_2(t)}\rho|\Lambda(\rho)-\xi|^{2m}dx+\int_{\Omega_2(t)}\rho|\xi|^{2m}dx+\int_{\Omega_1(t)-\Omega_2(t)}\rho|\xi|^{2m}dx)^\frac{\beta}{2m\beta+1}+C\nonumber\\
&\leq C(f^{2m}(t)+\int\rho|\xi|^{2m})^\frac{\beta}{2m\beta+1}+C\leq C[1+f(t)+(\int\rho|\xi|^{2m}dx)^\frac{\beta}{2m\beta+1}].\nonumber
\end{align}
Note that 
\begin{align}\label{3.24}
(\int&\rho|\xi|^{2m}dx)^\frac{\beta}{2m\beta+1}\leq||\rho||{_{2m\beta+1}^\frac{\beta}{2m\beta+1}}||\xi||{_{2m+\frac{1}{\beta}}^\frac{2m\beta}{2m\beta+1}}\nonumber\\
&\leq C||\rho||{_{2m\beta+1}^\frac{\beta}{2m\beta+1}}[(m+\frac{1}{2\beta})^\frac{1}{2}||\rho||{_{m+\frac{1}{2\beta}}^\frac{1}{2}}]^\frac{2m\beta}{2m\beta+1}\leq Cm^\frac{1}{2}||\rho||{_{2m\beta+1}^\frac{\beta(m+1)}{2m\beta+1}},
\end{align}
we can obtain
\begin{align}\label{3.25}
||\rho||{_{2m\beta+1}^\beta}(t)&\leq C[1+f(t)+m^\frac{1}{2}||\rho||{_{2m\beta+1}^\frac{\beta(m+1)}{2m\beta+1}}(t)]\nonumber\\
&\leq \frac{1}{2}||\rho||{_{2m\beta+1}^\beta}(t)+C(1+f(t)+m^\frac{m\beta+\frac{1}{2}}{m(2\beta-1)}).
\end{align}
Thus it holds that
\begin{align}\label{3.26}
&||\rho||{_{2m\beta+1}^\beta}(t)\leq C[f(t)+m^\frac{\beta}{2\beta-1}]\nonumber\\
&\leq C[m^\frac{\beta}{2\beta-1}+\int_{0}^{t}\phi^\frac{1}{2}(\tau)||\rho||{_{2m\beta+1}^\frac{\beta}{2}}(\tau)d\tau+\int_{0}^{t}(m^2\phi(\tau)+m)||\rho||{_{2m\beta+1}^{1+\frac{1}{2m}}}(\tau)d\tau]\nonumber\\
&\leq C[m^\frac{\beta}{2\beta-1}+\int_{0}^{t}||\rho||{_{2m\beta+1}^\beta}(\tau)d\tau+\int_{0}^{t}(m^2\phi(\tau)+m)||\rho||{_{2m\beta+1}^{1+\frac{1}{2m}}}(\tau)d\tau]
\end{align}
Applying Gronwall's inequality yields that
\begin{equation}\label{3.27}
||\rho||{_{2m\beta+1}^\beta}(t)\leq C[m^\frac{\beta}{2\beta-1}+\int_{0}^{t}(m^2\phi(\tau)+m)||\rho||{_{2m\beta+1}^{1+\frac{1}{2m}}}(\tau)d\tau]
\end{equation}
Denote 
\begin{equation}
y(t)=m^{-\frac{2}{\beta-1}}||\rho||_{2m\beta+1}(t)\nonumber.
\end{equation}
Then it holds that 
\begin{align}
y^\beta(t)&\leq C[m^\frac{\beta(1-3\beta)}{(2\beta-1)(\beta-1)}+m^\frac{1}{m(\beta-1)}\int_{0}^{t}\phi(\tau)y(\tau)^{1+\frac{1}{2m}}d\tau+m^{\frac{1}{m(\beta-1)}-1}\int_{0}^{t}y^{1+\frac{1}{2m}}(\tau)d\tau]\nonumber\\
&\leq C[1+\int_{0}^{t}(\phi(\tau)+1)y^\beta(\tau)d\tau].\nonumber
\end{align}
So applying the Gronwall's inequality again yields that 
\begin{equation}
y(t)\leq C,~~~~~\forall t\in[0,T]\nonumber.
\end{equation}
That is,
\begin{equation}
||\rho||_{2m\beta+1}(t)\leq Cm^\frac{2}{\beta-1},~~~~~\forall t\in[0,T].
\end{equation}
For $1<p<2\beta+1$, using of interpolation inequality, we have completed the proof of the Lemma.
 \end{proof}
   
\begin{lemma}\label{lem3.6}
For any $\varepsilon>0$, there exists a positive constant $C(\varepsilon)$ such that
\begin{equation}\label{3.31}
\sup\limits_{0\leq t\leq T}\log(e+Z^2(t))+\int_{0}^{T}\frac{\varphi^2(t)}{e+Z^2(t)}dt\leq C\Phi{_T^{1+\beta\varepsilon}},
\end{equation}
where $Z^2(t):=\int(\mu\omega^2+\frac{F^2}{2\mu+\lambda(\rho)}+|\nabla H|^2)dx$, $\varphi^2(t):=\int(\rho|\dot{u}|^2+|\nabla^2H|^2)dx,$ $\Phi_T:=\|\rho\|_\infty+1$.
\end{lemma}
\begin{proof}
First, we can rewrite the momentum equation into
\begin{equation}\label{3.32}
\rho\dot{u}=\nabla F+\mu\nabla^\perp\omega+H\cdot\nabla H,
\end{equation} 
where $\nabla^\perp:=(\partial_2,-\partial_1)$.\\
Multiplying (\ref{3.32}) by $\dot{u}$, integrating over $\mathbb{T}^2$ and using of integration by parts, one has
\begin{equation}\label{3.33}
\int\rho|\dot{u}|^2dx=-\int F\di\dot{u}dx-\mu\int\omega\nabla^\perp\cdot\dot{u}dx+\int H\cdot\nabla H\cdot\dot{u}dx.
\end{equation}
Direct calculations show that
\begin{align}\label{3.34}
\di\dot{u}&=\frac{D}{Dt}\di u+(\partial_1u\cdot\nabla)u_1+(\partial_2u\cdot\nabla)u_2\nonumber\\
&=\frac{D}{Dt}\di u-2\nabla u_1\cdot\nabla^\perp u_2+(\di u)^2,
\end{align}
\begin{equation}\label{3.35}
\nabla^\perp\cdot\dot{u}=\frac{D}{Dt}\omega-(\partial_1u\cdot\nabla)u_2+(\partial_2u\cdot\nabla)u_1=\frac{D}{Dt}\omega+\omega\di u.
\end{equation}
In order to obtain the upper bound of the density, the following observation is the key to handle the strongly coupled magnetic field with the velocity field, which help us to avoid the a priori estimates of $\nabla \dot{u}$.
\begin{align}\label{3.36}
&\int H\cdot\nabla H\cdot\dot{u}dx=\int H\cdot\nabla H\cdot(\partial_tu+u\cdot\nabla u)dx\nonumber\\
&=\frac{d}{dt}\int H\cdot\nabla H\cdot u-\int(H_t\cdot\nabla H\cdot u+H\cdot\nabla H_t\cdot u)dx-\int H\cdot\nabla(u\cdot\nabla u)\cdot H\nonumber\\
&=-\frac{d}{dt}\int H\otimes H:\nabla udx+\int(H_t\cdot\nabla u\cdot H+H\cdot\nabla u\cdot H_t)dx\nonumber\\
&~~~~-\int(H\cdot\nabla u^i\partial_iu\cdot H+H^ku^i\partial{_{ik}^2}u\cdot H)dx\nonumber\\
&=-\frac{d}{dt}\int H\otimes H:\nabla udx+\int(H_t\cdot\nabla u\cdot H+H\cdot\nabla u\cdot H_t)dx\\
&~~~~-\int[H\cdot\nabla u^k\partial_ku\cdot H-u\cdot\nabla H^k\partial_ku\cdot H-H\cdot\nabla u\cdot H\di u-(H\cdot\nabla u)\cdot(u\cdot\nabla H)]dx\nonumber\\
&=-\frac{d}{dt}\int H\otimes H:\nabla udx+\int(H{_t^k}-H\cdot\nabla u^k+u\cdot \nabla H^k)\partial_k u\cdot Hdx\nonumber\\
&~~~~+\int H\cdot\nabla u(H_t+H\di u+u\cdot\nabla H)dx\nonumber\\
&=-\frac{d}{dt}\int H\otimes H:\nabla udx+\int(\nu\Delta H-H\di u)\cdot \nabla u\cdot Hdx+\int H\cdot\nabla u\cdot(\nu\Delta H+H\cdot\nabla u)dx.\nonumber
\end{align}
It follows from $(\ref{3.33})$-$(\ref{3.36})$  and integration by parts that
\begin{align}\label{3.37}
\frac{1}{2}&\frac{d}{dt}\int(\mu\omega^2+\frac{F^2}{2\mu+\lambda(\rho)})+\int\rho|\dot{u}|^2\nonumber\\
&=-\frac{d}{dt}\int H\otimes H:\nabla udx-\frac{\mu}{2}\int\omega^2\di udx+2\int F\nabla u_1\cdot\nabla^\perp u_2dx\nonumber\\
&~~~~+\frac{1}{2}\int F^2\di u[{\rho(\frac{1}{2\mu+\lambda(\rho)})'}  -\frac{1}{2\mu+\lambda(\rho)}]dx\nonumber\\
&~~~~+\int F \di u[{\rho(\frac{P}{2\mu+\lambda(\rho)})'}-\frac{P}{2\mu+\lambda(\rho)}]dx\\
&~~~~+\frac{1}{2}\int FH^2\di u(\rho(\frac{1}{2\mu+\lambda(\varrho)})+\frac{1}{2\mu+\lambda(\rho)})\nonumber\\
&~~~~-\int \frac{F}{2\mu+\lambda(\rho)}(H\cdot\nabla u\cdot H+\nu H\cdot\Delta H)dx\nonumber\\
&~~~~+\int(\nu\Delta H-H\di u)\cdot u\cdot Hdx+\int H\cdot\nabla u\cdot(\nu\Delta H+H\cdot\nabla u)dx.\nonumber
\end{align}
On the other hand, multiplying $(\ref{MHD})_3$ by $-\Delta H$ and integrating over $\mathbb{T}^2$, one has
\begin{equation}\label{3.38}
\frac{1}{2}\frac{d}{dt}\int|\nabla H|^2dx+\nu\int|\nabla^2H|^2dx=\int(u\cdot\nabla H-H\cdot\nabla u+H\di u)\Delta Hdx.
\end{equation}
Combine (\ref{3.37}) and (\ref{3.38}), we have
\begin{align}\label{3.39}
\frac{d}{dt}Z^2+\varphi^2&\leq -\frac{d}{dt}\int H\otimes H:\nabla udx-\mu\int\omega^2\di udx-2\int F\nabla u_1\cdot\nabla\perp u_2dx\nonumber\\
&~~~~+\frac{1}{2}\int F^2\di u[{\rho(\frac{1}{2\mu+\lambda(\rho)})'}  -\frac{1}{2\mu+\lambda(\rho)}]dx\nonumber\\
&~~~~+\int F \di u[{\rho(\frac{P}{2\mu+\lambda(\rho)})'}-\frac{P}{2\mu+\lambda(\rho)}]dx\nonumber\\
&~~~~+\frac{1}{2}\int FH^2\di u(\rho(\frac{1}{2\mu+\lambda(\varrho)})'+\frac{1}{2\mu+\lambda(\rho)})\\
&~~~~-\int \frac{F}{2\mu+\lambda(\rho)}(H\cdot\nabla u\cdot H+\nu H\cdot\Delta H)dx\nonumber\\
&~~~~+\int(\nu\Delta H-H\di u)\cdot u\cdot Hdx+\int H\cdot\nabla u\cdot(\nu\Delta H+H\cdot\nabla u)dx\nonumber\\
&~~~~+\int(u\cdot\nabla H-H\cdot\nabla u+H\di u)\Delta Hdx=-\frac{d}{dt}J_0+\sum_{i=1}^{9}J_9.\nonumber
\end{align}
Note that
\begin{equation}
\Delta F=\di(\rho\dot{u}-H\cdot\nabla H);~~~~ \mu\Delta\omega=\nabla^\perp\cdot(\rho\dot{u}-H\cdot\nabla H)\nonumber.
\end{equation}
It follows from elliptic estimates, (\ref{GN}) and Lemma \ref{BE} that
\begin{equation}\label{3.40}
\|\nabla F\|_2+\|\nabla\omega\|_2\leq C(\|\rho\dot{u}\|_2+\|H\cdot\nabla H\|_2)\leq C(\Phi{_T^\frac{1}{2}}\varphi+\|\nabla H\|_2\varphi^\frac{1}{2}),
\end{equation}
\begin{equation}\label{3.41}
\|\omega\|_4\leq C\|\omega\|{_2^\frac{1}{2}}\|\nabla\omega\|{_2^\frac{1}{2}}\leq CZ^\frac{1}{2}(\Phi{_T^\frac{1}{4}}\varphi^\frac{1}{2}+\|\nabla H\|{_2^\frac{1}{2}}\varphi^\frac{1}{4}).
\end{equation}
Also, one has
\begin{equation}\label{3.42}
\|\nabla u\|{_2^2}+\|\omega\|{_2^2}+\|\nabla H\|{_2^2}+\|(2\mu+\lambda(\rho))^\frac{1}{2}\di u\|{_2^2}\leq C(1+Z^2).
\end{equation}
Now we estimates the terms on the RHS of (\ref{3.39}). First, (\ref{3.41}) and (\ref{3.42}) yields that 
\begin{align}\label{3.43}
|J_1|&\leq C\|\omega\|{_4^2}\|\nabla u\|_2\leq CZ(\Phi{_T^\frac{1}{2}}\varphi+\|\nabla H\|_2\varphi^\frac{1}{2})\|\nabla u\|_2\nonumber\\
&\leq\sigma\varphi^2+C(\Phi_T\|\nabla u\|{_2^2}Z^2+\|\nabla H\|{_2^4}+\|\nabla u\|{_2^2}Z^2)\\
&\leq\sigma\varphi^2+C\Phi_T(1+\|\nabla u\|{_2^2}+\|\nabla H\|{_2^2})(1+Z^2),\nonumber
\end{align}
Next, by using the duality between Hardy $\mathcal{H}^2$ and $\mathcal{BMO}$ spaces(\cite{Fefferman}) and "div-curl" lemma of compensated compactness(\cite{coif3}), since $\C \nabla u_1=\di\nabla^\perp u_2=0$, we can obtain that
\begin{align}\label{3.44}
|J_2|&\leq C\|F\|_{BMO}\|\nabla u_1\cdot\nabla^\perp u_2\|_{\mathcal{H}^1}\leq C\|\nabla F\|_2\|\nabla u_1\|_2\|\nabla^\perp u_2\|_2\nonumber\\
&\leq C(\Phi{_T^\frac{1}{2}}\varphi+\|\nabla H\|_2\varphi^\frac{1}{2})\|\nabla u\|{_2^2}\nonumber\\
&\leq\sigma\varphi^2+C(\Phi_T\|\nabla u\|{_2^4}+\|\nabla u\|{_2^4}+\|\nabla H\|{_2^4})\\
&\leq\sigma\varphi^2+C\Phi_T(\|\nabla u\|{_2^2}+\|\nabla H\|{_2^2})(1+Z^2).\nonumber
\end{align}
Note that $\|F\|_2\leq\Phi{_T^\frac{\beta}{2}}Z$, it follows from (\ref{3.40}) that, for any $\varepsilon>0$,
\begin{align}\label{3.45}
\|\frac{F^2}{2\mu+\lambda(\rho)}\|_2&\leq C\|\frac{F}{\sqrt{2\mu+\lambda(\rho)}}\|{_2^{1-\varepsilon}}\|F\|{_\frac{2(1+\varepsilon)}{\varepsilon}^{1+\varepsilon}}\leq CZ^{1-\varepsilon}\|F\|{_2^\varepsilon}\|\nabla F\|_2\nonumber\\
&\leq CZ\Phi{_T^\frac{\beta\varepsilon}{2}}(\Phi{_T^\frac{1}{2}}\varphi+\|\nabla H\|_2\varphi^\frac{1}{2}).
\end{align}
Thus, we can obtain from Lemma \ref{BP}, (\ref{3.40}), (\ref{3.42}) and (\ref{3.45}) that
\begin{align}\label{3.46}
&|J_3+J_4|\leq C\int |F|^2|\di u|\frac{1}{2\mu+\lambda(\rho)}+|F||\di u|\frac{P}{2\mu+\lambda(\rho)}dx\nonumber\\
&\leq C\|\nabla u\|_2(\|\frac{F^2}{2\mu+\lambda(\rho)}\|_2+\|F\|_{\frac{2(2+\varepsilon)}{\epsilon}}\|P\|_{2+\varepsilon})\nonumber\\
&\leq C\|\nabla u\|_2(\|\frac{F^2}{2\mu+\lambda(\rho)}\|_2+\|F\|{_{2}^\frac{\varepsilon}{2+\varepsilon}}\|\nabla F\|{_2^{\frac{2}{2+\varepsilon}}})\\
&\leq C\|\nabla u\|_2(Z\Phi{_T^\frac{\beta\varepsilon}{2}}(\Phi{_T^\frac{1}{2}}\varphi+||\nabla H||_2\varphi^\frac{1}{2})+\Phi{_T^{\frac{\beta\varepsilon}{2(2+\varepsilon)}}}Z^\frac{\varepsilon}{2+\varepsilon}(\Phi{_T^\frac{1}{2+\varepsilon}}\varphi^\frac{2}{2+\varepsilon}+\|\nabla H\|{_2^\frac{2}{2+\varepsilon}}\varphi^\frac{1}{2+\varepsilon}))\nonumber\\
&\leq \sigma\varphi^2(t)+C\Phi{_T^{1+\beta\varepsilon}}(\|\nabla u\|{_2^2}Z^2+\|\nabla u\|{_2^2}+Z^2+\|\nabla H\|{_2^4})\nonumber\\
&\leq \sigma\varphi^2(t)+C\Phi{_T^{1+\beta\varepsilon}}(1+\|\nabla u\|{_2^2}+\|\nabla H\|{_2^2})(1+Z^2)\nonumber
\end{align}
By Lemma \ref{BH}, (\ref{3.42}) and (\ref{3.45}), we have that
\begin{align}\label{3.47}
&|J_5+J_6|\leq C\int |\frac{F}{2\mu+\lambda(\rho)}|(|H|^2|\nabla u|+|H||\Delta H|)\nonumber\\
&\leq C\|\frac{F}{2\mu+\lambda(\rho)}\|_4(\|H\|{_8^2}\|\nabla u\|_2+\|H\|_4\|\nabla^2H\|_2)\leq C\|\frac{F^2}{2\mu+\lambda(\rho)}\|{_2^\frac{1}{2}}(\|\nabla u\|_2+\|\nabla^2H\|_2)\nonumber\\
&\leq CZ^\frac{1}{2}\Phi{_T^\frac{\beta\varepsilon}{4}}(\Phi{_T^\frac{1}{4}}\varphi^\frac{1}{2}+\|\nabla H\|{_2^\frac{1}{2}}\varphi^\frac{1}{4})(\|\nabla u\|_2+\varphi)\\
&\leq\sigma\varphi^2+C(\Phi{_T^{1+\beta\varepsilon}}Z^2+\|\nabla u\|{_2^2}+\|\nabla H\|{_2^4})\nonumber\\
&\leq\sigma\varphi^2+C\Phi{_T^{1+\beta\varepsilon}}(1+\|\nabla H\|{_2^2})(1+Z^2)\nonumber
\end{align}
Finally, integration by parts, Lemma \ref{BH}, \ref{BP}, (\ref{3.41}), (\ref{3.42}) and (\ref{3.45}) yield that
\begin{align}\label{3.48}
|\sum_{i=7}^{9}J_i|&\leq C\int(|H|^2|\nabla u|^2+|H||\nabla u||\Delta H|+|\nabla u||\nabla H|^2)dx\nonumber\\
&\leq C(\|H\|{_4^2}\|\nabla u\|{_4^2}+\|H\|_4\|\nabla u\|_4\|\nabla^2H\|_2+\|\nabla u\|_2\|\nabla H\|{_4^2})\nonumber\\
&\leq\sigma\varphi^2(t)+C(||\di u||{_4^2}+||\omega||{_4^2}+\|\nabla u\|{_2^2}\|\nabla H\|{_2^2})\\
&\leq\sigma\varphi^2(t)+C(\|\frac{F^2}{2\mu+\lambda(\rho)}\|_2+\|P\|{_4^2}+\|H\|{_8^4}+\|\omega\|{_4^2}+\|\nabla u\|{_2^2}\|\nabla H\|{_2^2})\nonumber\\
&\leq\sigma\varphi^2(t)+C(Z\Phi{_T^\frac{\beta\varepsilon}{2}}(\Phi{_T^\frac{1}{2}}\varphi+\|\nabla H\|_2\varphi^\frac{1}{2})+\|\nabla u\|{_2^2}\|\nabla H\|{_2^2}+1)\nonumber\\
&\leq\sigma\varphi^2(t)+C\Phi{_T^{1+\beta\varepsilon}}(1+||\nabla u||{_2^2}+||\nabla H||{_2^2})(1+Z^2(t)),\nonumber
\end{align}
Moreover, one has
\begin{equation}\label{3.49}
|J_0|\leq\int|H|^2|\nabla u|dx\leq\|\nabla u\|_2\|H\|{_4^2}\leq\|\nabla u\|_2\|\nabla H\|_2\leq\sigma\|\nabla u\|{_2^2}+C\|\nabla H\|{_2^2}.
\end{equation}
Substituting (\ref{3.43}), (\ref{3.44}), (\ref{3.46})-(\ref{3.49}) into (\ref{3.39}), one has
\begin{equation}
\frac{d}{dt}(e+Z^2(t))+\varphi^2(t)\leq C\Phi{_T^{1+\beta\varepsilon}}(1+||\nabla u||{_2^2}+||\nabla H||{_2^2})(e+Z^2(t)).\nonumber
\end{equation} 
Applying Gronwall's inequality and using of Lemma \ref{BE}, the proof of Lemma can be completed.
 \end{proof}

The following $L{^\infty_t}L{^p_x}$ of the momentum will play a crucial role in the estimate of the upper bound of the density as indicated by Lemma 3.4 in \cite{hl2} and \cite{hwc} for compressible Navier-Stokes equations case.
\begin{lemma}\label{lem3.7}
For any $p>3$ and $\sigma>0$, there exists a positive constant $C(\sigma)$ such that 
\begin{equation}\label{3.29}
\|\rho u\|_p\leq C\Phi{_T^{1-\frac{1}{p}+\sigma}}(e+\|\nabla u\|+\|\nabla H\|_2)^{1-\frac{2+\alpha}{p}+\sigma}\log^{\frac{p-2+\alpha}{2p}}(e+(\frac{\varphi^2}{e+Z^2})^\frac{1}{4})
\end{equation}
where $\Phi_T:=\|\rho\|_\infty+1$
\end{lemma}
\begin{proof}
First, for $\alpha:=\frac{\mu^\frac{1}{2}}{2(\mu+1)}\Phi{_T^{-\frac{\beta}{2}}}\in(0,\frac{1}{4}]$,
multiplying the momentum equation $(\ref{MHD})_2$ by $(2+\alpha)|u|^\alpha u$ and integrating over $\mathbb{T}^2$, we obtain,
\begin{align}\label{3.30}
&\frac{d}{dt}\int\rho|u|^{2+\alpha}dx+\mu(2+\alpha)\int|u|^\alpha|\nabla u|^2dx\nonumber\\
&~~~~+\mu(2+\alpha)\int\nabla\frac{|u|^2}{2}\cdot\nabla|u|^\alpha dx+(2+\alpha)\int(\mu+\lambda(\rho))(\di u)^2|u|^\alpha dx\nonumber\\
&=(2+\alpha)\int P\di(u|u|^\alpha)dx-(2+\alpha)\int(\mu+\lambda(\rho))\di uu\cdot\nabla|u|^\alpha dx\\
&~~~~+(2+\alpha)\int(H\cdot \nabla H-\nabla\frac{|H|^2}{2})\cdot u|u|^\alpha dx=:\sum_{i=1}^{3}I_i.\nonumber
\end{align}
Now we estimate the terms on the RHS of (\ref{3.30}). First, we can obtain that
\begin{align}
|I_1|&\leq(2+\alpha)(1+\alpha)\int|P||u|^\alpha|\nabla u|dx\nonumber\\
&\leq\frac{\mu(2+\alpha)}{2}\int|u|^\alpha|\nabla u|^2dx+C\int|P|^2|u|^\alpha dx\nonumber\\
&\leq\frac{\mu(2+\alpha)}{2}\int|u|^\alpha|\nabla u|^2dx+C\|P\|{_{2p_1}^2}\|u\|{_{\alpha q_1}^\alpha}\nonumber\\
&\leq\frac{\mu(2+\alpha)}{2}\int|u|^\alpha|\nabla u|^2dx+C(\|\nabla u\|{_2^2}+1),\nonumber
\end{align}
where $q_1$ is chosen large enough such that $\alpha q_1\geq 2$, then Lemma \ref{Ell2}$(2)$ and Lemma \ref{BP} are used in the last inequality above. Next, it is easy to yield that
\begin{align}
|I_2|&\leq(2+\alpha)\alpha\int(\mu+\lambda(\rho))|\di u||u|^\alpha|\nabla u|dx\nonumber\\
&\leq\frac{2+\alpha}{2}\int(\mu+\lambda(\rho))(\di u)^2|u|^\alpha dx+\frac{2+\alpha}{2}\int\alpha^2(\mu+\lambda(\rho))|u|^\alpha|\nabla u|^2dx\nonumber\\
&\leq\frac{2+\alpha}{2}\int(\mu+\lambda(\rho))(\di u)^2|u|^\alpha dx+\frac{2+\alpha}{2}\int\frac{\mu}{4(\mu+1)^2}\Phi{_T^{-\beta}}(\mu+\Phi{_T^\beta})|u|^\alpha|\nabla u|^2dx\nonumber\\
&\leq\frac{2+\alpha}{2}\int(\mu+\lambda(\rho))(\di u)^2|u|^\alpha dx+\frac{\mu(2+\alpha)}{4}\int|u|^\alpha|\nabla u|^2dx,\nonumber
\end{align}
By integration by parts, we have that
\begin{align}
|I_3|&=|(2+\alpha)\int[H\cdot\nabla(|u|^\alpha u)\cdot H+\frac{|H|^2}{2}(|u|^\alpha\di u+u\cdot\nabla|u|^\alpha)]dx|\nonumber\\
&\leq C\int|H|^2|u|^\alpha|\nabla u|dx\leq C(\|H\|{_8^2}\|u\|_4\|\nabla u\|_2+\|H\|{_4^2}\|\nabla u\|_2)\nonumber\\
&\leq C(\|\nabla u\|{_2^2}+1),\nonumber
\end{align}
where we have used Lemma \ref{Ell2}$(2)$ and Lemma \ref{BH}.\\
Therefore, we conclude that
\begin{align}
\frac{d}{dt}\int\rho|u|^{(2+\alpha)}dx\leq C(\|\nabla u\|{_2^2}+1).\nonumber
\end{align}
which implies that
\begin{equation}\label{HE}
\sup\limits_{0\leq t\leq T}\int\rho|u|^{2+\alpha}dx\leq C.
\end{equation}
Next, we obtain from (\ref{3.45}), (\ref{3.41}) and (\ref{3.42}) that
\begin{align}\label{3.51}
\|\nabla u\|_4&\leq C(\|\di u\|_4+\|\omega\|_4)\leq C(\|\frac{F^2}{2\mu+\lambda(\rho)}\|{_2^\frac{1}{2}}+\|\omega\|_4+\|P\|_4+\|H\|{_8^2})\nonumber\\
&\leq C(Z^\frac{1}{2}\Phi{_T^\frac{\beta\varepsilon}{4}}(\Phi{_T^\frac{1}{4}}\varphi^\frac{1}{2}+\|\nabla H\|{_2^\frac{1}{2}}\varphi^\frac{1}{4})+1)\nonumber\\
&\leq C(\Phi{_T^{\frac{1+\beta\varepsilon}{4}}}(e+Z^2)^\frac{1}{2}(\frac{\varphi^2}{e+Z^2})^\frac{1}{4}+\Phi{_T^{\frac{\beta\varepsilon}{4}}}(e+Z^2)^\frac{5}{8}(\frac{\varphi^2}{e+Z^2})^\frac{1}{8}+1)\\
&\leq C\Phi{_T^{\frac{1+\beta\varepsilon+2\beta}{4}}}(e+\|\nabla u\|+\|\nabla H\|)(\frac{\varphi^2}{e+Z^2})^\frac{1}{4}\nonumber\\
&\quad+C\Phi{_T^{\frac{2\beta\varepsilon+5\beta}{8}}}(e+\|\nabla u\|+\|\nabla H\|)^\frac{5}{4}(\frac{\varphi^2}{e+Z^2})^\frac{1}{8}+C\nonumber
\end{align}
where in the last inequality, we have used the following simple fact:
\begin{equation}
Z^2\leq C\Phi{_T^\beta}(\|\nabla u\|{_2^2}+\|\nabla H\|{_2^2})
\end{equation}
Then, by the Brezis-Wainger inequality, we have that
\begin{align}
\|\nabla u\|_\infty&\leq C\|\nabla u\|_2\log^\frac{1}{2}(e+\|\nabla u\|_4)+C\|u\|_2+C\nonumber\\
&\leq C\|\nabla u\|_2\{\log^\frac{1}{2}(\Phi{_T^{\frac{1+\beta\varepsilon+2\beta}{4}}}(e+\|\nabla u\|+\|\nabla H\|)(\frac{\varphi^2}{e+Z^2})^\frac{1}{4})\nonumber\\
&\quad+\log^\frac{1}{2}(\Phi{_T^{\frac{2\beta\varepsilon+5\beta}{8}}}(e+\|\nabla u\|+\|\nabla H\|)^\frac{5}{4}(\frac{\varphi^2}{e+Z^2})^\frac{1}{8})\}+C\|\nabla u\|_2+C\\
&\leq C\Phi{_T^\sigma}(e+\|\nabla u\|_2+\|\nabla H\|_2)^{1+\sigma}\log^\frac{1}{2}(e+(\frac{\varphi^2}{e+Z^2})^\frac{1}{4})\nonumber
\end{align}
which combining with (\ref{HE}) implies that
\begin{align}
\|\rho u\|_p&\leq \|\rho u\|{_{2+\alpha}^{\frac{2+\alpha}{p}}}\|\rho u\|{_\infty^{1-\frac{2+\alpha}{p}}}\leq\Phi{_T^{1-\frac{1}{p}}}\|\rho^\frac{1}{2+\alpha}u\|{_{2+\alpha}^{\frac{2+\alpha}{p}}}\|u\|{_\infty^{1-\frac{2+\alpha}{p}}}\nonumber\\
&\leq C\Phi{_T^{1-\frac{1}{p}+\sigma}}(e+\|\nabla u\|+\|\nabla H\|_2)^{1-\frac{2+\alpha}{p}+\sigma}\log^{\frac{p-2+\alpha}{2p}}(e+(\frac{\varphi^2}{e+Z^2})^\frac{1}{4})
\end{align}
%Next, for any $p>2,q>1$ satifying  $\frac{1}{p}=\frac{\frac{2}{p}}{2+\alpha}+\frac{1-\frac{2}{p}}{q}$,\\
%that is, $q=(1+\frac{2}{\alpha})(p-2)\leq C\Phi{_T^\frac{\beta}{2}}$\\
%Using interpolation inequality and Lemma 2.3.3(2), one has
%\begin{align}
%\|\rho u\|_p&\leq C\|\rho u\|{_{2+\alpha}^\frac{2}{p}}\|\rho u\|{_q^{1-\frac{2}{p}}}\leq C(\|\rho^\frac{1}{2+\alpha}u\|_{2+\alpha}\Phi{_T^\frac{1+\alpha}{2+\alpha}})^\frac{2}{p}(\Phi_T\|u\|_q)^{1-\frac{2}{p}}\nonumber\\
%&\leq C\Phi{_T^{1-\frac{2}{p(2+\alpha)}}}\|u\|{_q^{1-\frac{2}{p}}}\leq C\Phi{_T^{1-\frac{2}{p(2+\alpha)}}}(q^\frac{1}{2}\|\nabla u\|_2+1)^{1-\frac{2}{p}}\nonumber\\
%&\leq C\Phi{_T^{1-\frac{2}{p(2+\alpha)}+\frac{\beta}{4}(1-\frac{2}{p})}}(\|\nabla u\|_2+1)^{1-\frac{2}{p}}\nonumber\\
%&\leq C\Phi^{1+\frac{\beta}{4}}(\|\nabla u\|_2+1)^{1-\frac{2}{p}}\nonumber
%\end{align}
 \end{proof}
  
\begin{lemma}\label{lem3.8}
There exist a positive constant $C$ such that
\begin{equation}\label{2.3.50}
0\leq\rho(t,x)\leq C,~~~~\forall(t,x)\in[0,T]\times\mathbb{T}^2.
\end{equation}
\end{lemma}
\begin{proof}
First, by the definition of $\xi, \eta$ from (\ref{3.8}) and (\ref{3.9}), and $\di H=0$, we have
\begin{equation}
u\cdot\nabla\xi-\eta=[u_i,R_iR_j](\rho u_j)-[H_i,R_iR_j](H_j).\nonumber
\end{equation}
It follows from (\ref{3.15}) that 
\begin{equation}\label{tran}
\frac{D}{Dt}(\Lambda(\rho)-\xi)+P+\frac{|H|^2}{2}+[u_i,R_iR_j](\rho u_j)-[H_i,R_iR_j](H_j)+\int F(t,x)dx=0.\nonumber
\end{equation}
Along the partial path defined by
\begin{equation}
\left\{\begin{array}{lr}\frac{d\vec{X}(\tau;t,x)}{d\tau}=u(\tau,\vec{X}(\tau;t,x))\\
\vec{X}(\tau;t,x)|_{\tau=t}=x,\\
\end{array}\nonumber\right.
\end{equation}
one has
\begin{align}
\frac{d}{d\tau}&(\Lambda(\rho)-\xi)(\tau;\vec{X}(\tau;t,x))+P(\rho)(\tau;\vec{X}(\tau;t,x))+\frac{|H|^2}{2}(\tau;\vec{X}(\tau;t,x))\nonumber\\
&=-[u_i,R_iR_j](\rho u_j)(\tau;\vec{X}(\tau;t,x))+[H_i,R_iR_j](H_j)(\tau;\vec{X}(\tau;t,x))-\int F(\tau,x)dx.\nonumber
\end{align}
Integrating over $[0,t]$, it holds that
\begin{align}\label{3.56}
2\mu&\log\frac{\rho(t,x)}{\rho_0(\vec{X}_0)}+\frac{1}{\beta}(\rho^\beta(t,x)-\rho{_0^\beta}(\vec{X}_0))-\xi(t,x)+\xi_0(\vec{X}_0)\\
&\leq -\int_{0}^{t}[u_i,R_iR_j](\rho u_j)(\tau)d\tau+\int_{0}^{t}[H_i,R_iR_j](H_j)(\tau)d\tau-\int_{0}^{t}\int F(\tau,x)dxd\tau.\nonumber
\end{align}
Now we estimate the commutators. On the one hand, for the commutator $G=[u_i,R_iR_j](\rho u_j)$, by the Brezis-Wainger inequality in Lemma \ref{BW} and commutator estimates in Lemma \ref{CE}, we have that
\begin{align}
\|G\|_\infty&\leq C\|\nabla G\|_2\log^\frac{1}{2}(e+\|\nabla G\|_3)+C\|G\|_2+C\nonumber\\
&\leq C\|\nabla u\|_q\|\rho u\|_p\log^\frac{1}{2}(e+\|\nabla u\|_4\|\rho u\|_{12})+C\|\nabla u\|_2\|\rho u\|_2+C\\
&\leq C\|\nabla u\|_q\|\rho u\|_p\log^\frac{1}{2}(e+\|\nabla u\|_4\|\rho u\|_{12})+C\Phi{_T^\frac{1}{2}}\|\nabla u\|_2+C
\end{align}
where $\frac{1}{p}+\frac{1}{q}=\frac{1}{2}$.
It follows from (\ref{3.51}) that, for some $\theta$ satisfying $\frac{1}{q}=\frac{\theta}{2}+\frac{1-\theta}{4}$,
\begin{align}\label{3.59}
\|\nabla u\|_q&\leq\|\nabla u\|{_2^\theta}\|\nabla u\|{_4^{1-\theta}}\nonumber\\
&\leq C\|\nabla u\|{_2^\theta}\{\Phi{_T^{\frac{1+\beta\varepsilon+2\beta}{4}}}(e+\|\nabla u\|+\|\nabla H\|)(\frac{\varphi^2}{e+Z^2})^\frac{1}{4}\nonumber\\
&\quad+\Phi{_T^{\frac{2\beta\varepsilon+5\beta}{8}}}(e+\|\nabla u\|+\|\nabla H\|)^\frac{5}{4}(\frac{\varphi^2}{e+Z^2})^\frac{1}{8}+1\}^{1-\theta}\\
&\leq C\Phi{_T^{(\frac{1+\beta\varepsilon+2\beta}{4})(1-\theta)}}(e+\|\nabla u\|+\|\nabla H\|)(1+(\frac{\varphi^2}{e+Z^2})^\frac{1}{4})^{1-\theta}\nonumber\\
&\quad+C\Phi{_T^{(\frac{2\beta\varepsilon+5\beta}{8})(1-\theta)}}(e+\|\nabla u\|+\|\nabla H\|)^{\frac{5-\theta}{4}}(1+(\frac{\varphi^2}{e+Z^2})^\frac{1}{8})^{1-\theta}\nonumber
\end{align}
Next, (\ref{3.29}) yields, for $s=1-\frac{2+\alpha}{p}$, that
\begin{equation}\label{3.60}
\|\rho u\|_p\leq C\Phi{_T^{1-\frac{1}{p}+\sigma}}(e+\|\nabla u\|+\|\nabla H\|_2)^{s+\sigma}\log^{\frac{s}{2}}(e+(\frac{\varphi^2}{e+Z^2})^\frac{1}{4})
\end{equation}
Then, we can obtain from (\ref{3.59}) and (\ref{3.60}) that
\begin{align}
&\|\nabla u\|_q\|\rho u\|_q\leq C\Phi{_T^{1-\frac{1}{p}+(\frac{1+\beta\varepsilon+2\beta}{4})(1-\theta)+\sigma}}(e+\|\nabla u\|+\|\nabla H\|)^{1+s+\sigma}(1+(\frac{\varphi^2}{e+Z^2}))^\frac{1-\theta+\sigma}{4}\nonumber\\
&\quad+ C\Phi{_T^{1-\frac{1}{p}+(\frac{2\beta\varepsilon+5\beta}{8})(1-\theta)+\sigma}}(e+\|\nabla u\|+\|\nabla H\|)^{\frac{5-\theta}{4}+s+\sigma}(1+(\frac{\varphi^2}{e+Z^2}))^\frac{1-\theta+\sigma}{8}
\end{align}
By the definition of $s$ and $\theta$, we have that
\begin{equation*}
s=1-(2+\alpha)\frac{1-\theta}{4}
\end{equation*}
which implies that when $0<\theta<1$, 
\begin{equation*}
2s<1+\theta
\end{equation*}
Taking $\sigma>0$ small enough such that 
\begin{equation}
2s+3\sigma\leq 1+\theta,\quad 4s+5\sigma\leq 2\theta+2
\end{equation}
Therefore, we can obtain that
\begin{equation}
\int_{0}^{T}\|G\|_\infty dt\leq C\int_{0}^{T}(1+\frac{\varphi^2}{e+Z^2})dt+C\Phi{_T^r}\int_{0}^{T}(e+\|\nabla u\|{_2^2}+\|\nabla H\|{_2^2})\nonumber
\end{equation}
where 
\begin{equation*}
r=\text{max}\{(1-\frac{1}{p}+(\frac{1+\beta\varepsilon+2\beta}{4})(1-\theta)+\sigma)\frac{4}{\theta+3-\sigma},(1-\frac{1}{p}+(\frac{2\beta\varepsilon+5\beta}{8})(1-\theta)+\sigma)\frac{8}{7+\theta-\sigma}\}
\end{equation*}
Taking $\theta$ close to $1$ such that $r\leq 1+\beta\varepsilon$ and using of (\ref{3.31}), we have that
\begin{equation}\label{3.64}
\int_{0}^{T}\|G\|_\infty dt\leq C\Phi{_T^{1+\beta\varepsilon}}.
\end{equation}
On the other hand, for any $p>4$, using of the Gagliardo-Nirenberg inequality and commutator estimates in Lemma \ref{CE}, it holds that
\begin{align}
\|[H_i,R_iR_j](H_j)\|_\infty&\leq C\|[H_i,R_iR_j](H_j)\|{_p^{1-\frac{4}{p}}}\|\nabla[H_i,R_iR_j](H_j)\|{_\frac{4p}{p+4}^\frac{4}{p}}\nonumber\\
&\leq C(\|H\|_{BMO}\|H\|_p)^{1-\frac{4}{p}}(\|\nabla H\|_4\|H\|_p)^\frac{4}{p}\nonumber\\
&\leq C\|\nabla H\|{_2^{1-\frac{4}{p}}}\|\nabla H\|{_4^\frac{4}{p}}\leq C\|\nabla H\|{_2^{1-\frac{2}{p}}}\|\nabla^2H\|{_2^\frac{2}{p}}\nonumber\\
&\leq C(e+Z^2(t))^\frac{1}{2}(\frac{\varphi^2(t)}{e+Z^2(t)})^\frac{1}{p}\leq \sigma Z^2+C(1+\frac{\varphi^2(t)}{e+Z^2(t)})\nonumber\\
&\leq\sigma\Phi{_T^\beta}(\|\nabla u\|{_2^2}+\|\nabla H\|{_2^2})+C(1+\frac{\varphi^2(t)}{e+Z^2(t)})\nonumber
\end{align}
which implies that
\begin{equation}\label{3.65}
\int_{0}^{T}\|[H_i,R_iR_j](H_j)\|_\infty dt\leq \sigma\Phi{_T^\beta}+C\Phi{_T^{1+\beta\varepsilon}}.
\end{equation}
It follows from Lemma \ref{BE} and Lemma \ref{BP} that
\begin{align}\label{3.66}
&\int_{0}^{T}\int F(t,x)dxdt\leq \int_{0}^{T}[\int(2\mu+\lambda(\rho))\di u(t,x)dx+\int P(t,x)dx+\int|H|^2(t,x)dx]dt\nonumber\\
&~~\leq \int_{0}^{T}\int(2\mu+\lambda(\rho))dxdt+
\int_{0}^{T}\int[(2\mu+\lambda(\rho))(\di u)^2+\rho^\gamma+|H|^2]dxdt\leq C. 
\end{align}
Also, by Lemma \ref{Ell1}, it holds that for suitable large but fixed $m>1$,
\begin{equation}
\|\xi\|_{2m}\leq Cm^\frac{1}{2}\|\rho\|{_m^\frac{1}{2}}\leq C,~~~~\|\nabla\xi\|_2\leq\|\rho u\|_2\leq C\Phi{_T^\frac{1}{2}}\nonumber
\end{equation}
By Brezis-Wainger inequality, it holds that
\begin{align}\label{3.67}
\|\xi\|_\infty&\leq C(\|\xi\|_{2m}+\|\nabla\xi\|_2)\log^\frac{1}{2}(e+\|\xi\|_{W^{1,2m}})\leq C\Phi{_T^\frac{1}{2}}\log^\frac{1}{2}(e+\|\nabla\xi\|_{2m})\nonumber\\
&\leq C\Phi{_T^\frac{1}{2}}\log^\frac{1}{2}(e+\|\nabla u\|_2+\|\nabla H\|_2)\leq C\Phi{_T^{1+\frac{\beta\varepsilon}{2}}}.
\end{align}
Therefore, plugging $(\ref{3.64})$-$(\ref{3.67})$ into $(\ref{3.56})$, we can obtain that 
\begin{equation}
\Phi{_T^\beta}\leq C\Phi{_T^{1+\beta\varepsilon}}.\nonumber
\end{equation}
If $\beta>1$, choosing $\varepsilon$ small enough, we have 
\begin{equation}
\sup\limits_{0\leq t\leq T}||\rho||_\infty(t)\leq C.\nonumber
\end{equation}
which also yields that
\begin{equation*}
\sup\limits_{0\leq t\leq T}\|\rho\|_\infty(t)\geq \delta_1>0
\end{equation*}
when $\inf\limits_{x\in\mathbb{T}^2}\rho_0\geq\delta>0$.
This complete the proof of lemma.
\end{proof}

\subsection{Higher order estimates}

Based on the uniform estimates and the upper bound of density obtained in the last section under the assumption $\inf\limits_{x\in\mathbb{T}^2}\rho_0\geq\delta>0$, we observe that, for any $2\leq p< +\infty$,
\begin{align}\label{3.68}
\sup\limits_{t\in[0,T]}(\|(u,H)\|_p+\|(\nabla u,\nabla H)\|_2+\|\rho\|_\infty)+\int_{0}^{T}(\|\sqrt{\rho}\dot{u}\|{_2^2}+\|\nabla^2H\|{_2^2})dt\leq C
\end{align}
Now, we can derive the uniform higher order estimates to guarantee the required regularity of classical solutions.
\begin{lemma}\label{lem3.9}
There exists a positive constant $C$ only depending on $(\rho_0,u_0,H_0)$, such that 
\begin{equation}\label{3.69}
\sup\limits_{0 \leq t\leq T}(\|\sqrt{\rho}\dot u\|{^2_2}+\|H_t\|{^2_2})+\int_{0}^{T}(\|\nabla \dot u\|{^2_2}+\|\nabla H_t\|{^2_2})dt\leq C.
\end{equation}
\end{lemma}
\begin{proof}
First, applying the operator $\dot u^j[\partial_t+\di(u\cdot)]$ to $(\ref{MHD}){_2^j}$, summing with respect to $j$, and integrating the resulting equation over $\mathbb{T}^2$, we have 
\begin{align}\label{3.70}
\frac{d}{dt}\int\rho|\dot u|^2dx=&-2\int\dot u^j[\partial_jP_t+\di(u\partial_jP)]dx+2\mu\int\dot u^j[\partial_t\Delta u^j+\di(u\Delta u^j)]dx\nonumber\\
&+2\int\dot u^j[\partial_{jt}((\mu+\lambda)\di u)+\di(u\partial_j((\mu+\lambda)\di u))]dx\nonumber\\
&+2\int\dot u^j[\partial_t((H\cdot\nabla )H^j)+\di(u(H\cdot\nabla)H^j)]dx\\
&-2\int\dot u^j[\partial_{jt}(\frac{H^2}{2})+\di(u\partial_j(\frac{H^2}{2}))]dx=:\sum_{i=1}^{5}N_i.\nonumber
\end{align}
Now we estimates the RHS of (\ref{3.70}) terms by terms. Due to the continuity equation $(\ref{MHD})_1$, we can obtain
\begin{align}\label{3.71}
N_1&=-2\int\dot u^j[\partial_jP_t+\partial_j\di(uP)-\di(P\partial_ju)]dx\nonumber\\
&=2\int\di\dot u(P_t+\di(Pu))-2\int P\partial_ju^k\partial_k\dot u^jdx\\
&=2(1-\gamma)\int P(\di u)(\di\dot u)dx-2\int P\partial_ju^k\partial_k\dot u^jdx\nonumber\\
&\leq C\|\nabla \dot u\|_2\|\nabla u\|_2\leq \frac{\mu}{8}\|\nabla \dot u\|{_2^2}+C,\nonumber
\end{align}
where we have used (\ref{3.68}) in the last two inequalities. Note that $u{_t^j}=\dot{u}^j-u\cdot\nabla u^j$, one has
\begin{align}\label{3.72}
N_2&=2\mu\int\dot{u}^j[\Delta\dot{u}^j+\partial_i(-\partial_iu\cdot\nabla u^j+\di u\partial_iu^j)-\di(\partial_iu\partial_iu^j)]\nonumber\\
&=-\mu\int(|\nabla\dot u|^2+2\partial_i\dot u^j\partial_ku^k\partial_iu^j-2\partial_i\dot u^j\partial_ku^k\partial_iu^j+2\partial_k\dot u^j\partial_iu^k\partial_iu^j)\\
&\leq -\frac{7\mu}{8}\|\nabla\dot u\|{_2^2}+C\|\nabla u\|{_4^4}.\nonumber
\end{align}
Similarly,
\begin{align}\label{3.73}
N_3&=2\int\dot u^j[\partial_{jt}((\mu+\lambda)\di u)+\partial_j\di(u((\mu+\lambda)\di u))-\di(\partial_ju((\mu+\lambda)\di u))]dx\nonumber\\
&=-2\int\di\dot u[((\mu+\lambda)\di u)_t+\di(u((\mu+\lambda)\di u))]+\int(\mu+\lambda)\di u\partial_i\dot u^j\partial_j u^idx\nonumber\\
&=-2\int(\frac{D}{Dt}\di u+\partial_iu^j\partial_ju^i)[(\mu+\lambda)\frac{D}{Dt}\di u+2(\mu+(1-\beta)\lambda)(\di u)^2]dx\\
&~~~~+\int(\mu+\lambda)\di u\partial_i\dot u^j\partial_j u^idx\nonumber\\
&\leq-\frac{\mu}{2}\|\frac{D}{Dt}\di u\|{_2^2}+\frac{\mu}{8}\|\nabla\dot u\|{_2^2}+C\|\nabla u\|{_4^4}+C\nonumber
\end{align}
It follows from $\di H=0$ and (\ref{3.68}) that
\begin{align}\label{3.74}
N_4+N_5&=2\int\dot u^j[\partial_{it}(H^iH^j)+\di(uH^k\partial_kH^j)]+\di\dot uHH_t+\partial_k\dot u^ju^kH^i\partial_jH^i)dx\nonumber\\
&=2\int(-\partial_i\dot u^j(H^iH^j)_tdx-\partial_i\dot u^ju^iH^k\partial_kH^j+\di\dot uHH_t+\partial_k\dot u^ju^kH^i\partial_jH^i)dx\nonumber\\
&\leq C\|H\|_4\|H_t\|_4\|\nabla\dot u\|_2+C\|u\|_8\|H\|_8\|\nabla H\|_4\|\nabla\dot u\|_2\\
&\leq C\|H_t\|{_2^\frac{1}{2}}\|\nabla H_t\|{_2^\frac{1}{2}}\|\nabla\dot u\|_2+C\|\nabla H\|{_2^\frac{1}{2}}\|\nabla^2H\|{_2^\frac{1}{2}}\|\nabla\dot u\|_2\nonumber\\
&\leq\frac{\mu}{4}\|\nabla\dot u\|{_2^2}+\frac{\nu}{8}\|\nabla H_t\|{_2^2}+C\|H_t\|{_2^2}+C\|\nabla^2H\|{_2^2}+C.\nonumber
\end{align}
Next, applying $\partial_t$ to $(\ref{MHD})_3$, multiplying the resulting equation by $H_t$ and integrating over $\mathbb T^2$, we obtain from the integration by parts that
\begin{align}\label{3.75}
\frac{d}{dt}&\int|H_t|^2+2\nu\int|\nabla H_t|^2dx\nonumber\\
&=-2\int u_t\cdot\nabla H\cdot H_tdx+2\int H_t\cdot\nabla u\cdot H_tdx+2\int H\cdot\nabla u_t\cdot H_tdx\nonumber\\
&~~~~-\int|H_t|^2\di udx-2\int H\cdot\di u_tH_tdx\\
&=2\int(H\cdot\nabla\dot u-\dot u\cdot\nabla H-H\di\dot{u})\cdot H_tdx+2\int(H^i\partial_iH{_t^j}-H^i\partial_jH{_t^i})(u\cdot\nabla u^j)dx\nonumber\\
&~~~~+\int(2H_t\cdot\nabla u-H_t\di u-2u\cdot\nabla H_t)H_tdx=:\sum_{i=1}^{3}R_i,\nonumber
\end{align}
Now we estimates  the terms on the RHS of (\ref{3.75}).
\begin{align}\label{3.76}
|R_1|&\leq C\|H_t\|_4\|H\|_4\|\nabla\dot u\|_2+C\|\dot u\|_4\|\nabla H\|_2\|H_t\|_4\nonumber\\
&\leq C(\|\sqrt{\rho}\dot u\|_2+\|\nabla\dot{u}\|_2)\|H_t\|{_2^\frac{1}{2}}\|\nabla H_t\|{_2^\frac{1}{2}}\\
&\leq\frac{\mu}{8}\|\nabla\dot u\|{_2^2}+\frac{\nu}{8}\|\nabla H_t\|{_2^2}+C\|H_t\|{_2^2}+C\|\sqrt{\rho}\dot u\|{_2^2}\nonumber
\end{align}
where we have used Poincar\'{e} type inequality Lemma \ref{SP}(2).
\begin{align}\label{3.77}
|R_2|\leq C\|H\|_8\|u\|_8\|\nabla H_t\|_2\|\nabla u\|_4\leq \frac{\nu}{8}\|\nabla H_t\|{_2^2}+C\|\nabla u\|{_4^4}+C,
\end{align}
\begin{align}\label{3.78}
|R_3|&\leq C\|H_t\|{_4^2}\|\nabla u\|_2+C\|H_t\|_4\|u\|_4\|\nabla H_t\|_2\\
&\leq C\|H_t\|_2\|\nabla H_t\|_2\|\nabla u\|_2+C\|H_t\|{_2^\frac{1}{2}}\|\nabla H_t\|{_2^\frac{3}{2}}\leq\frac{\nu}{8}\|\nabla H_t\|{_2^2}+C\|H_t\|{_2^2},\nonumber
\end{align}
Thus, (\ref{3.69})-(\ref{3.78}) show that
\begin{align}\label{3.79}
\frac{d}{dt}&\int(\rho|\dot u|^2+|H_t|^2)dx+\frac{\mu}{4}\|\nabla\dot u\|{_2^2}+\frac{\nu}{2}\|\nabla H_t\|{_2^2}\nonumber\\
&\leq C(\|\sqrt{\rho}\dot u\|{_2^2}+\|H_t\|{_2^2})+C\|\nabla u\|{_4^4}+C.
\end{align}
Note that
\begin{align}\label{3.80}
\|\nabla u\|{_4^4}&\leq C(\|\omega\|{_4^4}+\|\di u\|{_4^4})\leq C(\|F\|{_4^4}+\|\omega\|{_4^4}+1)\nonumber\\
&\leq C(\|F\|{_2^2}\|\nabla F\|{_2^2}+\|\omega\|{_2^2}\|\nabla\omega\|{_2^2}+1)\leq C(\|\sqrt{\rho}\dot u\|{_2^2}+1).
\end{align}
It follows from (\ref{3.79})-(\ref{3.80}) that
\begin{align}
\frac{d}{dt}(\|\sqrt{\rho}\dot u\|{_2^2}+\|H_t\|{_2^2})+\|\nabla\dot u\|{_2^2}+\|\nabla H_t\|{_2^2}\leq C(\|\sqrt{\rho}\dot u\|{_2^2}+\|H_t\|{_2^2})+C.
\end{align}
Note that the compatibility condition shows that $\sqrt{\rho_0}\dot{u}_0=g\in L^2$, applying the Gronwall's inequality, we can complete the proof of this lemma.
\end{proof}
\begin{lemma}\label{lem3.10}
It holds for any $2\leq p<\infty$ that
\begin{equation}\label{3.82}
\sup\limits_{0\leq t\leq T}\|(\nabla\rho,\nabla P)\|_p+\int_{0}^{T}\|\nabla u\|{_\infty^2}dt\leq C.
\end{equation}
\end{lemma}
\begin{proof}
First, it follows from the interpolation inequality, (\ref{3.68}), (\ref{3.69}) and Lemma \ref{SP}(2), one has
\begin{equation}
\|\rho\dot u\|_p\leq C\|\rho\dot u\|{_2^\frac{2(p-1)}{p^2-2}}\|\rho\dot u\|{_{p^2}^\frac{p(p-2)}{p^2-2}}\leq C\|\dot u\|{_{H^1}^\frac{p(p-2)}{p^2-2}}\leq C\|\nabla\dot u\|_2+C,
\end{equation}
and
\begin{equation}
\||H||\nabla H|\|_p\leq\|H\|_{\frac{p^2}{p-1}}\|\nabla H\|_{p^2}\leq C\|\nabla H\|_{H^1}\leq C\|\nabla^2H\|_2+C,
\end{equation}
which imply that
\begin{align}\label{3.85}
\|\di u\|_\infty+\|\omega\|_\infty&\leq C(\|F\|_\infty+\|\omega\|_\infty+\|H\|{_\infty^2}+1)\nonumber\\
&\leq C(\|\nabla F\|{_4^\frac{2}{3}}+\|\nabla\omega\|{_4^\frac{2}{3}}+\|\nabla H\|{_4^\frac{4}{3}}+1)\nonumber
\\
&\leq C(\|\rho\dot u\|{_4^\frac{2}{3}}+\|\nabla H\|{_2^\frac{2}{3}}\|\nabla^2H\|{_2^\frac{2}{3}}+1)\\
&\leq C(\|\nabla\dot u\|{_2^\frac{2}{3}}+\|\nabla^2H\|{_2^\frac{2}{3}}+1),\nonumber
\end{align}
where we have used the Gagliardo-Nirenberg inequality. Next, applying $\nabla$ to $(\ref{MHD})_1$, and multiplying the resulting equation by the $p|\nabla\rho|^{p-2}\nabla\rho$, we have
\begin{align}
\frac{d}{dt}\int|\nabla\rho|^pdx=&-(p-1)\int|\nabla\rho|^p\di udx-p\int|\nabla\rho|^{p-2}\nabla\rho\cdot\nabla u\cdot\nabla\rho dx\nonumber\\&-p\int\rho|\nabla\rho|^{p-2}\nabla\rho\cdot\nabla(\di u)dx.\nonumber
\end{align}
It yields that
\begin{equation}\label{3.86}
\frac{d}{dt}\|\nabla\rho\|_p\leq C(\|\nabla u\|_\infty\|\nabla\rho\|_p+\|\nabla^2u\|_p).
\end{equation}
We deduce from the standard $L^p$-estimates for elliptic system that
\begin{align}\label{3.87}
\|\nabla^2u\|_p&\leq C(\|\nabla\di u\|_p+\|\nabla\omega\|_p)\nonumber\\
&\leq C(\|\nabla((2\mu+\lambda)\di u)\|_p+\|\di u\|_\infty\|\nabla\rho\|_p+\|\nabla\omega\|_p)\nonumber\\
&\leq C(\|\nabla F\|_p+\|\nabla\omega\|_p+\|\nabla P\|_p+\||H||\nabla H|\|_p+\|\di u\|_\infty\|\nabla\rho\|_p)\\
&\leq C(\|\rho\dot u\|_p+\|\nabla^2H\|_2+1)+C(\|\nabla\dot u\|_2+\|\nabla^2H\|_2+1)\|\nabla\rho\|_p\nonumber\\
&\leq C(\|\nabla\dot u\|_2+\|\nabla^2H\|_2+1)(e+\|\nabla\rho\|_p)\nonumber.
\end{align}
By the Beale-Kato-Majda type inequality, it follows from (\ref{3.85}), (\ref{3.87}) that
\begin{align}\label{3.88}
\|\nabla u\|_\infty&\leq C(\|\di u\|_\infty+\|\omega\|_\infty)\log(e+\|\nabla^2u\|_p)+C\|\nabla u\|_2+C\\
&\leq C(\|\nabla\dot u\|_2+\|\nabla^2H\|_2+1)\log(e+\|\nabla\rho\|_p)+C(\|\nabla\dot u\|_2+\|\nabla^2H\|_2+1)\nonumber
\end{align}
Combining (\ref{3.86}), (\ref{3.88}) with (\ref{3.85}), one has
\begin{align}
\frac{d}{dt}&\|\nabla\rho\|_p\leq C(\|\nabla\dot u\|_2+\|\nabla^2H\|_2+1)\log(e+\|\nabla\rho\|_p)\|\nabla\rho\|_p\nonumber\\&+C(\|\nabla\dot u\|_2+\|\nabla^2H\|_2+1)\|\nabla\rho\|_p+C(\|\nabla\dot u\|_2+1).
\end{align}
By the Gronwall's inequality, we obtain
\begin{equation}
\sup\limits_{0\leq t\leq T}||\nabla\rho||_p\leq C,\nonumber
\end{equation}
which combining with (\ref{3.87}) gives
\begin{equation}
\int_{0}^{T}\|\nabla u\|{_\infty^2}dt\leq C.\nonumber
\end{equation}
\end{proof}
\begin{lemma}\label{lem3.11}
It holds for any $2\leq p<\infty$ that
\begin{align}\label{2.4.22}
\sup\limits_{0\leq t\leq T}&(\|\sqrt{\rho}u_t\|{_2^2}+\|(\rho_t,P_t,\lambda_t)\|_{H^1}+\|(\rho,u,H)\|_{H^2})\nonumber\\
&+\int _{0}^{T}(\|(u_t,H_t)\|{_{H^1}^2}+\|(u,H)\|{_{H^3}^2}+\|(\rho_{tt},P_{tt},\lambda_{tt})\|{_2^2})dt\leq C.
\end{align}
\end{lemma}
\begin{proof}
First, by the standard $L^2$-estimates for the elliptic system $(\ref{MHD})_2$, one has
\begin{align}\label{2.4.23}
\|u\|_{H^2}&\leq C(\|u\|_2+\|\rho\dot{u}\|_2+\|\nabla P\|_2+\||H||\nabla H|\|_2)\nonumber\\
&\leq C(1+\|\nabla H\|_4)\leq C(1+\|\nabla^2H\|{_2^\frac{1}{2}})\leq\frac{1}{4}\|H\|_{H^2}+C,
\end{align}
where we have used (\ref{3.68}) and Lemma \ref{lem3.10} in the last two inequalities. Similarly, the standard $L^2$-estimates for the elliptic system $(\ref{MHD})_3$ gives that
\begin{align}\label{2.4.24}
\|H\|_{H^2}&\leq C(\|H\|_2+\|H_t\|_2+\||H||\nabla u|\|_2+\||u||\nabla H|\|_2)\nonumber\\
&\leq C(1+\|\nabla u\|_4+\|\nabla H\|_4)\leq \frac{1}{2}\|u\|_{H^2}+\frac{1}{4}\|H\|_{H^2}+C.
\end{align}
Combining (\ref{2.4.23}) with (\ref{2.4.24}), we have 
\begin{equation}
\|u\|_{H^2}+\|H\|_{H^2}\leq C.\nonumber
\end{equation}
It follows from the Sobolev embedding theorem that
\begin{equation}\label{2.4.25}
\sup\limits_{0\leq t\leq T}(\|u\|_\infty+\|H\|_\infty)\leq C,
\end{equation}
and
\begin{equation}
\sup\limits_{0\leq t\leq T}\|(\nabla u,\nabla H)\|_p\leq C.
\end{equation}
Note that
\begin{equation}
\|\sqrt{\rho}u_t\|{_2^2}\leq\|\sqrt{\rho}\dot{u}\|{_2^2}+\|\sqrt{\rho}u\cdot\nabla u\|{_2^2}\leq\|\sqrt{\rho}\dot{u}\|{_2^2}+\|\nabla u\|{_2^2},\nonumber
\end{equation}
and
\begin{equation}
\|(u_t,H_t)\|_{H^1}\leq\|\dot u\|_{H^1}+\|u\cdot\nabla u\|_{H^1}+\|H_t\|_2+\|\nabla H_t\|_2,\nonumber
\end{equation}
we can obtain
\begin{equation}\label{2.4.27}
\sup\limits_{0\leq t\leq T}\|\sqrt{\rho}u_t\|{_2^2}+\int_{0}^{T}\|(u_t,H_t)\|{_{H^1}^2}\leq C.
\end{equation}
Next, applying $\nabla^2$ to the continuity equation $(\ref{MHD})_1$, multiplying the resulting equation by $\nabla^2\rho$, and integrating over $\mathbb T^2$, one has that      
\begin{align}\label{2.4.28}
\frac{d}{dt}\|\nabla^2\rho\|{_2^2}&\leq C(\|\nabla u\|_\infty\|\nabla^2\rho\|{_2^2}+\|\nabla\rho\|_4\|\nabla^2\rho\|_2\|\nabla^2u\|_4+\|\rho\|_\infty\|\nabla^2\rho\|_2\|\nabla^3u\|_2)\nonumber\\
&\leq C[(\|\nabla u\|_\infty+1)\|\nabla^2\rho\|{_2^2}+\|\nabla^3u\|{_2^2}+1].
\end{align}
Similarly, 
\begin{equation}\label{2.4.29}
\frac{d}{dt}\|\nabla^2P\|{_2^2}\leq C[(\|\nabla u\|_\infty+1)\|\nabla^2P\|{_2^2}+\|\nabla^3u\|{_2^2}+1].
\end{equation}
Note that the standard elliptic estimates give that
\begin{align}\label{2.4.30}
\|\nabla^3u\|_2&\leq C(\|\nabla^2\di u\|_2+\|\nabla^2\omega\|_2)\nonumber\\
&\leq C(\|\nabla^2F\|_2+\|\nabla^2(\frac{H^2}{2})\|_2+\|\nabla\rho\cdot\nabla\di u\|_2+\|\di u\nabla^2\rho\|_2+\|\nabla^2\omega\|_2)\nonumber\\
&\leq C(\|\nabla(\rho\dot{u})\|_2+\|\nabla H\|{_4^2}+\||H||\nabla^2H|\|_2\nonumber\\
&~~~~~~+\|\nabla\rho\|_4\|\nabla^2u\|{_2^\frac{1}{2}}\|\nabla^3u\|{_2^\frac{1}{2}}+\|\di u\|_\infty\|\nabla^2\rho\|_2)\\
&\leq C(\|\nabla\rho\|_4\|\dot{u}\|_4+\|\rho\|_\infty\|\nabla\dot{u}\|_2+1+\|\nabla^3u\|{_2^\frac{1}{2}}+\|\nabla u\|_\infty\|\nabla^2\rho\|_2)\nonumber\\
&\leq \frac{1}{2}\|\nabla^3u\|_2+C(\|\nabla u\|_\infty+1)\|\nabla^2\rho\|_2+\|\nabla\dot{u}\|_2+1.\nonumber
\end{align}
Similarly,
\begin{align}\label{2.4.31}
\|\nabla^3H\|_2&\leq C(\|\nabla H_t\|_2+\|\nabla u\|_4\|\nabla H\|_4+\|u\|-\infty\|\nabla^2H\|_2+\|H\|_\infty\|\nabla^2u\|_2)\nonumber\\
&\leq C(\|\nabla H_t\|_2+1).
\end{align}
Substituting (\ref{2.4.30}) into (\ref{2.4.28}) and (\ref{2.4.29}), we can obtain
\begin{equation}
\frac{d}{dt}\|(\nabla^2\rho,\nabla^2P)\|{_2^2}\leq C[(\|\nabla u\|{_\infty^2}+1)\|(\nabla^2\rho,\nabla^2P)\|{_2^2}+\|\nabla\dot{u}\|{_2^2}+1].
\end{equation}
Applying the Gronwall's inequality, we have
\begin{equation}
\sup\limits_{0\leq t\leq T}\|(\nabla^2\rho,\nabla^2P)\|{_2^2}\leq C,
\end{equation}
which combing with (\ref{2.4.30}) and (\ref{2.4.31}) implies that
\begin{equation}
\int_{0}^{T}\|(u,H)\|{_{H^3}^2}dt\leq C.
\end{equation}
Finally, due to the continuity equation $(\ref{MHD})_1$, one can get
\begin{equation}\label{2.4.33}
\rho_t=-u\cdot\nabla\rho-\rho\di u,~~P_t=-u\cdot\nabla P-\gamma P\di u,~~\lambda_t=-u\cdot\nabla\lambda-\beta\lambda\di u,
\end{equation}
which gives
\begin{equation}\label{2.4.34}
\|(\rho_t,P_t,\lambda_t)\|_{H^1}\leq C(\|u\|_\infty\|\nabla\rho\|_{H^1}+\|\rho\|_\infty\|\nabla u\|_{H^1})\leq C.
\end{equation}
Applying the operator $\partial_t$ to (\ref{2.4.33}), it follows from (\ref{2.4.25}), (\ref{2.4.27}), (\ref{2.4.34}) and Lemma 2.4.2 that
\begin{align}\label{2.4.35}
\int_{0}^{T}\|(\rho_{tt},P_{tt},\lambda_{tt})\|{_2^2}dx&\leq C\int_{0}^{T}(\|u_t\|{_4^2}+\|u\|{_\infty^2}\|\nabla\rho_t\|{_2^2}+\|\rho_t\|{_4^2}\|\nabla u\|{_4^2}+\|\nabla u_t\|{_2^2})dx\nonumber\\
&\leq C\int_{0}^{T}(\|u_t\|{_{H^1}^2}+1)dx\leq C.
\end{align}
Thus, the proof of this lemma is completed.
\end{proof}
\begin{lemma}\label{3.12}
It holds for any $2<q<+\infty$ that
\begin{align}
\sup\limits_{0\leq t\leq T}&[t(\|(u_t,H_t)\|{_{H^1}^2}+\|(\rho_{tt},P_{tt},\lambda_{tt})\|{_2^2}+\|(u,H)\|{_{H^3}^2})+\|(\rho,P)\|_{W^{2,q}}]\nonumber\\
&+\int_{0}^{T}t(\|\sqrt{\rho}u_{tt}\|{_2^2}+\|H_{tt}\|{_2^2}+\|(u_t,H_t)\|{_{H^2}^2}+\|(u,H)\|{_{H^4}^2})\leq C.
\end{align}
\end{lemma}
\begin{proof}
First, applying the operator $u_{tt}\partial_t$ to $(\ref{MHD})_2$, and integrating with respect to $x$ over $\mathbb T^2$ yields that
\begin{align}\label{2.4.37}
&\|\sqrt{\rho}u_{tt}\|{_2^2}+\frac{1}{2}\frac{d}{dt}\int(\mu|\nabla u_t|^2+(\mu+\lambda(\rho))|\di u_t|^2)dx=\frac{1}{2}\int\lambda_t|\di u_t|^2dx\nonumber\\
&-\int(\nabla P_t+\rho_tu_t+\rho_tu\cdot\nabla u+\rho u\cdot\nabla u_t+\rho u_t\cdot u_t)\cdot u_{tt}dx+\int\nabla(\lambda_t\di u)\cdot u_{tt}dx\nonumber\\
&+\int[(H_t\cdot\nabla)H+(H\cdot\nabla)H_t]\cdot u_{tt}dx-\int\nabla(H\cdot H_t)\cdot u_{tt}dx.
\end{align}
Note that
\begin{align}
\int&\nabla(\lambda_t\di u)\cdot u_{tt}dx=-\int\lambda_t\di u\di u_{tt}dx\nonumber\\
&=-\frac{d}{dt}\int\lambda_t\di u\di u_tdx+\int(\lambda_t|\di u_t|^2+\lambda_{tt}\di u\di u_t)dx,\nonumber
\end{align}
and
\begin{align}
-\int&\nabla(H\cdot H_t)\cdot u_{tt}dx=\int H\cdot H_t\di u_{tt}dx\nonumber\\
&=\frac{d}{dt}\int H\cdot H_t\di u_t-\int(|H_t|^2\di u_t+H\cdot H_{tt}\di u_t)dx.\nonumber
\end{align}
Substituting the above identity into $(\ref{2.4.37})$ yields that
\begin{align}\label{2.4.38}
&\|\sqrt{\rho}u_{tt}\|{_2^2}+\frac{1}{2}\frac{d}{dt}\int(\mu|\nabla u_t|^2+(\mu+\lambda(\rho))|\di u_t|^2+\lambda_t\di u\di u_t-H\cdot H_t\di u_t)dx\nonumber\\
&=\frac{3}{2}\int\lambda_t|\di u_t|^2dx+\int\lambda_{tt}\di u\di u_tdx\nonumber\\
&~~~~-\int(\nabla P_t+\rho_tu_t+\rho_tu\cdot\nabla u+\rho u\cdot\nabla u_t+\rho u_t\cdot u_t)\cdot u_{tt}dx\\
&~~~~+\int[(H_t\cdot\nabla)H+(H\cdot\nabla)H_t]\cdot u_{tt}dx-\int(|H_t|^2\di u_t+H\cdot H_{tt})\di u_tdx.\nonumber
\end{align}
Now we estimate the terms on the RHS of $(\ref{2.4.38})$. It follows from $(\ref{2.4.33})$ that
\begin{align}\label{2.4.39}
|\frac{3}{2}\int\lambda_t|\di u_t|^2dx|&=|\frac{3}{2}\int u\cdot\nabla\lambda|\di u_t|^2dx-\frac{3\beta}{2}\int\lambda\di u|\di u_t|^2dx|\nonumber\\
&=|3\int\lambda\di u_tu\cdot\nabla(\di u_t)dx+\frac{3(1-\beta)}{2}\int\lambda\di u|\di u_t|^2dx|\nonumber\\
&\leq C\|\di u_t\|_2\|\nabla(\di u_t)\|_2+C\|\nabla u\|_\infty\|\di u_t\|{_2^2}.
\end{align}
Note that the standard $L^2$-estimates for elliptic system 
\begin{align*}
\mu\Delta u_t+\nabla((\mu+\lambda)\di u_t)=&\rho u_{tt}+\rho_tu_t+(\rho u\cdot\nabla u)_t+\nabla P_t+\nabla(\lambda_t\di u)\nonumber\\
&-(H_t\cdot\nabla)H-(H\cdot\nabla)H_t+\nabla(HH_t)
\end{align*}
show that
\begin{align}\label{2.4.40}
\|\nabla^2u_t\|_2\leq C(\|\sqrt{\rho}u_{tt}\|_2+\|(u_t,H_t)\|_4+\|(\nabla u_t,\nabla H_t)\|_2+\|\di u\|_\infty+\|\nabla^3 u\|+1),
\end{align}
where we have used Lemma \ref{lem3.11} above. Substituting $(\ref{2.4.40})$ into $(\ref{2.4.39})$ yields that
\begin{align}
|\frac{3}{2}\int\lambda_t|\di u_t|^2dx|\leq&\frac{1}{8}\|\sqrt{\rho}u_{tt}\|{_2^2}+C(\|\nabla u\|_\infty+1)\|\nabla u_t\|{_2^2}+C\|\nabla H_t\|{_2^2}\nonumber\\
&+C(\|u_t\|{_4^2}+\|H_t\|{_4^2}+\|\nabla u\|{_\infty^2}+\|\nabla^3u\|{_2^2}+1).
\end{align}
\begin{align}
-\int\nabla P_t\cdot u_{tt}dx&=\int P_t\di u_{tt}dx=\frac{d}{dt}\int P_t\di u_tdx-\int P_{tt}\di u_tdx\nonumber\\
&\leq\frac{d}{dt}\int P_t\di u_tdx+\|P_{tt}\|{_2^2}+\|\di u_t\|{_2^2}.
\end{align}
\begin{align}
-\int\rho_tu_t\cdot u_{tt}dx=\int\rho_t(\frac{|u_t|^2}{2})_tdx=\frac{d}{dt}\int\rho_t\frac{|u_t|^2}{2}dx-\int\rho_{tt}\frac{|u_t|^2}{2}dx,
\end{align}
while
\begin{align}
-\int&\rho_{tt}\frac{|u_t|^2}{2}dx=\int\di(\rho u)_t\frac{|u_t|^2}{2}dx=-\int(\rho u)_t\cdot \nabla u_t\cdot u_tdx\nonumber\\
&\leq \|\sqrt{\rho}\|_\infty\|\sqrt{\rho}u_t\|_2\|u_t\|_4\|\nabla u_t\|_4+\|u\|_\infty\|\rho_t\|_4\|u_t\|_4\|\nabla u_t\|_2\nonumber\\
&\leq C(\|u_t\|_4\|\nabla^2u_t\|_2+\|u_t\|_4\|\nabla u_t\|_2)\\
&\leq C\|u_t\|_4(\|\sqrt{\rho}u_{tt}\|_2+\|u_t\|_4+\|\nabla u_t\|_2+\|\nabla u\|_\infty+\|\nabla^3 u\|_2)\nonumber\\
&\leq\frac{1}{8}\|\sqrt{\rho}u_{tt}\|{_2^2}+C(\|u_t\|{_4^2}+\|\nabla u_t\|{_2^2}+\|\nabla u\|{_\infty^2}+\|\nabla^3 u\|{_2^2}+1).\nonumber
\end{align}
Moreover,it follows from Lemma \ref{lem3.11} that
\begin{align}
&-\int\rho_tu\cdot\nabla u\cdot u_{tt}dx\nonumber\\
&=-\frac{d}{dt}\int\rho_tu\cdot\nabla u\cdot u_tdx+\int\rho_{tt}u\cdot\nabla u\cdot u_tdx+\int\rho_tu_t\cdot\nabla u\cdot u_tdx+\int\rho_tu\cdot\nabla u_t\cdot u_tdx\nonumber\\
&\leq-\frac{d}{dt}\int\rho_tu\cdot\nabla u\cdot u_tdx+C(\|\rho_{tt}\|_2\|u_t\|_4+\|u_t\|{_4^2}+\|\nabla u_t\|_2\|u_t\|_4)\\
&\leq-\frac{d}{dt}\int\rho_tu\cdot\nabla u\cdot u_tdx+C(\|\rho_{tt}\|{_2^2}+\|u_t\|{_4^2}+\|\nabla u_t\|{_2^2}),\nonumber
\end{align}
\begin{align}
-\int\rho u\cdot\nabla u_t\cdot u_{tt}dx\leq\|\sqrt{\rho}u_{tt}\|_2\|\sqrt{\rho}u\|_\infty\|\nabla u_t\|_2\leq\frac{1}{8}\|\sqrt{\rho}u_{tt}\|{_2^2}+C\|\nabla u_t\|{_2^2},
\end{align}
\begin{align}
-\int\rho u_t\cdot\nabla u\cdot u_{tt}dx\leq\|\sqrt{\rho}u_{tt}\|_2\|\nabla u\|_4\|u_t\|_4\leq\frac{1}{8}\|\sqrt{\rho}u_{tt}\|{_2^2}+C\|u_t\|{_4^2},
\end{align}
and
\begin{align}
-\int\lambda_{tt}\di u\di u_tdx\leq\|\lambda_{tt}\|_2\|\nabla u\|_\infty\|\di u_t\|_2\leq\frac{1}{2}(\|\lambda_{tt}\|{_2^2}+\|\nabla u\|{_\infty^2}\|\di u_t\|{_2^2}).
\end{align}
At the same time, it holds that
\begin{align}
\int&(H_t\cdot\nabla)H\cdot u_{tt}dx=\frac{d}{dt}\int(H_t\cdot\nabla)H\cdot u_t-\int(H_{tt}\cdot\nabla)H\cdot u_tdx+\int(H_t\cdot\nabla)u_t\cdot H_tdx\nonumber\\
&\leq\frac{d}{dt}\int(H_t\cdot\nabla)H\cdot u_t+\|H_{tt}\|_2\|\nabla H\|_4\|u_t\|_4+\|H_t\|_4\|\nabla u_t\|_2\|H_t\|_4\nonumber\\
&\leq\frac{d}{dt}\int(H_t\cdot\nabla)H\cdot u_t+\frac{1}{8}\|H_{tt}\|{_2^2}+C\|(\nabla u_t,\nabla H_t)\|{_2^2}+C\|u_t\|{_4^2},
\end{align}
where we have used $\di H=0$ and Lemma \ref{lem3.11}. Similarly, we have
\begin{align}
\int&(H\cdot\nabla)H_t\cdot u_{tt}dx\nonumber\\
&=\frac{d}{dt}\int(H\cdot\nabla)H_t\cdot u_tdx+\int(H_t\cdot\nabla)u_t\cdot H_tdx+\int(H\cdot\nabla)u_t\cdot H_{tt}dx\nonumber\\
&\leq\frac{d}{dt}\int(H\cdot\nabla)H_t\cdot u_tdx+\|H_t\|{_4^2}\|\nabla u\|_2+\|H\|_\infty\|\nabla u_t\|_2\|H_{tt}\|_2\\
&\leq\frac{d}{dt}\int(H\cdot\nabla)H_t\cdot u_tdx+\frac{1}{8}\|H_{tt}\|{_2^2}+C\|\nabla u_t\|{_2^2}+C\|\nabla H_t\|{_2^2},\nonumber
\end{align}
and
\begin{align}
-\int(|H_t|^2\di u_t+H\cdot H_{tt}\di u_t)dx&\leq\|H_t\|{_4^2}\|\di u_t\|_2+\|H\|_\infty\|H_tt\|_2\|\di u_t\|_2\nonumber\\
&\leq \frac{1}{8}\|H_{tt}\|{_2^2}+C\|\nabla u_t\|{_2^2}+C\|\nabla H_t\|{_2^2}.
\end{align}
Note that the standard $L^2$-estimates for elliptic system gives that
\begin{align}
\|H_{tt}\|_2&\leq-\nu\frac{d}{dt}\int|\nabla H_t|^2+\|H_t\|_4\|\nabla u\|_4\|H_{tt}\|_2+\|H\|_4\|\nabla u_t\|_2\|H_{tt}\|_2\nonumber\\
&~~~~+\|u_t\|_4\|\nabla H\|_4\|H_{tt}\|_2+\|u\|_\infty\|\nabla H_t\|_2\|H_{tt}\|_2+\|H\|_\infty\|\nabla u_t\|_2\|H_{tt}\|_2\nonumber\\
&\leq-\nu\frac{d}{dt}\int|\nabla H_t|^2+\frac{1}{8}\|H_{tt}\|{_2^2}+C\|\nabla H_t\|{_2^2}+C\|\nabla u_t\|{_2^2}.
\end{align}
Combining all the above estimates, one has
\begin{align}
&\frac{1}{2}\|\sqrt{\rho}u_{tt}\|{_2^2}+\frac{1}{2}\|H_{tt}\|{_2^2}+\frac{d}{dt}G(t)\\
&\leq C[\|(\rho_{tt},P_{tt},\lambda_{tt})\|{_2^2}+\|u_t\|{_4^2}+\|\nabla^3 u\|{_2^2}+(\|\nabla u\|{_\infty^2}+1)(\|\nabla u_t\|{_2^2}+\|\nabla H_t\|{_2^2}+1)],\nonumber
\end{align}
where 
\begin{align}
G(t)&=\int(\mu|\nabla u_t|^2+\nu|\nabla H_t|^2+(\mu+\lambda)|\di u_t|^2)dx\nonumber\\
&~~~~+\int(\lambda_t\di u\di u_t-P_t\di u_t+\rho_t\frac{|u_t|^2}{2}+\rho_tu\cdot\nabla u\cdot u_t)dx\\
&~~~~-\int((H_t\cdot\nabla)H\cdot u_t+(H\cdot\nabla)H_t\cdot u_t+H\cdot H_t\di u_t)dx.\nonumber
\end{align}
Note that 
\begin{align}
|\int(\lambda_t\di u\di u_t-P_t\di u_t)dx|&\leq\frac{\mu}{8}\|\nabla u_t\|{_2^2}+C\|\lambda_t\|{_4^2}\|\di u\|{_4^2}+C\|P_t\|{_2^2}\nonumber\\
&\leq\frac{\mu}{8}\|\nabla u_t\|{_2^2}+C,\nonumber
\end{align}
\begin{align}
&|\int(\rho_t\frac{|u_t|^2}{2}+\rho_tu\cdot\nabla u\cdot u_t)dx|\nonumber\\
&=|\int\di(\rho u)(\frac{|u_t|^2}{2}+u\cdot\nabla u\cdot u_t)dx|=|\int\rho u\cdot(\nabla u_t\cdot u_t+\nabla(u\cdot\nabla u\cdot u_t))dx|\nonumber\\
&\leq\|\sqrt{\rho}u_{tt}\|_2\|\sqrt{\rho}u\|_\infty(\|\nabla u_t\|_2+\|\nabla u\|{_4^2}+\|u\|_\infty\|\nabla^2u\|_2)+\|\rho u^2\|_\infty\|\nabla u\|_2\|\nabla u_t\|_2\nonumber\\
&\leq\frac{\mu}{8}\|\nabla u_t\|{_2^2}+C,\nonumber
\end{align}
and
\begin{align}
&|-\int((H_t\cdot\nabla)H\cdot u_t+(H\cdot\nabla)H_t\cdot u_t+H\cdot H_t\di u_t)dx|\nonumber\\
&=|\int((H_t\cdot\nabla)u_t\cdot H+(H\cdot\nabla)u_t\cdot H_t-H\cdot H_t\di u_t)dx|\nonumber\\
&\leq C\|H\|_\infty\|H_t\|_2\|\nabla u_t\|_2\leq\frac{\mu}{8}\|\nabla u_t\|{_2^2}+C.\nonumber
\end{align}
Therefore, it holds that
\begin{equation}\label{2.4.55}
c(\|\nabla u_t\|{_2^2}+\|\nabla H_t\|{_2^2}-1)\leq G(t)\leq C(\|\nabla u_t\|{_2^2}+\|\nabla H_t\|{_2^2}+1),
\end{equation} 
for some positive constants $c$ and $C$. Thus, we have
\begin{align}
&\frac{1}{2}\|\sqrt{\rho}u_{tt}\|{_2^2}+\frac{1}{2}\|H_{tt}\|{_2^2}+\frac{d}{dt}G(t)\nonumber\\
&\leq C[\|(\rho_{tt},P_{tt},\lambda_{tt})\|{_2^2}+\|u_t\|{_4^2}+\|\nabla^3 u\|{_2^2}+(\|\nabla u\|{_\infty^2}+1)(G(t)+1)].
\end{align}
Multiplying the above inequality by $t$ and then integrating the resulting inequality with respect to $t$ over the interval $[\tau,t_1]$ with $\tau,t_1\in[0,T]$ give that
\begin{align}\label{2.4.57}
\int_{\tau}^{t_1}t(\|\sqrt{\rho}u_{tt}\|{_2^2}+\frac{1}{2}\|H_{tt}\|{_2^2})dt&+t_1G(t_1)\leq\tau G(\tau)+C\int_{\tau}^{t_1}[(\|\nabla u\|{_\infty^2}+1)(tG(t)+1)]dt\nonumber\\
&+C\int_{\tau}^{t_1}[\|(\rho_{tt},P_{tt},\lambda_{tt})\|{_2^2}+\|u_t\|{_4^2}+\|\nabla^2u\|{_2^2}+G(t)]dt.
\end{align}
It follows from Lemma \ref{lem3.11} and $(\ref{2.4.55})$ that $G(t)\in L^1(0,T)$. Thus, there exists a subsequence $\tau_k$ such that 
\begin{equation}
\tau_k\rightarrow 0,~~~~~~~~~\tau_kG(\tau_k)\rightarrow 0,~~~~~~~~~as~~ k\rightarrow+\infty.\nonumber
\end{equation}
Taking $\tau=\tau_k$ in $(\ref{2.4.57})$, then $k\rightarrow+\infty$ and using the Gronwall's inequality, one has that
\begin{equation}
\sup\limits_{0\leq t\leq T}[t(\|\nabla u_t\|{_2^2}+\|\nabla H_t\|{_2^2})]+\int_{0}^{T}t(\|\sqrt{\rho}u_{tt}\|{_2^2}+\frac{1}{2}\|H_{tt}\|{_2^2})dt\leq C,
\end{equation}
which combining with $(\ref{2.4.22})$ and $(\ref{2.4.40})$ gives that
\begin{align}
\sup\limits_{0\leq t\leq T}t\|(\rho_{tt},P_{tt},\lambda_{tt})\|{_2^2}+\int_{0}^{T}t(\|\nabla^2u_t\|{_2^2}+\|\nabla^2H_t\|{_2^2})dt\leq C,
\end{align}
where we have used 
\begin{align}
t\|u_t\|{_2^2}&\leq t(\|\sqrt{\rho}u_t\|{_2^2}+\|\nabla u_t\|{_2^2})
\leq C.
\end{align}
So we can conclude that 
\begin{align}
\sup\limits_{0\leq t\leq T}t\|(u_t,H_t)\|{_{H^1}^2}+\int_{0}^{T}t\|(u_t,H_t)\|{_{H^2}^2}dt\leq C.
\end{align}
Next, applying the operator $\nabla^2$ to the elliptic equation $(\ref{MHD})_1$, multiplying the resulting equation by $q|\nabla^2\rho|^{q-2}\nabla^2\rho$ with $q>2$, and integrating over $\mathbb T^2$, one has that
\begin{align}\label{2.4.62}
\frac{d}{dt}\|\nabla^2\rho\|_q&\leq C(\|\nabla u\|_\infty\|\nabla^2\rho\|_q+\|\nabla\rho\|_{2q}\|\nabla^2\|_{2q}+\|\rho\|_\infty\|\nabla^3u\|_q)\nonumber\\
&\leq C(\|\nabla u\|_\infty\|\nabla^2\rho\|_q+\|\nabla^2u\|_{W^{1,q}}).
\end{align}
Similarly, one has 
\begin{equation}\label{2.4.63}
\frac{d}{d}\|\nabla^2P\|_q\leq C(\|\nabla u\|_\infty\|\nabla^2P\|_q+\|\nabla^2u\|_{W^{1,q}}).
\end{equation}
Applying $\partial_i$ with $i=1,2$ to the elliptic system $(\ref{MHD})_2$ to obtain
\begin{align}\label{2.4.66}
\mu\Delta(\partial_iu)+\nabla((\mu+\lambda)\di(\partial_iu))&=-\nabla(\partial_i\lambda\di u)+\partial_i\rho u_t+\rho\partial_iu_t+\partial_iu\cdot\nabla u\nonumber\\
&~~~~+\rho\partial_iu\cdot\nabla u+\rho u\cdot\partial_iu+\nabla\partial_iP\\
&~~~~+\partial_i(-H\cdot\nabla H)+\nabla(H\cdot\partial_iH)=:\Psi.\nonumber
\end{align}
Then the standard elliptic regularity estimates imply that
\begin{align}\label{2.4.64}
\|\nabla u\|_{W^{2,q}}&\leq C(\|\nabla u\|_q+\|\Psi\|_q)\nonumber\\
&\leq C[1+(\|\nabla u\|_\infty+1)\|(\nabla^2\rho,\nabla^2P)\|_q+\|\nabla u\|_{W^{1,q}}\nonumber\\
&~~~~~~+\|u_t\|_{W^{1,q}}+\|\nabla H\|_{W^{1,q}}+\|\nabla^2H\|_q]\\
&\leq C[1+(\|\nabla u\|_\infty+1)\|(\nabla^2\rho,\nabla^2P)\|_q+\|(u,H)\|_{H^3}\nonumber\\
&~~~~~~~+\|(u_t,H_t)\|_{H^1}+\|\nabla u_t\|_q+\|\nabla^2H\|_q].\nonumber
\end{align}
Similarly, 
\begin{align}\label{2.4.65}
\|\nabla H\|_{W^{2,q}}&\leq C(\|\nabla H\|_q+\|\nabla H_t\|_q+\|\nabla H\|_{2q}\|\nabla u\|_{2q}+\|\nabla^2 u\|_q+\|\nabla^2 H\|_q)\nonumber\\
&\leq C\|(\nabla u,\nabla H)\|_{W^{1,q}}\leq C\|(u,H)\|_{H^3}.
\end{align}
Thus it follows from $(\ref{2.4.62})$-$\ref{2.4.65}$ that
\begin{align}\label{2.4.67}
\frac{d}{dt}\|(\nabla^2\rho,\nabla^2P)\|_q&\leq C[1+(\|\nabla u\|_\infty+1)\|(\nabla^2\rho,\nabla^2P)\|_q\nonumber\\
&~~~~~~+\|(u,H)\|_{H^3}+\|(u_t,H_t)\|_{H^1}+\|\nabla u_t\|_q]
\end{align}
Note that Poincar\'{e}'s inequality implies that
\begin{equation}
\int_{0}^{T}\|\nabla u_t\|_qdt\leq C\int_{0}^{T}\|\nabla^2u_t\|dt\leq C\sup\limits_{0\leq t\leq T}\sqrt{t}\|\nabla^2u_t\|_2\int_{0}^{T}t^\frac{1}{2}dt\leq C.
\end{equation}
Therefore, by the Gronwall's inequality, one has
\begin{equation}
\sup\limits_{0\leq t\leq T}\|(\nabla^2\rho,\nabla^2P)\|_q\leq C.
\end{equation}
Finally, since 
\begin{align*}
\|\nabla\dot{u}\|_2\leq \|\nabla u_t\|_2+\|\nabla(u\cdot\nabla u)\|_2\leq C\|\nabla u_t\|_2+C,
\end{align*}
it follows from (\ref{3.88}), (\ref{2.4.30}) and (\ref{2.4.31}) that
\begin{equation}
t\|(u,H)\|{_{H^3}^2}\leq C.
\end{equation}
Applying $\partial_j$ with $j=1,2$ to (\ref{2.4.66}), the standard $L^2$-estimates easily gives that
\begin{equation}
\int_{0}^{T}t\|(u,H)\|{_{H^4}^2}\leq C.
\end{equation} 
\end{proof}
\begin{lemma}
It holds for any $2<q<+\infty$ that
\begin{align}
\sup\limits_{0\leq t\leq T}&[t^2(\|\sqrt{\rho}u_{tt}\|{_2^2}+\|H_{tt}\|{_2^2}+\|(u_t,H_t)\|{_{H^2}^2}+\|(u,H)\|{_{W^{3,q}}^2})]\nonumber\\
&+\int_{0}^{T}t^2(\|\nabla u_{tt}\|{_2^2}+\|\nabla H_{tt}\|{_2^2})dt\leq C.
\end{align}
\end{lemma}
\begin{proof}
First, applying the operator $\partial_{tt}$ to the equation $(\ref{MHD})_2$ gives that
\begin{align}
&\rho u_{ttt}+\rho u\cdot \nabla u_{tt}-\mu\Delta u_{tt}-\nabla((\mu+\lambda)\di u_{tt})\nonumber\\
&=-\nabla P_{tt}-\rho_{tt}(u_t+u\cdot\nabla u)-2\rho_t(u_{tt}+u_t\cdot\nabla u+u\cdot\nabla u_t)-2\rho u_t\cdot\nabla u_t\nonumber\\
&~~~~-\rho u_{tt}\cdot\nabla u+2\nabla(\lambda_t\di u_t)+\nabla(\lambda_{tt}\di u)\nonumber\\
&~~~~+(H_{tt}\cdot\nabla)H+2(H_t\cdot\nabla)H_t+(H\cdot\nabla)H_{tt}-\nabla(|H_t|^2+HH_{tt}).\nonumber
\end{align}
Multiplying the above equation by $t^2u_{tt}$ and integrating the resulting equation with respect to $x$ over $\mathbb T^2$ yields that
\begin{align}\label{2.4.70}
&\frac{1}{2}\frac{d}{dt}(t^2\int\rho|u_{tt}|^2dx)-t\int\rho|u_{tt}|^2dx+t^2\int\mu|\nabla u_{tt}|^2+(\mu+\lambda)(\di u_{tt})^2dx=t^2\int P_{tt}\di u_{tt}dx\nonumber\\
&-t^2\int\rho_{tt}(u_t+u\cdot\nabla u)\cdot u_{tt}dx-2t^2\int\rho_t(u_{tt}+u_t\cdot\nabla u+u\cdot\nabla u_t)\cdot u_{tt}dx\nonumber\\
&-2t^2\int\rho u_t\cdot\nabla u_t\cdot u_{tt}dx-t^2\int\rho u_{tt}\cdot\nabla u\cdot u_{tt}dx-2t^2\int\lambda_t\di u_t\di u_{tt}dx-t^2\int\lambda_{tt}\di u\di u_{tt}dx\nonumber\\
&-t^2\int(H_{tt}\cdot\nabla)u_{tt}\cdot H+2(H_t\cdot\nabla)u_{tt}\cdot H_t+(H\cdot\nabla)u_{tt}\cdot H_{tt}dx+t^2\int(|H_t|^2+HH_t)\di u_{tt}dx\nonumber\\
&=:\sum_{1}^{9}K_i.
\end{align}
Clearly,
\begin{equation}
|K_1|\leq\varepsilon t^2\|\di u_{tt}\|{_2^2}+Ct^2\|P_{tt}\|{_2^2}.
\end{equation}
Now we estimate $K_2$. First, rewrite $K_2$ as
\begin{align}
K_2&=t^2\int\di(\rho u)_t\dot u\cdot u_{tt}dx=-t^2\int(\rho u)_t\cdot\nabla(\dot{u}\cdot u_{tt})dx\nonumber\\
&=-t^2\int\rho u_t\cdot\nabla(\dot{u}\cdot u_{tt})dx-t^2\int\rho_tu\cdot\nabla u_{tt}\cdot\dot{u}-t^2\int\rho_tu\cdot\dot{u}\cdot u_{tt}dx\nonumber\\
&=-t^2\int\rho u_t\cdot\nabla(\dot{u}\cdot u_{tt})dx-t^2\int\rho_tu\cdot\nabla u_{tt}\cdot\dot{u}-t^2\int\rho u\cdot\nabla(u\cdot\dot{u}\cdot u_{tt})dx\nonumber\\
&=:K_{21}+K_{22}+K_{23}.
\end{align}
Now, direct estimates yields that
\begin{align}
|K_{21}|&\leq t^2\|\sqrt{\rho}u_{tt}\|_2\|\sqrt{\rho}\|_\infty\|u_t\|_\infty\|\nabla\dot{u}\|_2+t^2\|\rho\|_\infty\|\nabla u_{tt}\|_2\|u_t\|_4\|\dot{u}\|_4\nonumber\\
&\leq Ct^2[\|\sqrt{\rho}u_{tt}\|_2\|u_t\|_{H^2}\|\nabla\dot{u}\|_2+\|\nabla u_{tt}\|_2\|u_t\|_{H^1}\|\dot{u}\|_{H^1}]\\
&\leq\varepsilon t^2\|\nabla u_{tt}\|{_2^2}+C(t^2\|\sqrt{\rho}u_{tt}\|{_2^2}\|\nabla\dot{u}\|{_2^2}+t^2\|u_t\|{_{H^2}^2}+\|\nabla\dot{u}\|{_2^2}+1),\nonumber
\end{align}
\begin{align}
|K_{22}|&\leq t^2\|\nabla u_{tt}\|_2\|u\|_\infty\|\rho_t\|_4\|\dot{u}\|_4\leq t^2\|\nabla u_{tt}\|_2\|\rho_t\|_{H^1}\|\dot{u}\|_{H^1}\nonumber\\
&\leq\varepsilon t^2\|\nabla u_{tt}\|{_2^2}+C(\|\nabla\dot{u}\|{_2^2}+1),
\end{align}
and
\begin{align}
&|K_{23}|\leq t^2(\|\sqrt{\rho}u_{tt}\|_2\|\sqrt{\rho}u\|_\infty\|\nabla u\|_\infty\|\nabla\dot{u}\|_2+\|\nabla u_{tt}\|_2\|\rho u^2\|_\infty\|\nabla\dot{u}\|_2\nonumber\\
&~~~~\|\sqrt{\rho}u_{tt}\|_2\|\sqrt{\rho}u^2\|_\infty(\|\nabla^2u_t\|_2+\|u\|_\infty\|\nabla^3u\|_2+\|\nabla u\|_\infty\|\nabla^2u\|_2))\\
&\leq\varepsilon t^2\|\nabla^2u_{tt}\|{_2^2}+C[t^2\|\sqrt{\rho}u_{tt}\|{_2^2}(\|\nabla u\|{_\infty^2}+1)+t^2\|\nabla^2u_t\|{_2^2}+t^2\|u\|{_{H^3}^2}+\|\nabla\dot{u}\|{_2^2}].\nonumber
\end{align}
As to other terms, one has,
\begin{align}
|K_3|&=-2t^2\int\rho u\cdot\nabla((u_{tt}+u_t\cdot\nabla u+u\cdot\nabla u_t)\cdot u_{tt})dx\nonumber\\
&\leq Ct^2[\|\sqrt{\rho}u\|_\infty\|\sqrt{\rho}u_{tt}\|_2\|\nabla u_{tt}\|_2+\|\rho u\|_\infty\|\nabla u_{tt}\|_2(\|u_t\cdot\nabla u\|_2+\|u\cdot\nabla u_t\|_2)\nonumber\\
&~~~~~+\|\sqrt{\rho u}\|_\infty\|\sqrt{\rho}u_{tt}\|_2(\|\nabla(u_t\cdot\nabla u)\|_2+\|\nabla(u\cdot\nabla u_t)\|)]\nonumber\\
&\leq\varepsilon t^2\|\nabla u_{tt}\|{_2^2}+Ct^2(t^2\|\sqrt{\rho}u_{tt}\|{_2^2}+t^2\|u_t\|{_{H^2}^2}+\|\nabla^3 u\|_3+1),
\end{align}
\begin{align}
|K_4|\leq t^2\|\sqrt{\rho}u_{tt}\|_2\|\sqrt{\rho}\|_\infty\|u_t\|_4\|\nabla u_t\|_4\leq\varepsilon t^2\|\sqrt{\rho}u_{tt}\|{_2^2}+C(\|\nabla^2\|{_2^2}+1),
\end{align}
\begin{align}
|K_5|\leq Ct^2\|\sqrt{\rho}u_{tt}\|{_2^2}\|\nabla u\|_\infty,
\end{align}
\begin{align}
|K_6|\leq Ct^2\|\di u_{tt}\|_2\|\lambda_t\|_4\|\nabla u_t\|_4\leq\varepsilon t^2\|\di u_{tt}\|{_2^2}+C(t^2\|\nabla^2u_t\|{_2^2}+1),
\end{align}
\begin{align}
|K_7|\leq t^2\|\di u_{tt}\|_2\|\lambda_{tt}\|_4\|\nabla u\|_\infty\leq\varepsilon t^2\|\di u_{tt}\|{_2^2}+Ct^2\|\lambda_{tt}\|{_2^2}\|u\|{_{H^3}^2},
\end{align}
\begin{align}
|K_8+K_9|&\leq Ct^2[\|H_{tt}\|_2\|\nabla u_{tt}\|_2+\|H_t\|{_4^2}\|\nabla u_{tt}\|_2+(\|H_t\|{_4^2}+\|H_{tt}\|_2)\|\di u_{tt}\|_2]\nonumber\\
&\leq\varepsilon t^2(\|\nabla u_{tt}\|{_2^2}+\|\di u_{tt}\|{_2^2})+C(t^2\|H_{tt}\|{_2^2}+1).
\end{align}
Substituting the above estimates on $K_i$ into $(\ref{2.4.70})$ and then integrating the resulting inequality with respect $t$ over $[\tau,t_1]$ give that
\begin{align}
t{_1^2}(\|\sqrt{\rho}u_{tt}(t_1)\|{_2^2}+\|H_{tt}(t_1)\|{_2^2}+\int_{\tau}^{t_1}t^2\|\nabla u_{tt}\|{_2^2}dt\leq C+C\tau^2\|\sqrt{\rho}u_{tt}(\tau)\|{_2^2}.
\end{align}
Since $t\sqrt{\rho}u_{tt}\in L^2([0,T]\times\mathbb T^2)$, there exists a subsequence $\tau_k$ such that
\begin{equation}
\tau_k\rightarrow 0,~~~~~~~~\tau{_k^2}\|\sqrt{\rho}u_{tt}(\tau_k)\|{_2^2}\rightarrow 0,~~~~~~~~as~~k\rightarrow +\infty.
\end{equation}
Letting $\tau=\tau_k$ and $k\rightarrow +\infty$, one obtains that
\begin{align}
t^2(\|\sqrt{\rho}u_{tt}\|{_2^2}+\|H_{tt}\|{_2^2})+\int_{0}^{t}s^2\|\nabla u_{tt}(s)\|{_2^2}dt\leq C.
\end{align}
By $(\ref{2.4.40})$ that
\begin{equation}
\sup\limits_{0\leq t\leq T}t^2\|\nabla^2u_t\|{_2^2}\leq C\sup\limits_{0\leq t\leq T}(t^2\|\sqrt{\rho}u_{tt}\|{_2^2}+t^2\|u_t\|{_{H^1}^2}+t^2\|u\|{_{H^3}^2}+1)\leq C.
\end{equation}
Finally, by $(\ref{2.4.64})$ and $(\ref{2.4.65})$, we obtain that
\begin{equation}
\sup\limits_{0\leq t\leq T}t^2\|(\nabla u,\nabla H)\|{_{W^{2,q}}^2}\leq C\sup\limits_{0\leq t\leq T}(t^2\|(u,H)\|{_{H^3}^2}+t^2\|u_t\|{_{H^2}^2}+1)\leq C.
\end{equation}
So the proof of the Lemma is completed.
\end{proof}

\section{A Priori estimate (II)}

\noindent\;\;\;\; In this section, we will mainly derive the a priori estimates upper bound of density for the (IVP) of $(\ref{MHD})$ and $(\ref{IVC1})$-$(\ref{IVC2})$ on the whole space $\mathbb{R}^2$ under the assumption $\inf\limits_{x\in\mathbb{R}^2}\rho_0\geq\delta>0$. These estimates is uniform with respect to $\delta$.

Comparing with the periodic problems of the compressible MHD equations studied in Chapter 2, some new difficulties must be overcome to obtain the upper bound of the density. First, the Poincar\'{e}-type inequality fails for the 2D Cauchy problem on $\mathbb{R}^2$ so that the $L^p$-integrability of the velocity $u$ does not follow from $||\nabla u||_2$ directly, although the bound $||\nabla u||_2$ also derived from the elementary energy estimates easily.
While the $L^p$-integrability $(2\leq p<\infty)$ of the velocity $u$ plays a crucial role in the arbitrary $L^p$-integrability of the density $\rho$. We will make use of the Caffarelli-Kohn-Nirenberg inequality to derive the weighted estimates of the velocity $|x|^\frac{a}{2}\nabla u\in L^2((0,T)\times \mathbb{R}^2)$, which is strongly coupled with the higher integrability estimates of the density $\rho$. For the 2D Cauchy problem with non-vacuum far fields, we will seek for the arbitrary $L^p$ estimates of $\rho-\tilde{\rho}$ since the loss of integrability of the density $\rho$ itself.

First we state the elementary energy estimates and a priori $L{^\infty_t} L{^p_x}$ estimates for magnetic field $H$ for 2D Cauchy problem on $\mathbb{R}^2$. We omit the proof of these estimates for simplicity since the similarity to Lemma \ref{BE} and Lemma \ref{BH} proved in the last subsection for the periodic problem.
\begin{lemma}
There exists a positive constant $C$ only depending on $(\rho_0,u_0,H_0)$, such that for $\tilde{\rho}\geq 0$, it holds that
\begin{align}\label{BER}
\sup\limits_{0 \leq t\leq T}&(\|\sqrt{\rho}u\|{^2_2}+\|\Psi(\rho,\tilde{\rho})\|_1+\|H\|{^2_2})\nonumber\\
&+\int_{0}^{T}(\|\nabla u\|{^2_2}+\|\omega\|{^2_2}+\|(2\mu+\lambda(\rho))^{1/2}\di u\|{^2_2}+\|\nabla H\|{^2_2})dt\leq C
\end{align}
where the potential energy $\Psi(\rho,\tilde{\rho})$ is given either $\tilde{\rho}>0$ or $\tilde{\rho}=0$ by
\begin{equation}
\Psi(\rho,\tilde{\rho})=\frac{1}{\gamma-1}[\rho^\gamma-\tilde{\rho}^\gamma-\gamma\tilde{\rho}^{\gamma-1}(\rho-\tilde{\rho})].
\end{equation}
Moreover, if $\tilde{\rho}=0$, one has $\sup\limits_{0\leq t\leq T}\|\rho\|_1\leq C$. 
\end{lemma}
\begin{lemma}
For any $p\geq 2$, there exists a positive constant $C$ such that 
\begin{equation}\label{3.3.3}
\sup\limits_{0\leq t\leq T}\|H\|_p\leq C. 
\end{equation}
\end{lemma}

Now we derive the $L{_t^\infty} L{_x^p}$ $(2\leq p<+\infty)$ estimates of the density $\rho-\tilde{\rho}~(\tilde{\rho}\geq 0)$ in the following two cases separately.

{\bf{Case I: initial density with vacuum at far fields, that is, $\tilde{\rho}=0$.}}

The following weighted energy estimates are fundamental and crucial in this section.
\begin{lemma}
If one of the following restrictions holds:
\begin{align}
& 1)~~1<a<2\sqrt{\sqrt{2}-1}, \beta>0, \gamma>1,\\
& 2)~~0<a\leq 1, \beta>\frac{1}{2}, 1<\gamma\leq 2\beta,
\end{align}
then it holds that for sufficiently large $m>1$ and any $t\in[0,T]$
\begin{align}
\int&|x|^a(\rho|u|^2+\rho^\gamma+|H|^2)(t,x)dx\nonumber\\
&+\int_{0}^{t}[\||x|^\frac{a}{2}\nabla u\|{_2^2}+\||x|^\frac{a}{2}\nabla H\|{_2^2}|+\|x|^\frac{a}{2}\di u\|{_2^2}+\||x|^\frac{a}{2}\sqrt{\lambda(\rho)}\di u\|{_2^2}](s)ds\nonumber\\
&\leq C_a[1+\int_{0}^{t}(\|\rho\|{_{2m\beta+1}^\beta}(s)+1)(\|\nabla u\|{_2^2}(s)+1)ds]
\end{align}
where the positive constant $C_a$ may depend on $a$ but is independent of $m$.
\end{lemma}
\begin{proof}
Multiplying the equation $(\ref{MHD})_2$ by $u$ and equation $(\ref{MHD})_3$ by $H$, summing the resulting equations and using the continuity equation $(\ref{MHD})_1$, we have 
\begin{align}
&\frac{d}{dt}(\rho\frac{|u|^2}{2}+\frac{|H|^2}{2}+\frac{\rho^\gamma}{\gamma-1})+\di(\rho u\frac{|u|^2}{2}+\frac{\gamma\rho^\gamma u}{\gamma-1})\nonumber\\
&=\di((u\times H)\times H)+\di[\mu\nabla(\frac{|u|^2}{2})+\nu\nabla(\frac{|H|^2}{2})+(\mu+\lambda(\rho)(\di u)u)]\nonumber\\
&\quad-\mu|\nabla u|^2-\nu|\nabla H|^2-(\mu+\lambda(\rho)(\di u)^2)\nonumber,
\end{align}
Multiplying the above identity by $|x|^a$ and integrating over $\mathbb{R}^2$,it holds that
\begin{align}\label{3.3.7}
&\frac{d}{dt}\int|x|^a[\rho\frac{|u|^2}{2}+\frac{|H|^2}{2}+\frac{\rho^\gamma}{\gamma-1}]dx\nonumber\\
&~~~~+\int[\mu\||x|^\frac{a}{2}\nabla u\|{_2^2}+\nu\||x|^\frac{a}{2}\nabla H\|{_2^2}+\mu\||x|^\frac{a}{2}\di u\|{_2^2}+\||x|^\frac{a}{2}\sqrt{\lambda(\rho)}\di u\|{_2^2}](s)dx\nonumber\\
&=\int[\rho u\frac{|u|^2}{2}+\frac{\gamma\rho^\gamma u}{\gamma-1}]\cdot\nabla(|x|^a)dx\\
&\quad-\int[\mu\nabla(\frac{|u|^2}{2})+(\mu+\lambda(\rho)\di uu)]\cdot\nabla(|x|^a)dx\nonumber\\
&\quad-\int\nu\nabla(\frac{|H|^2}{2})\cdot\nabla(|x|^a)dx-\int((u\times H)\times H)\cdot\nabla(|x|^a)dx=:\sum_{i=1}^{4}I_i\nonumber,
\end{align}
Now we estimate the terms on the RHS of $(\ref{3.3.7})$. As to the first two terms, which had been estimated in Lemma 3.2 of Jiu-Wang-Xin\cite{jwx2} for compressible Navier-Stokes, it holds that
\begin{equation}\label{3.3.8}
|I_1|\leq \sigma|||x|^\frac{a}{2}\nabla u||{_2^2}+C_\sigma(||\rho||{_{2m\beta+1}^\beta})(||\nabla u||{_2^2}+1)
\end{equation}
and
\begin{align}\label{3.3.9}
|I_2|&\leq \sigma[\||x|^\frac{a}{2}\sqrt{\lambda(\rho)}\di u\|{_2^2}+\||x|^\frac{a}{2}\nabla u\|{_2^2}]\nonumber\\
&~~+\frac{\mu a^2}{2}\||x|^\frac{a}{2}\nabla u\|{_2^2}+\frac{\mu a^2}{2}\||x|^\frac{a}{2}\di u\|_2\||x|^\frac{a}{2}\nabla u\|_2\\
&~~+C_\sigma(\|\rho\|{_{2m\beta+1}^\beta})\|\nabla u\|{_2^2}\nonumber
\end{align}
Next, the last two terms estimate as follows,
\begin{equation}\label{3.3.10}
|I_3|\leq\nu a\int|H||\nabla H||x|^{a-1}dx\leq\nu a\||x|^{\frac{a}{2}-1}H\|_2\||x|^\frac{a}{2}\nabla H\|_2\leq\frac{\nu a^2}{2}\||x|^\frac{a}{2}\nabla H\|{_2^2}
\end{equation}
where we have used the Caffarelli-Kohn-Nirenberg inequality with best constant shown in Lemma \ref{CKN}(2). While,
\begin{align}\label{3.3.11}
|I_4|&\leq 2\int|u||H|^2|x|^{a-1}dx\leq 2\||x|^{a-1}u\|_{q_1}\|H^2\|_{p_1}\leq C\|\nabla u\|{_2^{\theta_1}}\||x|^{\kappa_1}u\|{_{r_1}^{1-\theta_1}}\nonumber\\
&\leq C\|\nabla u\|{_2^{\theta_1}}\||x|^\frac{a}{2}\nabla u\|{_2^{1-\theta_1}}\leq\sigma\||x|^\frac{a}{2}\nabla u\|{_2^2}+C(\|\nabla u\|{_2^2}+1)
\end{align}
where we have used the Caffarelli-Kohn-Nirenberg inequality and Holder's inequality so that the indexes $p_1>1,q_1>1,r_1>0,\theta_1\in(0,1)$ satisfy
\begin{equation*}
\frac{1}{p_1}+\frac{1}{q_1}=1;~~~~
\frac{1}{q_1}+\frac{a-1}{2}=(1-\theta_1)(\frac{1}{r_1}+\frac{\kappa_1}{2});~~~~
\frac{1}{r_1}+\frac{\kappa_1}{2}=\frac{1}{2}+\frac{\frac{a}{2}-1}{2}.
\end{equation*}
which follows that
\begin{equation*}
p_1=\frac{4}{a(1+\theta_1)+2}>1
\end{equation*}
by choosing $0<a<2, \theta_1\in(0,1).$ Substituting $(\ref{3.3.8})$-$(\ref{3.3.11})$ into $(\ref{3.3.7})$, one has
\begin{align}\label{3.3.12}
\frac{d}{dt}&\int|x|^a[\rho\frac{|u|^2}{2}+\frac{|H|^2}{2}+\frac{\rho^\gamma}{\gamma-1}]dx+J(t)\\
&\leq\sigma(|||x|^\frac{a}{2}\nabla u||{_2^2}+3|||x|^\frac{a}{2}\sqrt{\lambda(\rho)}\di u||{_2^2})+C_\sigma(||\rho||{_{2m\beta+1}^\beta}+1)(||\nabla u||{_2^2}+1)\nonumber
\end{align}
where $J(t)$ is defined as follows
\begin{align}\label{3.3.13}
&J(t):=\mu(1-\frac{a^2}{2})\||x|^\frac{a}{2}\nabla u\|{_2^2}(t)-\frac{\mu a^2}{2}\||x|^\frac{a}{2}\nabla u\|_2\||x|^\frac{a}{2}\di u\|_2(t)+\mu\||x|^\frac{a}{2}\di u\|{_2^2}(t)\nonumber\\
&\quad+\||x|^\frac{a}{2}\sqrt{\lambda(\rho)}\di u\|{_2^2}+\nu(1-\frac{a^2}{2})\||x|^\frac{a}{2}\nabla H\|{_2^2}
\end{align}
If the weight $a$ satisfies
\begin{equation*}
 0<a^2<4(\sqrt{2}-1)
 \end{equation*}
then there exists a positive constant $C_a$ such that
\begin{equation}\label{3.3.14}
J(t)\geq C_a(\||x|^\frac{a}{2}\nabla u\|{_2^2}+\||x|^\frac{a}{2}\di u\|{_2^2}+\||x|^\frac{a}{2}\sqrt{\lambda(\rho)}\di u\|{_2^2}+\||x|^\frac{a}{2}\nabla H\|{_2^2})
\end{equation}
Therefore, by choosing $\sigma$ suitable small, the result directly follows  from $(\ref{3.3.12})$-$(\ref{3.3.14})$ and the fact 
\begin{equation*}
\int|x|^a\rho{_0^\gamma}dx\leq \||x|^a\rho_0\|_1\|\rho{_0^{\gamma-1}}\|_\infty\leq C\|\rho_0(1+|x|^{a_1})\|_1\|\rho_0\|{_{W^{2,q}(\mathbb R^2)}^{\gamma-1}}\leq C.
\end{equation*}
\end{proof}

Applying the operator $\di$ to the momentum equation $(\ref{MHD})_2$,we have
\begin{equation}\label{3.3.15}
[\di(\rho u)]_t+\di[\di(\rho u\otimes u-H\otimes H)]=\Delta F.
\end{equation}
where the effective flux $F$ is given by
\begin{equation}\label{3.3.16}
F:=(2\mu+\lambda(\rho))\di u-P(\rho)-\frac{|H|^2}{2}
\end{equation}
Similar to the periodic problem studied in last subsection, we need the following elliptic estimates to derive the arbitrary $L^p$ integrability of the density.

Consider the following two elliptic problems on $\mathbb R^2$:
\begin{align}
-\Delta\xi=\di(\rho u),~~~~\xi\rightarrow 0~~as~~|x|\rightarrow +\infty,\label{3.3.17}\\
-\Delta\eta=\di[\di(\rho u\otimes u-H\otimes H)],~~~~\eta\rightarrow 0~~as~~|x|\rightarrow +\infty,\label{3.3.18}
\end{align}
We have the following elliptic estimates like ones in last subsection:
\begin{lemma}\label{EllR1}
\begin{flalign*}
\begin{split}
&\quad(1)~\|\nabla\xi\|_{2m}\leq Cm\|\rho\|_{\frac{2mk}{k-1}}\|u\|_{2mk},~for~ any~ k>1,m\geq 1;\\
&\quad(2)~\|\nabla\xi\|_{2-r}\leq C\|\rho\|{_\frac{2-r}{r}^\frac{1}{2}},~for~ any~ 0<r<1;\\
&\quad(3)~\|\eta\|\leq Cm(\|\rho\|_\frac{2mk}{k-1}\|u\|{_{4mk}^2}+\|H\|{_{4m}^2}),~for~ any~ k>1,m\geq 1;
\end{split}&
\end{flalign*}
where $C$ are positive constants independent of $m,k$ and $r$.
\end{lemma}
Based on the above lemma, it holds that
\begin{lemma}\label{EllR2}
\begin{flalign*}
\begin{split}
&\quad(1)~\|\xi\|_{2m}\leq Cm^\frac{1}{2}\|\nabla\xi\|_\frac{2m}{m+1}\leq Cm^\frac{1}{2}\|\rho\|{_m^\frac{1}{2}}, ~for~ any~ m\geq 2;\\
&\quad(2)~\|u\|_{2m}\leq Cm^\frac{1}{2}\|\nabla u\|{_2^{1-\frac{1}{ma}}}\||x|^\frac{a}{2}\nabla u\|{_2^\frac{1}{ma}},~for~ any~ m+1\geq\frac{4}{a};\\
&\quad(3)~\|\nabla\xi\|_{2m}\leq Cm^\frac{3}{2}k^\frac{1}{2}\|\rho\|_\frac{2mk}{k-1}\|\nabla u\|{_2^{1-\frac{2}{mka}}}\||x|^\frac{a}{2}\nabla u\|{_2^\frac{2}{mka}},~ for~ any~ m+1\geq\frac{4}{a},k>1;\\
&\quad(4)~\|\eta\|_{2m}\leq Cm^2k\|\rho\|_{\frac{2mk}{k-1}}\|\nabla u\|{_2^{2-\frac{2}{mka}}}\||x|^\frac{a}{2}\nabla u\|{_2^\frac{2}{mka}},~ for ~any~ m+1\geq\frac{4}{a},k>1;
\end{split}&
\end{flalign*}
\end{lemma}
Note that, as shown in Lemma \ref{EllR2}(2), $||u||_{2m}$ can't be bounded by $||\nabla u||_2$ only but instead of by the additional weighted-norm of $\nabla u$. This is the key difference to the periodic problem. The proof of the lemma follows from the H\"{o}lder inequality and the Caffarelli-Kohn-Nirenberg inequality in Lemma \ref{CKN}(2) directly. One can refer to \cite{jwx2}.

Finally, with these estimates at hand, we can obtain the $L{_t^\infty} L{_x^p}$-estimates of the density $\rho$. One can refer to the Lemma 3.5 of Jiu-Wang-Xin\cite{jwx2} and the Lemma \ref{BP} in the last subsection for the proof, we omit here for simplicity.
\begin{lemma}
Assume $\beta>1$, for any $p\geq 1$,
\begin{equation}
\sup\limits_{0\leq t\leq T}||\rho||_p\leq Cp^\frac{2}{\beta-1},
\end{equation}
where C is a positive constant independent of $p$.
\end{lemma}
{\bf{Case II: initial density with nonvacuum at far fields, that is, $\tilde{\rho}>0$}.}

Corresponding to Case I, we first obtain the following weighted energy estimates 
\begin{lemma}
For $a>0$ satisfying $a^2<\frac{4(\sqrt{2+\frac{\lambda(\tilde{\rho})}{\mu}}-1)}{1+\frac{\lambda(\tilde{\rho})}{\mu}}$ and $\gamma<2\beta$, it holds that for sufficiently large $m>1$ and any $t\in[0,T]$
\begin{align}
\int&|x|^a[\rho|u|^2+\Psi(\rho,\tilde{\rho})+|H|^2](t,x)dx\nonumber\\
&+\int_{0}^{t}[\||x|^\frac{a}{2}\nabla u\|{_2^2}+\||x|^\frac{a}{2}\nabla H\|{_2^2}|+\||x|^\frac{a}{2}\di u\|{_2^2}+\||x|^\frac{a}{2}\sqrt{\lambda(\rho)}\di u\|{_2^2}](s)ds\nonumber\\
&\leq C_a[1+\int_{0}^{t}(\|\rho-\tilde{\rho}\|{_{2m\beta+1}^\beta}(s)+1)(\|\nabla u\|{_2^2}(s)+1)ds]
\end{align}
where the positive constant $C_a$ may depend on $a$ but is independent of $m$.
\end{lemma}
\begin{proof}
Multiplying the equation $(\ref{MHD})_2$ by $u$ and equation $(\ref{MHD})_3$ by $H$, summing the resulting equations and using the continuity equation $(\ref{MHD})_1$, we have 
\begin{align}
&\frac{d}{dt}(\rho\frac{|u|^2}{2}+\frac{|H|^2}{2}+\Psi(\rho,\tilde{\rho}))+\di(\rho u\frac{|u|^2}{2}+\Psi(\rho,\tilde{\rho})u+(P(\rho)-P(\tilde{\rho}))u)\nonumber\\
&=\di((u\times H)\times H)+\di[\mu\nabla(\frac{|u|^2}{2})+\nu\nabla(\frac{|H|^2}{2})+(\mu+\lambda(\rho)\di uu)]\nonumber\\
&~~-\mu|\nabla u|^2-\nu|\nabla H|^2-(\mu+\lambda(\rho)(\di u)^2)\nonumber,
\end{align}
Multiplying the above identity by $|x|^a$ and integrating over $\mathbb{R}^2$,it holds that
\begin{align}\label{3.3.21}
&\frac{d}{dt}\int|x|^a[\rho\frac{|u|^2}{2}+\frac{|H|^2}{2}+\Psi(\rho,\tilde{\rho})]dx\nonumber\\
&+\int[\mu\||x|^\frac{a}{2}\nabla u\|{_2^2}+\nu\||x|^\frac{a}{2}\nabla H\|{_2^2}+\mu\||x|^\frac{a}{2}\di u\|{_2^2}+\||x|^\frac{a}{2}\sqrt{\lambda(\rho)}\di u\|{_2^2}](s)dx\nonumber\\
&=\int[\rho u\frac{|u|^2}{2}+\Psi(\rho,\tilde{\rho})u+(P(\rho)-P(\tilde{\rho}))u]\cdot\nabla(|x|^a)dx\\
&-\int[\mu\nabla(\frac{|u|^2}{2})+(\mu+\lambda(\rho)\di uu)]\cdot\nabla(|x|^a)dx\nonumber\\
&-\int\nu\nabla(\frac{|H|^2}{2})\cdot\nabla(|x|^a)dx-\int((u\times H)\times H)\cdot\nabla(|x|^a)dx=:\sum_{i=1}^{4}I_i\nonumber,
\end{align}
Now we estimate the terms on the RHS of $(\ref{3.3.21})$. As to the first two terms, which had been estimated in Lemma 3.2 of Jiu-Wang-Xin\cite{jwx3} for compressible Navier-Stokes, it holds that
\begin{equation}\label{3.3.22}
|I_1|\leq \sigma\||x|^\frac{a}{2}\nabla u\|{_2^2}+C(1+\||x|^\frac{a}{2}\sqrt{\rho}u\|{_2^2}+\|\rho-\tilde{\rho}\|{_{2m\beta}^\beta})(\|\nabla u\|{_2^2}+1)
\end{equation}
and
\begin{align}\label{3.3.23}
|I_2|&\leq \sigma[\||x|^\frac{a}{2}\sqrt{\lambda(\rho)}\di u\|{_2^2}+\||x|^\frac{a}{2}\nabla u\|{_2^2}]\nonumber\\
&+\frac{\mu a^2}{2}\||x|^\frac{a}{2}\nabla u\|{_2^2}+\frac{\mu a^2}{2}\||x|^\frac{a}{2}\di u\|_2\||x|^\frac{a}{2}\nabla u\|_2\\
&+\frac{a^2\sqrt{\lambda(\tilde{\rho})}}{2}\||x|^\frac{a}{2}\sqrt{\lambda(\rho)}\di u\|_2\||x|^\frac{a}{2}\nabla u\|_2+C(1+\|\rho-\tilde{\rho}\|{_{2m\beta}^\beta})(\|\nabla u\|{_2^2}+1)\nonumber
\end{align}
Next, the last two terms can be estimate as follows,
\begin{equation}\label{3.3.24}
|I_3|\leq\nu a\int|H||\nabla H||x|^{a-1}dx\leq\nu a\||x|^{\frac{a}{2}-1}H\|_2\||x|^\frac{a}{2}\nabla H\|_2\leq\frac{\nu a^2}{2}\||x|^\frac{a}{2}\nabla H\|{_2^2}
\end{equation}
where we have used the Caffarelli-Kohn-Nirenberg inequality with best constant.
\begin{align}\label{3.3.25}
|I_4|&\leq 2\int|u||H|^2|x|^{a-1}dx\leq 2\||x|^{a-1}u\|_{q_2}\|H^2\|_{p_2}\leq C\|\nabla u\|{_2^{\theta_2}}\||x|^{\kappa_2}u\|{_{r_2}^{1-\theta_2}}\nonumber\\
&\leq C\|\nabla u\|{_2^{\theta_2}}\||x|^\frac{a}{2}\nabla u\|{_2^{1-\theta_2}}\leq\sigma\||x|^\frac{a}{2}\nabla u\|{_2^2}+C(\|\nabla u\|{_2^2}+1)
\end{align}
where we have used the Caffarelli-Kohn-Nirenberg inequality and H\"{o}lder inequality so that the indexes $p_2>1,q_2>1,r_2>0,\theta_2\in(0,1)$, determined by the indexes relationship in integral inequalities, satisfy
\begin{equation}
\frac{1}{p_2}+\frac{1}{q_2}=1;~~~~\frac{1}{q_2}+\frac{a-1}{2}=(1-\theta_2)(\frac{1}{r_2}+\frac{\kappa_2}{2});~~~~\frac{1}{r_2}+\frac{\kappa_2}{2}=\frac{1}{2}+\frac{\frac{a}{2}-1}{2}.\nonumber
\end{equation}
It follows that
\begin{equation}
p_2=\frac{4}{a(1+\theta_2)+2}>1\nonumber
\end{equation}
by choosing $0<a<2, \theta_2\in(0,1).$ Substituting $(\ref{3.3.22})$-$(\ref{3.3.25})$ into $(\ref{3.3.21})$,one has
\begin{align}\label{3.3.26}
\frac{d}{dt}&\int|x|^a[\rho\frac{|u|^2}{2}+\frac{|H|^2}{2}+\Psi(\rho,\tilde{\rho})]dx+J(t)\nonumber\\
&\leq\sigma(\||x|^\frac{a}{2}\nabla u\|{_2^2}+\||x|^\frac{a}{2}\sqrt{\lambda(\rho)}\di u\|{_2^2})\\
&~~~~+C(1+\||x|^\frac{a}{2}\sqrt{\rho}u\|{_2^2}+\|\rho-\tilde{\rho}\|{_{2m\beta}^\beta})(\|\nabla u\|{_2^2}+1)\nonumber
\end{align}
where $J(t)$ is defined as follows
\begin{align}\label{3.3.27}
&J(t):=\mu(1-\frac{a^2}{2})\||x|^\frac{a}{2}\nabla u\|{_2^2}(t)-\frac{\mu a^2}{2}\||x|^\frac{a}{2}\nabla u\|_2\||x|^\frac{a}{2}\di u\|_2(t)+\mu\||x|^\frac{a}{2}\di u\|{_2^2}(t)\nonumber\\
&~~~~+\||x|^\frac{a}{2}\sqrt{\lambda(\rho)}\di u\|{_2^2}-\frac{a^2\sqrt{\lambda(\tilde{\rho}))}}{2}\||x|^\frac{a}{2}\sqrt{\lambda(\rho}\di u\|_2+\nu(1-\frac{a^2}{2})\||x|^\frac{a}{2}\nabla H\|{_2^2}
\end{align}
Similar to the proof in \cite{jwx3}, if the weight $a$ satisfies
\begin{equation}
 0<a^2<\frac{4(\sqrt{2+\frac{\lambda(\tilde{\rho})}{\mu}}-1)}{1+\frac{\lambda(\tilde{\rho})}{\mu}}\nonumber
 \end{equation}
then there exists a positive constant $C_a$ such that
\begin{equation}\label{3.3.28}
J(t)\geq C_a(|||x|^\frac{a}{2}\nabla u||{_2^2}+|||x|^\frac{a}{2}\di u||{_2^2}+|||x|^\frac{a}{2}\sqrt{\lambda(\rho)}\di u||{_2^2}+|||x|^\frac{a}{2}\nabla H||{_2^2})
\end{equation}
It follows from $(\ref{3.3.26})$-$(\ref{3.3.28})$ that
\begin{align}
&\frac{d}{dt}\int|x|^a[\rho\frac{|u|^2}{2}+\frac{|H|^2}{2}+\Psi(\rho,\tilde{\rho})]dx\nonumber\\
&~~~~+C_a(\||x|^\frac{a}{2}\nabla u\|{_2^2}+\||x|^\frac{a}{2}\di u\|{_2^2}+\||x|^\frac{a}{2}\sqrt{\lambda(\rho)}\di u\|{_2^2}+\||x|^\frac{a}{2}\nabla H\|{_2^2})\nonumber\\
&\leq\sigma(\||x|^\frac{a}{2}\nabla u\|{_2^2}+C(1+\||x|^\frac{a}{2}\sqrt{\rho}u\|{_2^2}+\|\rho-\tilde{\rho}\|{_{2m\beta}^\beta})(\|\nabla u\|{_2^2}+1)
\end{align}
Therefore, the result follows from Gronwall's inequality which complete the proof of the lemma. 
\end{proof}

Applying the operator $\di$ to the momentum equation $(\ref{MHD})_2$,we have
\begin{equation}\label{3.3.30}
[\di(\rho u)]_t+\di[\di(\rho u\otimes u-H\otimes H)]=\Delta F.
\end{equation}
where the effective flux $F$ is given by
\begin{equation}\label{3.3.31}
F:=(2\mu+\lambda(\rho))\di u-(P(\rho)-P(\tilde{\rho}))-\frac{|H|^2}{2}
\end{equation}
Consider the following elliptic problems in $\mathbb{R}^2$
\begin{equation}\label{3.3.32}
-\Delta\xi_1=\di((\sqrt{\rho}-\sqrt{\tilde{\rho}})\sqrt{\rho}u);~~~~~~\xi_1 \longrightarrow 0~~as~~ |x| \longrightarrow +\infty
\end{equation}
\begin{equation}\label{3.3.33}
-\Delta\xi_2=\sqrt{\tilde{\rho}}\di(\sqrt{\rho}u);~~~~~~\xi_2 \longrightarrow 0~~as~~ |x| \longrightarrow +\infty
\end{equation}
\begin{equation}\label{3.3.34}
-\Delta\eta=\di(\di(\rho u\otimes u-H\otimes H));~~~~~~\eta \longrightarrow 0~~as~~ |x| \longrightarrow \infty
\end{equation}
Similar to the proof of Lemma 3.3-3.5 in \cite{jwx3} and the corresponding lemmas for periodic problem in the lase section, we can obtain the following lemmas,
\begin{lemma}
\begin{flalign*}
\begin{split}
&\quad(1)~\|\nabla\xi_1\|_{2m}\leq Cm\|\rho-\tilde{\rho}\|_{\frac{2mk}{k-1}}\|u\|_{2mk},~ for~ any~ k>1,m\geq 1;\\
&\quad(2)~\|\nabla\xi_1\|_{2-r}\leq C\|\sqrt{\rho}-\sqrt{\tilde{\rho}}\|{_\frac{2(2-r)}{r}^\frac{1}{2}},~ for~ any~ 0<r<1;\\
&\quad(3)~\|\nabla\xi_2\|_{2m}\leq Cm[\|\rho-\tilde{\rho}\|_{\frac{2mk}{k-1}}\|u\|_{2mk}+\sqrt{\tilde{\rho}}\|u\|_{2m}], ~for~ any~ k>1,m\geq 1;\\
&\quad(4)~\||x|^\frac{a}{2}\nabla\xi_2\|_2\leq C\||x|^\frac{a}{2}\sqrt{\rho}u\|_2,~ for ~a ~satisfying~ 0<a^2<\frac{4(\sqrt{2+\frac{\lambda(\tilde{\rho})}{\mu}}-1)}{1+\frac{\lambda(\tilde{\rho})}{\mu}}\nonumber;\\
&\quad(5)~\|\eta||\leq Cm(\|\rho-\tilde{\rho}\|_\frac{2mk}{k-1}\|u\|_{2mk}+\tilde{\rho}\|u\|{_{4m}^2}+\|H\|{_{4m}^2}), ~for~ any~ k>1,m\geq 1,
\end{split}&
\end{flalign*}
where $C$ are positive constants independent of $m,k$ and $r$.
\end{lemma}
\begin{lemma}
\begin{flalign*}
\begin{split}
&\quad(1)~\|\xi_1\|_{2m}\leq Cm^\frac{1}{2}\|\nabla\xi\|_\frac{2m}{m+1}\leq Cm^\frac{1}{2}\|\rho-\tilde{\rho}\|_{2m}, ~for~ any~ m\geq 2;\\
&\quad(2)~\|\xi_2\|_{2m}\leq Cm^\frac{1}{2}\||x|^\frac{a}{2}\sqrt{\rho}u\|{_2^\frac{2}{ma}},~for~ any~m+1>\frac{4}{a};\\
&\quad(3)~\|u\|_{2m}\leq Cm^\frac{1}{2}[\|\nabla u\|_2+1], ~for~ any~ m\geq 1~;\\
&\quad(4)~\|\nabla\xi_1\|_{2m}\leq Cm^\frac{3}{2}k^\frac{1}{2}\|\rho-\tilde{\rho}\|_{\frac{2mk}{k-1}}(\|\nabla u\|_2+1),~ for~ any m\geq 1,k>1;\\
&\quad(5)~\|\nabla\xi_2\|_{2m}\leq Cm^\frac{3}{2}(k^\frac{1}{2}\|\rho-\tilde{\rho}\|_{\frac{2mk}{k-1}}+1)(\|\nabla u\|_2+1), ~for~ any~ k>1,m\geq 1;\\
&\quad(6)~\|\eta\|_{2m}\leq Cm^2(k\|\rho-\tilde{\rho}\|_{\frac{2mk}{k-1}}+1)(\|\nabla u\|_2+1),~for~ any~ m\geq 1,k>1;
\end{split}&
\end{flalign*}
where $C$ are positive constants independent of $m,k$.
\end{lemma}

\begin{lemma}\label{lem4.10}
Assume $\beta>1$, for any $p\geq 2$,
\begin{equation}
\sup\limits_{0\leq t\leq T}\|\rho-\tilde{\rho}\|_p\leq Cp^\frac{2}{\beta-1},
\end{equation}
where C is a positive constant independent of $p$.
\end{lemma}

The remaining a priori estimates to give the upper bound of the density $\rho$ either $\tilde{\rho}=0$ or $\tilde{\rho}>0$ are all similar to relevant ones(Lemma \ref{lem3.6}-\ref{lem3.8}) for the periodic problems proved in the lase section. We only need to instead the corresponding estimates for $\|u\|_{2m}$ and $\rho$ by Lemma \ref{EllR2}(2) for $\tilde{\rho}=0$ and Lemma \ref{lem4.10} for $\tilde{\rho}>0$ respectively. So, we obtain that
\begin{lemma}
There exist a positive constant $C$ such that
\begin{equation}\label{3.3.36}
0\leq\rho(t,x)\leq C,~~~~\forall(t,x)\in[0,T]\times\mathbb{R}^2.
\end{equation}
\end{lemma} 

\section{Proof of main results}

\subsection{Proof of Theorem \ref{thm1}}
We first construct the approximation of the initial data as follows.\\
Define $\rho{_0^\delta}=\rho_0+\delta, P{_0^\delta}=P(\rho_0)+\delta$ for any small positive constant $\delta>0$. To approximate the initial velocity, we define $u{_0^\delta}$ to be the unique solution to the following elliptic problem
\begin{equation}\label{2.5.1}
\mathcal{L}_{\rho{_0^\delta}}u{_0^\delta}:=\mu\Delta u{_0^\delta}+\nabla((\mu+\lambda(\rho{_0^\delta}))\di u{_0^\delta})=\nabla P{_0^\delta}-(\nabla\times H_0)\times H_0+\sqrt{\rho_0}g
\end{equation}
with the periodic boundary condition on $\mathbb T^2$ and $\int_{\mathbb T^2}u{_0^\delta}dx=\int_{\mathbb T^2}u_0dx=:\bar{u}_0$.
It should be noted that $u{_0^\delta}$ is uniquely determined due to the compatibility condition.\\
It follows from (\ref{2.5.1}) that
\begin{align}
\mathcal{L}_{\rho_0}u{_0^\delta}:&=\mu\Delta u{_0^\delta}+\nabla((\mu+\lambda(\rho_0))\di u{_0^\delta})\nonumber\\
&=-\nabla[(\lambda(\rho{_0^\delta})-\lambda(\rho_0))\di u{_0^\delta}]+\nabla P{_0^\delta}-(\nabla\times H_0)\times H_0+\sqrt{\rho_0}g.
\end{align}
By the elliptic regularity, it holds that
\begin{align}\label{2.5.3}
||u{_0^\delta}-\bar{u}_0||_{H^2}&\leq C(||\lambda(\rho{_0^\delta})-\lambda(\rho_0)||_\infty||\nabla(\di u{_0^\delta})||_2+||\nabla(\lambda(\rho{_0^\delta})-\lambda(\rho_0))||_\infty||\di u{_0^\delta}||_2\nonumber\\
&~~~~~~+||\nabla P{_0^\delta}||_2+||(\nabla\times H_0)\times H_0||_2+||\sqrt{\rho_0}g||_2)\nonumber\\
&\leq C(\delta||u{_0^\delta}||_{H^2}+||P_0||_{H^1}+||H_0||_{H^2}+||\sqrt{\rho_0}||_\infty||g||_2)\nonumber\\
&\leq C(\delta||u{_0^\delta}||_{H^2}+1), 
\end{align}
where the generic positive constant $C$ is independent of $\delta>0$.\\
Therefore, if $\delta<<1$, then (\ref{2.5.3}) yields that
\begin{equation}\label{2.5.4}
||u{_0^\delta}||_{H^2}\leq C,
\end{equation}
where the generic positive constant $C$ is independent of $\delta>0$.\\
Replacing the initial data $(\rho_0,u_0,H_0)$ by $(\rho{_0^\delta},u{_0^\delta},H_0)$, the Theorem \ref{thm1} and the a priori estimates obtained in section 3 yields that the problem (\ref{MHD})-(\ref{IVC1}) has a unique classical solution $(\rho^\delta,u^\delta,H^\delta)$ satisfying the regularity (\ref{reg1}).
Next, we show that this gives the unique classical solution to original problem by taking the limit $\delta\rightarrow 0$.\\
Due to the compatibility condition (\ref{com}) and (\ref{2.5.1}), it holds that
\begin{equation*}
\mathcal{L}_{\rho_0}(u{_0^\delta}-u_0)=\nabla[(\lambda(\rho{_0^\delta})-\lambda(\rho_0))\di u{_0^\delta}].
\end{equation*}
Therefore, by the elliptic regularity, it follows from (\ref{2.5.4}) that
\begin{align}
||u{_0^\delta}-u_0||_{H^2}&\leq C(||\lambda(\rho{_0^\delta})-\lambda(\rho_0)||_\infty||\nabla^2u{_0^\delta}||_2+||\nabla(\lambda(\rho{_0^\delta})-\lambda(\rho_0))||_\infty||\di u{_0^\delta}||_2)\nonumber\\
&\leq C\delta\rightarrow 0,~~~~as~~\delta\rightarrow 0.
\end{align}
Clearly, $||\rho{_0^\delta}-\rho_0||_{W^{2,q}}\rightarrow 0$, as $\delta\rightarrow 0$. Thus, since the uniform-in-$\delta$ bounds we have, the approximation solution $(\rho^\delta,u^\delta,H^\delta)$ converge to the solution $(\rho,u,H)$ with the regularity (\ref{reg1}) to the problem (\ref{MHD})-(\ref{VKC}) with the initial data (\ref{IVC1}).\\
Finally, the regularity (\ref{reg1}) of the solution implies that it is a classical solution.
Since $(u,H)\in L^2(0,T;H^3(\mathbb T^2))$ and $u_t\in L^2(0,T;H^1(\mathbb T^2))$, so the Sobolev embedding theorem implies that
\begin{equation*}
u\in C([0,T];H^2(\mathbb T^2))\hookrightarrow C([0,T]\times\mathbb T^2).
\end{equation*} 
It follows from $(\rho,P)\in L^\infty(0,T;W^{2,q}(\mathbb T^2))$ and $(\rho_t,P_t)\in L^\infty(0,T;H^1(\mathbb T^2))$ that $(\rho,P)\in C([0,T];W^{1,q}(\mathbb T^2))\cap C([0,T];W^{2,q}(\mathbb T^2)-weak)$. This and (\ref{2.4.67}) then imply that
\begin{equation*}
(\rho,P(\rho))\in C([0,T];W^{2,q}(\mathbb T^2)).
\end{equation*}
Since for any $\tau\in(0,T)$,
\begin{equation*}
(\nabla(u,H),\nabla^2(u,H))\in L^\infty(\tau,T;W^{1,q}(\mathbb T^2)),~~~~(\nabla(u_t,H_t),\nabla^2(u_t,H_t)\in L^\infty(\tau,T;L^2(\mathbb T^2))).
\end{equation*}
Therefore, the Aubin-Lions lemma gives that 
\begin{equation*}
(\nabla(u,H),\nabla^2(u,H))\in C([\tau,T]\times\mathbb T^2).
\end{equation*}
Due to the fact that
\begin{equation*}
\nabla(\rho,P)\in C([0,T];W^{1,q}(\mathbb T^2))\hookrightarrow C([0,T]\times\mathbb T^2)
\end{equation*}
and the continuity equation $(\ref{MHD})_1$, it holds that
\begin{equation*}
\rho_t\in C([\tau,T]\times\mathbb T^2).
\end{equation*}
It follows from the momentum equation $(\ref{MHD})_2$ that
\begin{equation*}
(\rho u)_t\in C([\tau,T]\times\mathbb T^2).
\end{equation*}
Thus we completed the proof of Theorem \ref{thm1}.

\subsection{Proof of Theorem \ref{thm2}}
Now we give the approximation scheme of the initial data to obtain the global classical solution of the Cauchy problem on $\mathbb{R}^2$ permitting the appearance of the vacuum. We construct the approximation of the initial data as follows, proposed in \cite{hl2,jwx2} for compressible Navier-Stokes equations.\\
Define $\rho{_0^\delta}=\rho_0+\delta e^{-|x|^2}, P{_0^\delta}=P(\rho_0)+\delta e^{-|x|^2}$ for any small positive constant $\delta>0$. To approximate the initial velocity, we define $u{_0^\delta}$ as
\begin{equation}
u{_0^\delta}=\left\{\begin{array}{lr}
\tilde{u}{_0^\delta},~|x|\leq M+1, \\
u_0,~|x|\geq M+1,
\end{array}\right.
\end{equation}
where $\tilde{u}{_0^\delta}$ to be the unique solution to the following elliptic problem in $\Omega_M:=\{x:|x|<M+1\}$
\begin{equation}\label{3.4.1}
L_{\rho{_0^\delta}}\tilde{u}{_0^\delta}:=\mu\Delta \tilde{u}{_0^\delta}+\nabla((\mu+\lambda(\rho{_0^\delta}))\di \tilde{u}{_0^\delta})=\nabla P{_0^\delta}-(\nabla\times H_0)\times H_0+\sqrt{\rho_0}g
\end{equation}
with boundary value $\tilde{u}{_0^\delta}=u_0$ at $|x|=M+1$.
It should be noted that $u{_0^\delta}$ is uniquely determined due to the compatibility condition.\\
It follows from (\ref{3.4.1}) that
\begin{align}
L_{\rho_0}\tilde{u}{_0^\delta}:&=\mu\Delta \tilde{u}{_0^\delta}+\nabla((\mu+\lambda(\rho_0))\di \tilde{u}{_0^\delta})\nonumber\\
&=-\nabla[(\lambda(\rho{_0^\delta})-\lambda(\rho_0))\di \tilde{u}{_0^\delta}]+\nabla P{_0^\delta}-(\nabla\times H_0)\times H_0+\sqrt{\rho_0}g
\end{align}
By the elliptic regularity, it holds that
\begin{align}\label{3.4.3}
||\tilde{u}{_0^\delta}||_{H^2(\Omega_M)}&\leq C(||\lambda(\rho{_0^\delta})-\lambda(\rho_0)||_\infty||\nabla(\di \tilde{u}{_0^\delta})||_2+||\nabla(\lambda(\rho{_0^\delta})-\lambda(\rho_0))||_\infty||\di \tilde{u}{_0^\delta}||_2\nonumber\\
&~~~~~~+||\nabla P{_0^\delta}||_2+||(\nabla\times H_0)\times H_0||_2+||\sqrt{\rho_0}g||_2)\nonumber\\
&\leq C(\delta||\tilde{u}{_0^\delta}||_{H^2}+||P_0||_{H^1}+||H_0||_{H^2(\Omega_M)}+||\sqrt{\rho_0}||_\infty||g||_2)\nonumber\\
&\leq C(\delta||\tilde{u}{_0^\delta}||_{H^2(\Omega_M)}+1) 
\end{align}
where the generic positive constant $C$ is independent of $\delta>0$.\\ 
Therefore, if $\delta<<1$, then (\ref{3.4.3}) yields that
\begin{equation}\label{3.4.4}
||\tilde{u}{_0^\delta}||_{H^2(\Omega_M)}\leq C
\end{equation}
where the generic positive constant $C$ is independent of $\delta>0$.\\ 
Due to the compatibility condition (\ref{3.1.6}) and (\ref{3.4.1}), it holds that
\begin{equation*}
L_{\rho_0}(\tilde{u}{_0^\delta}-u_0)=\nabla[(\lambda(\rho{_0^\delta})-\lambda(\rho_0))\di u{_0^\delta}]~~ in~~\Omega_M
\end{equation*}
with $(\tilde{u}{_0^\delta}-u_0)=0$ at $|x|=M+1$.
Therefore, by the elliptic regularity, it follows from (\ref{3.4.4}) that
\begin{align}
||\tilde{u}{_0^\delta}-u_0||_{H^2(\Omega_M)}&\leq C(||\lambda(\rho{_0^\delta})-\lambda(\rho_0)||_\infty||\nabla^2u{_0^\delta}||_2+||\nabla(\lambda(\rho{_0^\delta})-\lambda(\rho_0))||_\infty||\di u{_0^\delta}||_2)\nonumber\\
&\leq C\delta\rightarrow 0,~~~~as~~\delta\rightarrow 0.
\end{align}
which follows that 
\begin{equation*}
u{_0^\delta}\rightarrow u_0, ~in~ H^2(\mathbb R^2),
\end{equation*} 
and
\begin{equation*}
\sqrt{\rho{_0^\delta}}u{_0^\delta}(1+|x|^\frac{\alpha}{2})\rightarrow \sqrt{\rho_0}u_0)(1+|x|^\frac{\alpha}{2}),~ in ~L^2(\mathbb R^2),
\end{equation*} 
Clearly,
\begin{equation*}
||\rho{_0^\delta}-\rho_0||_{W^{2,q}}\rightarrow 0,~~  ||(\rho{_0^\delta}-\rho_0)(1+|x|^{\alpha_1})||_1,~~ and ~~||(\nabla u{_0^\delta}-\nabla u_0)|x|^\frac{\alpha}{2}||_2
\end{equation*}  
as $\delta\rightarrow 0$.
Thus, since the uniform-in-$\delta$ bounds we have, the approximation solution $(\rho^\delta,u^\delta,H^\delta)$ converge to the solution $(\rho,u,H)$ with the regularity (\ref{reg2}) to the problem (\ref{MHD})-(\ref{VKC}) with the initial data (\ref{IVC2}).\\
Finally, the regularity (\ref{reg2}) of the solution implies that it is a classical solution by the Sobolev embedding inequalities and Aubin-Lions Lemma(refer to the last subsection or Section 5 in \cite{jwx2}).

\subsection{Proof of Theorem \ref{thm3}}

The approximation of the initial data have been constructed in \cite{jwx3} for compressible Navier-Stokes equations, we rewrite these for MHD equation here for the completeness. Since $\lim\limits_{|x|\rightarrow +\infty}\rho_0=\tilde{\rho}>0$, there exists a large number $M>0$ such that if $|x|>M$, $\rho_0(x)\geq \frac{\tilde{\rho}}{2}$. Then, for any $0<\delta<\frac{\tilde{\rho}}{2}$, we define 
\begin{equation}
\rho{_0^\delta}(x)=\left\{\begin{array}{lr} \rho_0(x)+\delta,~if~|x|\leq M, \\
\rho_0(x)+\delta s(x),~if~M\leq|x|\leq M+1, \\
\rho_0(x),~if~|x|\geq M+1,
\end{array}\right.
\end{equation}
where $s(x)=s(|x|)$ is a smooth and describing function satisfying $s(x)\equiv 1$ if $|x|\leq M$ and $s(x)=0$ if $|x|\geq M+1$. Similarly, one can construct the approximation of the initial pressure $P{_0^\delta}$. Then the approximate of the velocity $u{_0^\delta}$ can be the same to the above one shown in proof of Theorem \ref{thm1}. We omit it here. Hence, the uniform estimates show that the approximation solution $(\rho^\delta,u^\delta,H^\delta)$ converge to the solution $(\rho,u,H)$ with the regularity (\ref{reg3}) to the problem (\ref{MHD})-(\ref{VKC}) with the initial data (\ref{IVC2}).\\
Finally, the regularity (\ref{reg3}) of the solution implies that it is a classical solution by the Sobolev embedding inequalities and Aubin-Lions Lemma(refer to the last subsection or Section 5 in \cite{jwx2}).

{\bf Acknowledgement.} This is the main part of the author's M.Phil thesis written under the supervision of Professor Zhouping Xin at the Institute of Mathematical Sciences at the Chinese University of Hong Kong. The author would like to express great gratitude to Prof. Xin for his support and guidance. The author also would like to thank Dr. Jinkai Li for his valuable discussion and advice.

\bigskip

{

}

\end{document}